\documentclass[a4paper,12pt]{article}
\usepackage[utf8]{inputenc}
  \usepackage[usenames,dvipsnames]{xcolor}
\usepackage{amsmath,mathrsfs}
\usepackage{amsfonts}
  \usepackage[all]{xy}
  \usepackage[english]{babel}
  \usepackage{graphicx}
  \usepackage{enumerate}
  \usepackage{wrapfig}
\usepackage{amsthm}
\usepackage{authblk}

\usepackage{hyperref}

  \usepackage[usenames,dvipsnames]{xcolor}
  \usepackage{mathbbol}
\usepackage{pgf,tikz}
\usetikzlibrary{calc}

  \def\ZZ{\mathbb{Z}}
  \def\NN{\mathbb N}
  \def\r{\mathfrak{r}}
  \def\D{\mathcal{D}}
  
  \def\H{\mathbb{H}}
  \def\M{\mathcal{M}}

  \def\chleq{\leq\!\!\!\leq }
   \def \chll{<\!\!\!<}

  \def\U{\mathcal U}

  \def\GGL{\SO_0(1,2)}
  \def\isom{Isom}
  \def\lsem{[\![}
  \def\rsem{]\!]}
\def\V{\mathcal V}

\def\C{\mathcal C}
\def\A{\mathcal A}
\def\TT{\mathcal T}
\def\R{\mathbb{R}}
\def\E{\mathbf E}

\def\PSigmaS{\P}

\def\grad{\overrightarrow{\mathbf{grad}}}
\def\d{\mathrm{d}}

\def\D{\mathcal{D}}
\def\SO{\mathrm{SO}}
\def\E{\mathbb{E}}

\def\sing{\mathrm{Sing}}
\def\RR{\mathbb{R}}
\newcommand{\mass}[1]{{\mathbb E}^{1,2}_{#1}}
\def\isom{\mathrm{Isom}}

\def\susp{\mathrm{susp}}

\def\reg{\Reg}

\def\V{\mathcal V}

\def\W{\mathcal W} 

\def\SO{\mathrm{SO}}

\def\R{\mathcal R}

\def\SS{\mathbb{S}}

\def\quadform{Q}

\def\H{\mathbb{H}}
\def\C{\mathcal C}
\def\E{\mathbb{E}}
\def\sing{\mathrm{Sing}}
\def\isom{\mathrm{Isom}}
\def\d{\mathrm{d}}

\def\reg{\mathrm{Reg}}
\def\D{\mathcal{D}}

\def\U{\mathcal{U}}

\def\grad{\mathrm{grad}}

\def\M{\mathcal M}

\def\r{\mathfrak{r}}

\def\sing{\mathrm{Sing}}

\def\susp{\mathrm{susp}}

\def\PP{\mathbb{P}}
\def\P{\mathcal{P}}

\def\CellPSigmabar{\overline{\P}_\TT}

\def\A{\mathcal A}

\def\CellPSigma{\P_\TT}
\def\GG{\isom_0(\mass{})}

\def\calP{\mathcal{P}}
\newcommand{\fonction}[5]{\displaystyle#1:\begin{array}{l|rcl}
& \displaystyle #2 & \longrightarrow & \displaystyle #3 \\
	& \displaystyle #4 & \longmapsto & \displaystyle #5 \end{array}}
\newcommand{\fonctiondeux}[7]{\displaystyle#1:\begin{array}{l|rcl}
& \displaystyle #2 & \longrightarrow & \displaystyle #3 \\
	& \displaystyle #4 & \longmapsto & \displaystyle #5 \\ 
	& \displaystyle #6 & \longmapsto & \displaystyle #7\end{array}}	
	
\newcommand{\fonctionn}[4]{\displaystyle\begin{array}{l|rcl}
& \displaystyle #1 & \longrightarrow & \displaystyle #2 \\
	& \displaystyle #3 & \longmapsto & \displaystyle #4 \end{array}}

\definecolor{MyGreen}{rgb}{0.0,.5,0.0}

\definecolor{MyDarkRed}{rgb}{0.7,0,0}

\usepackage{float}

\title{Alexandrov Theorem for 2+1 flat radiant spacetimes}
\author{L\'eo Brunswic}
\affil{Univ Lyon, Ens de Lyon, Univ Lyon1, CNRS, Centre de Recherche Astrophysique de
Lyon UMR5574, F–69007, Lyon, France}

\newtheorem{theo}{Theorem}	

\newtheorem{lem}{Lemma}[section]		  
\newtheorem{prop}[lem]{Proposition}
\newtheorem{cor}[lem]{Corollary}

\newtheorem*{theon}{Theorem}
\newtheorem*{problem}{Problem}
\newenvironment{claim}[1]{\par\noindent\underline{Claim:}\space#1}{}

\theoremstyle{definition}
\newtheorem{defi}[lem]{Definition}	
\newtheorem*{rem}{Remark}			 
\begin{document}

\maketitle
\begin{abstract}	
A classical Theorem of Alexandrov states that the map associating its boundary to a convex polyhdedron of the 3-dimensional Euclidean space is a bijection from the set of convex polyhdedron up to congruence to the set of isometry classes of locally Euclidean metric on the 2-sphere with conical singularities smaller that $2\pi$. Fillastre  proved a similar statement for locally Euclidean metric on higher genus surfaces with conical singularities bigger than $2\pi$ by embedding their universal covering in 3-dimensional Minkowski space as the boundary of Fuchsian polyhedra.
The original proofs of Alexandrov and Fillastre both rely on invariance of domain Theorem hence are not effective. Volkov, in his thesis, provided a variational, hence effective, proof of Alexandrov Theorem which has then  been generalised by Bobenko, Izmestiev and Fillastre. The present work goes further by adapting Volkov's variational method to provide an effective version of Fillastre Theorem and extend Fillastre's result: we show that for any closed locally Euclidean surface $\Sigma$ with conical singularities of arbitrary angles $(\theta_i)_{i\in\lsem 1,n\rsem}$ and any choice of Lorentzian angles $(\kappa_i)_{i\in \lsem1,n\rsem}$ such that $\kappa_i<\theta_i$ and $\kappa_i\leq 2\pi$, there exists a locally Minkoswki 3-manifold $M$ of linear holonomy with conical singularities $(\kappa_i)_{i\in \lsem1,n\rsem}$ and a convex polyedron $P$ in $M$ whose boundary is isometric to $\Sigma$; furthermore such a couple $(M,P)$ is unique.

\end{abstract}
\tableofcontents

\section{Introduction}
	\begin{wrapfigure}{r}{0.5\textwidth}
			\includegraphics[width=0.45\linewidth]{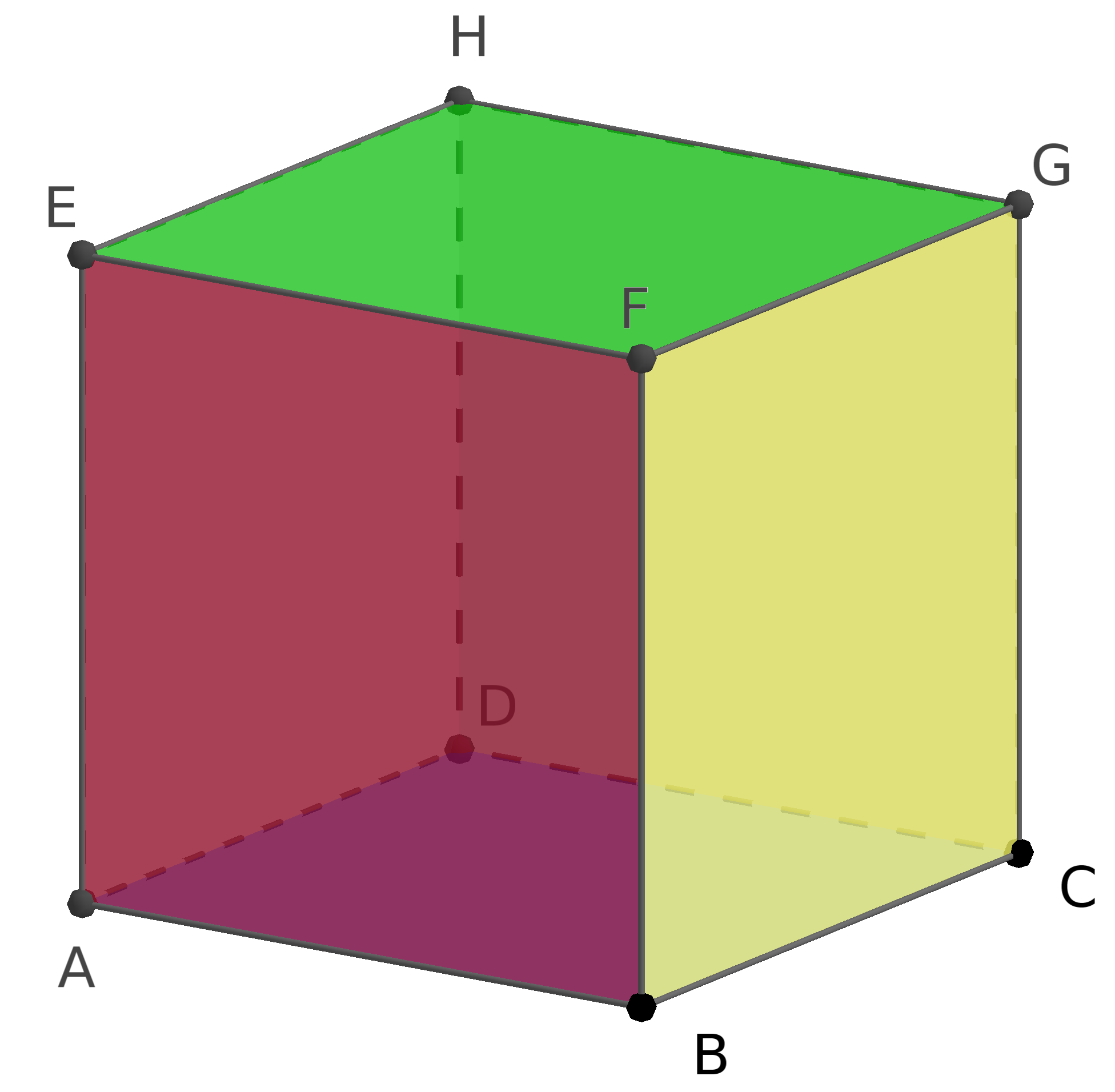}
			\includegraphics[width=0.45\linewidth]{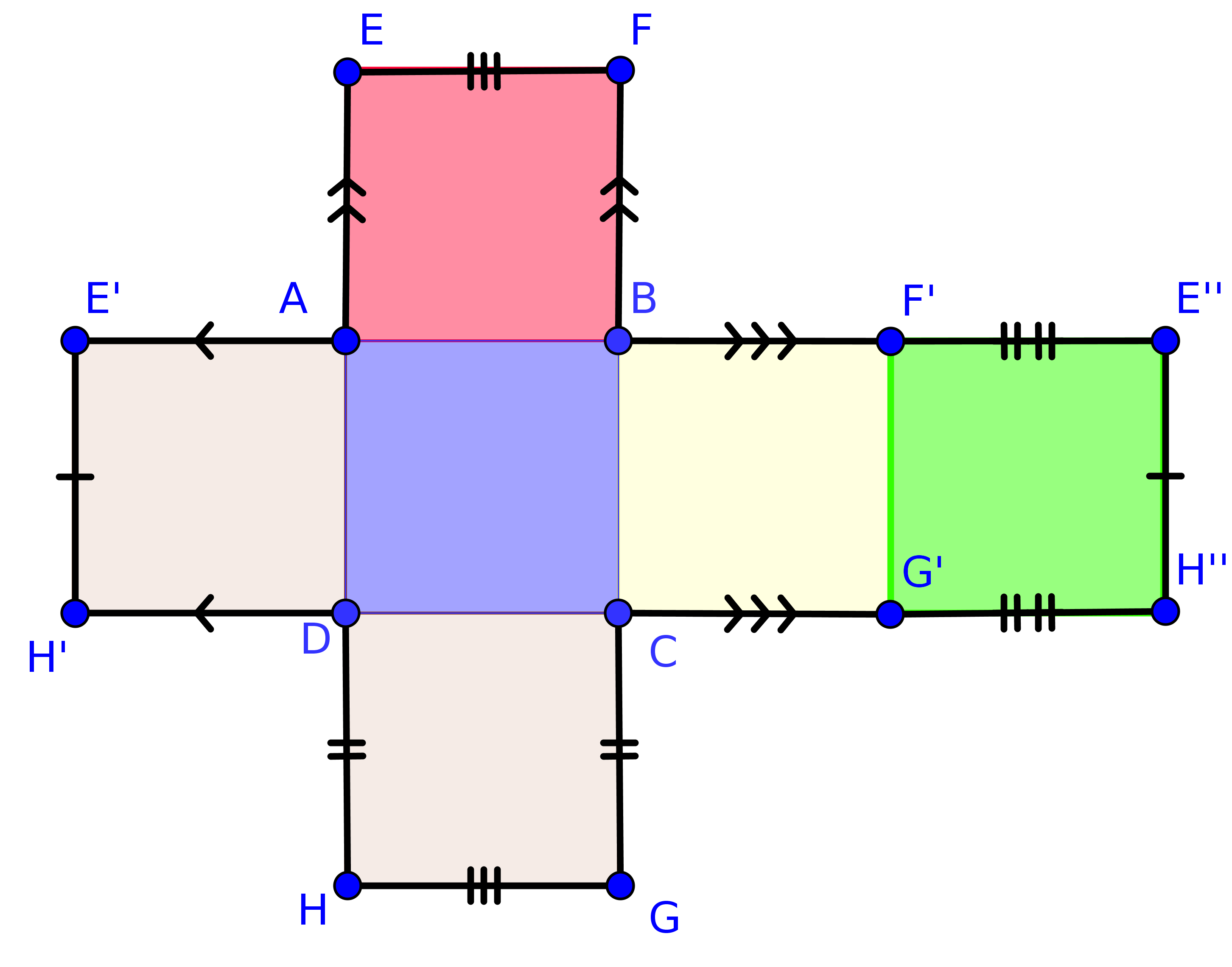}
		\end{wrapfigure}
		Let $C$ be a cube in the 3-dimensional Euclidean space $\E^3$
		and consider $\Sigma:=\partial C$
		its boundary represented on the figure. On the one hand, $\Sigma$ is a surface homeomorphic to the 2-dimensional sphere $\SS^2$ ; on the other hand, $\Sigma$ is naturally endowed with a locally euclidean metric with 6 conical singularities of angles $3\pi/2$.
		More generally, the boundary of any compact convex polyhedron in $\E^3$ is homeomorphic to the sphere and is naturally endowed with a locally Euclidean metric with conical singularities of angles lesser than $2\pi$.
		A classical theorem of Alexandrov \cite{MR13540} shows that this construction is actually bijective.
		\begin{theon}[\cite{MR13540}] Let $\Sigma$ be a locally Euclidean
			surface with conical singularities of angles lesser than $2\pi$ and homeo\-morphic to the sphere $\SS^2$,
			there exists a compact convex polyhedron $P$ in $\E^3$ such that $\partial P$ is isometric to $\Sigma$. Furthermore, two such polyhedra are congruent.
		\end{theon}

		Using a so-called deformation method, Alexandrov proved generalisations to convex polyhedron in $\H^3$ and $\SS^3$ ; this method is however not effective since it does not provide an efficient way to construct the  convex polyhedra these Theorems predict.
		In the 2000's,
		Izmestiev and Bobenko
		gave a new proof of Alexandrov theorem by a variationnal, therefore effective, method while Rivin, Hodgson, Schlenker and Fillastre
		proved generalisations to Lorentzian spaceforms (Minkoswki, de Sitter and Anti-de Sitter) in which case conical singularities of the locally Euclidean surface have angles greater than $2\pi$.
		The Alexandrov problem can then be stated in a more general context that has been recently studied in a systematic way by Fillastre and Izmestiev.
		\begin{problem}
		 Let $\Sigma$ be a closed surface of genus $g$ endowed with singular metric of constant curvature
	 $K\in \{-1,0,1\}$ and cone angles all bigger that
	 $2\pi$ (case $\varepsilon=-)$ or all lesser than $2\pi$ (case $\varepsilon =+$).
		 Denoting by $X_K^\varepsilon$ the model space of constant curvature $K$ riemannian if $\varepsilon=+$ and lorentzian if $\varepsilon=-$.

		 Is there a convex polyhedron $P$ of $X_{K}^{\varepsilon}$ which boundary  is isometric to the universal cover of $\Sigma$?
		 Furthermore, is this polyhedron essentially unique?
		\end{problem}

\begin{wrapfigure}{r}{0.4\textwidth}
		  \begin{tabular}{|c|c|c|c|c|}\hline
			$g$	&$K$	&$\varepsilon$	& DM & VM \\ \hline
			$0$	&$-1$	&$+$&\cite{MR13540}&\cite{MR2410380}  \\
			$0$	&$0$	&$+$&\cite{MR2127379}& \\
			$0$	&$1$	&$+$&\cite{MR2127379}& \\
			$0$	&$1$	&$-$&\cite{Hodgson1993}& \\ \hline
			$1$	&$-1$	&$+$&&\cite{MR2469522}\\
			$1$	&$1$	&$-$&& \cite{MR2813423}\\
			\hline
			$\geq2$&$-1$&$+$&\cite{MR2313089}& \\
			$\geq2$&$-1$&$-$&\cite{MR2794916}& \\
			$\geq2$&$0$	&$-$&\cite{MR2794916}& [B] \\
			$\geq2$&$1$	&$-$&\cite{MR2208419}& \\ \hline
		  \end{tabular}
		\end{wrapfigure}

Gauss-Bonnet formula gives a constraint on $(g,K,\varepsilon)$ ; the table on the right is base upon \cite{fillastre:hal-00535675} and sums up all possible situations
together with references to proofs by deformation (DM) and/or variational (VM) methods ; [B] refers to the present work. Proving Alexandrov-Fillastre Theorem --- case $(g,K,\varepsilon)=(\geq2,0,-)$ and  $X_K^\varepsilon$ is Minkowski space $ \E^{1,2}$ --- by a variational method is primary motivation of the present work. 
	To this end, we adapt the variational method successfully used by Bobenko, Fillastre and Izmestiev  \cite{MR2410380,MR2453328,MR2469522}; we derive Alexandrov-Fillastre, and obtain a as a result a generalization to a class of singular locally minkowski 3-manifolds : radiant spacetimes we shall describe precisely in the next section.
	 \begin{theon} Let $\Sigma$ be a closed locally Euclidean surface of genus $g$ with $s$ marked conical singularities of angles $(\theta_i)_{i\in \lsem 1,s\rsem}$.
For all $\kappa\in \prod_{i=1}^s[0,2\pi]\cap [0,\theta_i[$, there exists a radiant singular flat spacetime   $M$ homeomorphic to $\Sigma\times \RR$ with exactly $s$ singular lines of angles $\kappa_1,\cdots, \kappa_s$ and
a convex polyhedron $P\subset M$ which boundary is isometric to $\Sigma$.
	Furthermore, such a couple $(M,P)$ is unique up to equivalence.
	\end{theon}
	Equivalence in our contexte has to be understood in the following way: 
	
	Note that we accept Lorentzian conical singularities of angle $0$, the meaning of such singularities will be made clear in the following sections along with a general description of radiant spacetimes.

	The variational method proceeds as follows. Define $\mass{}$ the 3-dimensionnal Minkoswki space, namely the affine space $\RR^3$ together with the quadratic form $Q = -\d t^2+\d x^2+\d y^2$ in cartesian coordinates $(t,x,y)$, and proceed as follows: 
	
					\begin{enumerate}
				\item consider a closed locally Euclidean surface $\Sigma$ of genus $g$ with $s\in \NN^*$ marked conical singularities $\theta_1,\cdots, \theta_s \in \RR_+^*$ and define $S$ the set of marked points;
				\item choose an arbitrary couple $(\tau,\TT)$ with $\tau : S \rightarrow \RR_+^*$  and $\TT$ a triangulation of $\Sigma$ whose set of vertices is $S$;
				\item for each triangle $T$ of $\TT$, choose a direct affine embeding of $T$ into $J^+(O):=\{t>0, Q<0\}\subset \mass{}$ in such a way that for each vertex $s$ of $T$ we have $T(s)=-\tau(s)$; 
				\item to each triangle $T$ is then associated the cone of rays from $O:=(0,0,0)$ through $T$ in $\mass{}$;  glue these cones together following the same combinatorics as $\TT$; the gluing is a 3-manifold $M$ endowed with a flat Lorentzian metrics on the complement of the rays through the vertices of $\TT$, furthermore we have a natural embedding of $\iota: \Sigma \rightarrow M$ in such a way that $\iota(\Sigma)$ is the boundary of a polyhedron $P$ of $M$;
				\item study the domain of $\tau \in (\RR_+)^{S}$ so that the polyhedron $P$ is convex, $\iota$ is then called convex, and show that for a given $\tau$ there is at most one triangulation $\TT$ (up to equivalence) for which the embedding $\iota$ is convex; a $\tau$ is then admissible if it does has such a triangulation;
				\item  define an Einstein-Hilbert functional on the space of admissible $\tau \in (\RR_+^*)^{S}$ in such a way that its critical points each induces a manifold $M$ with a non singular metric around the rays through the vertices of $\TT$; 	
				\item finally,  study this functional and show it admits a unique critical point.
		   \end{enumerate}
		  
		Another view point on our result is given by Penner  \cite{penner1987,MR3052157}. Penner constructed a cellulation of the Decorated Teichmüller space of a closed surface $\Sigma$ with $s$ marked points $S=\{\sigma_1,\cdots,\sigma_s\}$ viewed as the space of marked finite volume complete hyperbolic surface with $s$ cusps homeomorphic to $\Sigma\setminus S$ together with a choice of a positive number on each cusp.
		Consider such a surface $\Sigma^*$, the universal covering of $\Sigma^*$ naturally identifies with the usual hyperbolic plane $\H^2:=\{(t,x,y)\in \mass{}~|~Q(t,x,y)=-1,~t>0\}$ in $\mass{}$ and the positive number $\lambda_\sigma$ on each cusp $\sigma$ corresponds to a point on the lightlike rays corresponding to the cusp:
		\begin{itemize}
		 \item there exists a unique horocycle $\mathcal H_{\sigma,\lambda_\sigma}$ of length  $\lambda_\sigma$ around $\sigma$;
		 \item consider a ray $\R$ fixed by a parabolic holonomy of $\Sigma^*$ and consider a points $p\in \R$, the intersection of the future light cone of $p$ (ie the set $\{q\in \mass{}~|~ Q(q-p)=0, t(q-p)>0\}$) with $\H^2$ is a horocycle around $\R$ and every horocycles are obtained in this manner.
		\end{itemize}
		Penner then considers the surface obtained as the boundary of the convex hull of these points, he shows the surface obtained is locally Euclidean, its quotient by the holonomy of $\Sigma^*$ is a locally Euclidean surface $\Sigma_{\E^2}$ with $s$ conical singularities. Furthermore, the convex hull is a polyedron, the faces of which induce a cellulation on $\Sigma_{\E^2}$ with marked points $S$. He notes that this cellulation is simply the Delaunay cellulation of $(\Sigma_{\E^2},S)$. 
		It is not hard to see that:
		\begin{enumerate}
		 \item this construction actually defines a natural bijection from the decoracted Teichmüller space of $(\Sigma,S)$ to the deformation space of locally Euclidean metric on $\Sigma$ with arbitrary conical singularities on $S$;
		 \item The quotient by the holonomy of $\Sigma^*$ of the union of  $I^+(O):=\{ (t,x,y)\in \mass{}~|~-t^2+x^2+y^2 <0,~t>0\}$ with the rays fixed by parabolic holonomy of $\Sigma^*$, is a radiant spacetime with $s$  conical singularities all of angle $0$.
		\end{enumerate}
		Penner construction can thus be seen as the special of our Theorem were $\kappa=0$ and $(\Sigma,S)$ runs through all locally Euclidean surface with $s$ conical singularities at $S$ of arbitrary angles.

	\subsection*{Acknowledgements}
		This work has been initiated as part of a PhD project at Laboratoire de Mathématiques d'Avignon, Université d'Avignon et des Pays de Vaucluse, under Thierry Barbot's supervision, then continued as part of a project that has received funding from the European Research Council (ERC) under the European Union's Horizon 2020 research and innovation programme (grant agreement ERC advanced grant 740021--ARTHUS, PI: Thomas Buchert). The author thanks Thierry Barbot and Thomas Buchert for their continuous support, encouragments and valuable remarks,  François Fillastre and Marc Troyanov for their many corrections and comments on earlier versions of the manuscript,  Graham Smith for his remarks that led to significative simplifications, as well as Masoud Hasani, Jean-Marc Schlenker, Frédéric Paulin, Gye-Seon Lee, Suhyoung Choi,  Francis Bonahon, Anna Wienhard, Ivan Izmestiev, Erwann Delay, Miguel Sánchez, Abdelghani Zeghib, Philippe Delanoë,    Daniel Monclair, Vincent Pécastaing, Roman Prosanov, Rabah Souam, Andrea Seppi, Tengren Zhang,  Jeffrey Danciger, Nicolas Tholozan,  Qiyu Chen and Clément Guérin  for valuable discussions.
		
\section{Radiant 2+1 Singular spacetimes}

		We denote by $\mass{}$ the 3-dimensionnal Minkoswki space (ie the oriented affine space $\RR^3$ together with the quadratic form $Q:=-\d t^2+\d x^2+\d y^2$) and by $\GG$ the identity component of the Lie group of affine isometries of $\mass{}$. We denote by $O:=(0,0,0)\in \mass{}$ the origin of $\mass{}$.
		A vector $u\neq 0$ is spacelike (resp. timelike, resp. lightlike, resp. causal) if $Q(u)>0$ (resp. $Q(u)<0$, resp. $Q(u)=0$, resp. $Q(u)\geq 0$). A causal vector is future (resp. past) if its $t$ coordinate is positive (resp. negative).
		Minkowski space is naturally endowed with two order relations: the causal order $\leq$  and the chronological order $\chleq$ (the associated strict relation is denoted by $\chll$).  Given $p,q\in \mass{}$ then  $p<q$ (resp. $p\ll q$) if $q-p$ is future causal (resp. future timelike). The group $\GG$ preserves the orientation of $\mass{}$ as well as the causal and the chronological orders. We define the causal future of $p$ denoted by $J^+(p) := \{q\in M~|~ p \leq q\}$, as well as the chronological future of $p$ denoted by $I^+(p):=\{q\in M~|~ p\chll q\}$. The causal past as well as the chronological past are defined accordingly. 		
		A plane in $\mass{}$ is spacelike (resp. timelike, resp. lightlike) if the induced quadratic form is positive definite (resp. definite, resp. degenerated), a normal to such a place is a timelike vector (resp. spacelike vector, resp. lightlike vector). 
		
		\subsection{Flat spacetime analytical structures}
		
		The couple $(\GG,\mass{})$ is an analytical geometrical structure in the sense of Ehresmann \cite{MR794193}, Thurston or Goldmann \cite{MR957518}. A flat spacetime is a $(\GG,\mass{})$-manifold and for brievety sake, we will write $\mass{}$-manifold instead of $(\GG,\mass{})$-manifold. In any flat spacetime $M$ and for any $p\in M$,   A curve in a flat spacetime is future causal (resp. timelike) if is is increasing for the causal (resp. chronological) order.
		
		The couple $(\SO_0(1,2),I^+(O))$  is also an analytical geometrical structure.
		For brievety sake, we will write $\C$ instead of $I^+(O)$ and $\C$-manifold instead of $(\SO_0(1,2),\C)$-manifold. Such a manifold  is in particular a $(\GG,\mass{})$-manifold but is naturally endowed with a stronger structure. Indeed, in $\C$  the foliation given by the future causal geodesic rays from the origin is invariant under the action of $\GGL$ hence any $\C$-manifold is naturally endowed with a causal geodesic foliation. A $\C$-manifold $M$ thus comes with a height function $Q : M \rightarrow \RR_+$ induced by the function $Q:\C\rightarrow \RR_+$.

		\subsection{Singular $(G,X)$-manifolds}
		
		Let $(G,X)$ be an analytical structure, following \cite{GXramcover} we define a singular $(G,X)$-manifold as a Hausforff second countable topological $M$ space endowed with a $(G,X)$-structure on a open and dense  subset $\U$ locally connected in $M$. There exists a unique maximal extension of this $(G,X)$-structure to a maximal open and dense  subset $\reg(M)$ locally connected in $M$ called the regular locus of $M$.
		A morphism of such manifold is given by a continuous map which is a $(G,X)$-morphism on a open and dense  subset $\U'\subset \reg(M)$ locally connected in $M$. 
		
		A singular $(G,X)$-manifold is locally modeled on a familly $(X_\alpha)_{\alpha\in A}$ if for all $\alpha\in A$, $X_\alpha$ is a singular $(G,X)$-manifold and for all $x\in M$, there exists a neighborhood $\U$ of $x$ and an open $\V$ of some $X_\alpha$ such that $\U$ is isomorphic to $\V$.

		\subsection{Local models of singular lines}
		
			We now introduce the local models of the singular $\C$-manifolds we will consider.
			\begin{defi}[Massive particles model space]
		  Let $\alpha\in \RR_+^*$, the manifold $\mass{\alpha}$ is $\RR^3$ endowed with the flat Lorentzian metric $$\d s_\alpha ^2 = -\d t^2+\d r^2+\left(\frac{\alpha}{2\pi} r\right)^2 \d \theta^2$$ on $\reg(\mass{\alpha}):=\{r>0\}$ complement of the line $\sing(\mass{\alpha}):=\{r=0\}$ where $(t,r,\theta)$ are cylindrical coordinates of $\RR^3$. 
		\end{defi}

		For $\alpha>0$, the metric on $\mass{\alpha}$ induces a  unique $(\GG,\mass{})$-structure on $\reg(\mass{\alpha})$ such that the curves $t\mapsto c(t)=(t,r_0,\theta_0)$ are future causal for $r_0>0$ and all $\theta_0\in \RR/2\pi\ZZ$.
		\begin{defi}[BTZ line model space]
		  The manifold $\mass{0}$ is $\RR^3$ endowed with the flat Lorentzian metric $$\d s_0 ^2 = -2\d \tau \d \r+\d \r^2+ \r^2\d \theta^2$$ on $\reg(\mass{0}):=\{\r>0\}$ complement of the line $\sing(\mass{0}):=\{\r=0\}$ where $(\tau,\r,\theta)$ are cylindrical coordinates of $\RR^3$. 
		\end{defi}
		
		The metric on $\mass{0}$ induces a  unique $(\GG,\mass{})$-structure on $\reg(\mass{0})$ such that the curves $\tau\mapsto c(\tau)=(\tau,\r_0,\theta_0)$ are future causal for $\r_0>0$ and all $\theta_0\in \RR/2\pi\ZZ$.

		Note that the singular line of a massive particle is a timelike line while the singular line of $\mass{0}$ is lightlike. In a radiant 2+1 singular spacetime, the singular lightlike lines are all BTZ lines and the timelike singular lines are all massive particles.
		
		The model spaces $\mass{\geq0}$ are  singular $\mass{}$-manifold but not singular $\C$-manifolds. We thus introduce the following
		
		\begin{defi} For $\alpha\geq 0$ define $\C_\alpha := Int(J^+(O))$ with $O=(0,0,0)\in \mass{\alpha}$.
		\end{defi}
				
		\subsection{Causal structure}
			A $\C$-manifold $M$ comes with a causal structure eg a familly $(\leq_\U,\ll_\U)_{\U}$ of transitive relations each defined on an open subset $\U$ of $M$ which is inherited from the causal and chronological relation of $\C$. The causal structure on $\reg(\C_\alpha)$ can be extended to $\C_\alpha$ so that any $\C_{\geq0}$-manifold $M$ comes with a causal structure.
			A future causal curve is then a curve in $M$ which is locally increasing for $\leq$, the causal past/future of a point $p$ can then be defined accordingly and we denote them by $J^-(p)$ and $J^+(p)$ respectively.
			
			Note that $\leq_\U$ is an order relation for $\U$ small enough but this is not necessarily the case for $\leq_M$. We say that a $\C_{\geq0}$-manifold $M$ is {\it causal}  if $\leq_M$ is an order relation, we say furthermore that $M$ is {\it globally hyperbolic} if it is causal and for any $p,q\in M$, $J^+(p)\cap J^-(q)$ is compact. A {\it Cauchy-surface} of $M$ is a topological 2-dimensional submanifold $\Sigma$ in $M$ which intersects every future causal curves exactly once. One can prove a version of Geroch Theorem for $\C_{\geq0}$-manifolds (see for instance \cite{MR3783554}) which states that a $\C_{\geq0}$-manifold $M$ admits a Cauchy-surface if and only if it is globally hyperbolic. A $\C_{\geq0}$-manifold is {\it Cauchy-compact} is it admits a compact Cauchy-surface.
			
			A morphism $M_1\rightarrow M_2$ between globally hyperbolic $\C_{\geq0}$-manifolds is a Cauchy-embedding if it is injective and sends a Cauchy-surface of $M_1$ to a Cauchy-surface of $M_2$, the latter is then called a Cauchy-extension of $M_1$. A manifold $M_1$ is {\it Cauchy-maximal} if for any Cauchy-embedding $M_1\xrightarrow{\varphi}M_2$, the map $\varphi$ is an isomorphism. One can prove \cite{thesis,BTZI} a version of Choquet-Bruhat-Geroch Theorem for $\C_{\geq0}$-manifolds following the lines of \cite{sbierski_geroch} which states that any $\C_{\geq0}$-manifold admits a unique Cauchy-maximal Cauchy-extension.

		\subsection{$\H^2$-structure of the space of leaves and suspensions}
		
		Let $M$ be a $\C_{\geq0}$-manifold, $\reg(M)$ admits a natural causal geodesic foliation. We notice that in the model spaces $\C_\alpha$ the foliation can be extended to the whole $\C_\alpha$, furthermore by Propsition 1 of \cite{BTZI} if $\varphi:\U_\alpha\rightarrow \U_\beta$ is an a.e. $\GGL$-isomorphism between neighborhoods of singular points in $\C_\alpha$ and $\C_\beta$ respectively then $\alpha=\beta$ and $\varphi$ is induced  by an element of $\GGL$; hence the extended foliation to the whole $\C_\alpha$ induces a causal foliation on $M$.

		\begin{defi} For $\alpha\in \RR_+$, define $\H_\alpha^2$ as the space of leaves of $\C_\alpha$ and define the natural projection $\pi_\alpha : M\rightarrow \H_\alpha$.
		\end{defi}

		\begin{prop} For $\alpha \geq 0$, $\H_\alpha^2$ is homeomorphic to $\RR^2$ and comes with a natural singular $\H^2$-structure whose singular locus contains at most one point. 
		Furthermore,
		\begin{itemize}
		 \item if $\alpha=2\pi$, $\H^2_\alpha$ is regular and isomorphic to $\H^2$
		 \item if $2\pi\neq\alpha>0$, the singular point is a conical singularity of angle $\alpha$;
		 \item if $\alpha=0$, the singular point is a cusp.
		\end{itemize}
		\end{prop}
		\begin{proof}
			\begin{itemize}
			 \item To begin with, in $\C_\alpha$, define the plane $\Pi := \{t=1\}$ if $\alpha>0$ and $\Pi:=\{\tau=1\}$ if $\alpha=0$. The plane $\Pi$ intersects each leaf exactly once $\pi_{|\Pi}$ is an homeomorphism.
			 
			 \item Define the surface $\mathcal H^* := \{\tau=\frac{1+\r^2}{2\r}, \r>0\}$ if $\alpha=0$ and $\mathcal H^* := \{t^2-r^2=1, r>0\}$ if $\alpha >0 $. The Lorentzian metric of $\C_\alpha$ induces a hyperbolic metric on $\mathcal H^*$ which  intersects each leaf of $\reg(\C_\alpha)$ exactly once and the projection $\C_\alpha\rightarrow \H^2_\alpha$ induces an homeomorphism $\mathcal H^* \simeq (\H^2_\alpha\setminus \sing(\C_\alpha))$. Hence $\H^2_\alpha$ has a $\H^2$-structure defined the complement of $\sing(\C_\alpha))$ eg on the complement of a subset containing at most one point. 
			 
			 \item If $\alpha=2\pi$ then $\C_\alpha\simeq \C$ and the result follows.
			 \item If $\alpha=0$, one can check that $\mathcal H^*$ is complete and that the singular point of $\H_\alpha^2$ has a neighborhood of finite volume. The singular point is thus a cusp.
			 \item If $2\pi\neq \alpha>0$, then one can check that the length of the circle of radius $r>0$ in $\H^2_\alpha$ around the singular point is $\alpha r$. The singular point is a conical singularity of angle $\alpha$.
			\end{itemize}

		\end{proof}
	
		\begin{defi} Let $\Sigma$ be a $\H_{\geq0}$-manifold, let $(\U_i,\V_i,\varphi_i,\alpha_i)_{i\in I}$ be a $\H^2_{\geq0}$-atlas of $\Sigma$ with $\V_i\subset \H^2_{\alpha_i}$, let $\U_{ij}:=\U_i\cap \U_j$  and $\V_{ij}:=\varphi_i(\U_i\cap \U_j)$ for $i,j\in I$ such that $\U_i\cap \U_j\neq \emptyset$. We add the convention that $\alpha_i\neq 2\pi$ if and only if $\V_i$ contains a neighborhood of the singular point of $\H^2_{\alpha_i}$ so that for any $i,j\in I$ such that $\U_{ij}\neq \emptyset$ and $\U_i$ contains a singular point, then $\alpha_i=\alpha_j$ and the change of charts $\V_{ij}\xrightarrow{\varphi_{ij}} \V_{ji}$ comes from some $\phi_{ij}\in \GGL$ acting both on $\H^2_{\alpha_i}$ and $\C_{\alpha_i}$.
		
		 Define the suspension $\susp(\Sigma)$ of $\Sigma$ as the gluing of $\left(\pi_{\alpha_i}^{-1}(\V_i)\right)_{i\in I}$  via the maps $\left(\pi_{\alpha_i}^{-1}(\V_{ij}) \xrightarrow{\phi_{ij}} \pi_{\alpha_j}^{-1}(\V_{ji})\right)_{i,j\in I}$.
		\end{defi}
		\begin{rem} The suspension $\susp$ is a functor from the category of $\H^2_{\geq0}$-manifolds to the category of $\C_{\geq0}$-manifolds.
		 
		\end{rem}

		\begin{rem} By construction, $\susp(\Sigma)$ is a $\C_{\geq0}$-manifold with a natural projection $\susp(\Sigma) \rightarrow \Sigma$.   One can check that diamonds $J^+(p)\cap J^-(q)$ are compact and that $\susp(\Sigma)$ is causal, hence globally hyperbolic. Furthermore, the projection natural projection induces an homeomorphism $\pi : \Sigma_0 \rightarrow \Sigma$ for any  Cauchy-surface $\Sigma_0$.

		\end{rem}
		Another way to construct the suspension of a $\H^2_{\geq0}$-surface $\Sigma$ is to choose a geodesic cellulation of $\Sigma$ such that each cell is a polygon of $\overline {\H}^2$. 
		The surface $\Sigma$ can thus be seen as a gluing of a familly of cells $\mathcal P = (P_i)_{i\in I}$ along their edges $\mathcal E =(e_{i}^{(j)})_{i\in I,j\in J_i}$ (where $J_i$ parametrizes the edges of $P_i$) via  isometries $\phi_{e,e'} \in \GGL$ sending the edge $e$ to the edge $e'$. We denote by $\mathcal G$ the set of couples $(e,e')\in \mathcal E$ such that $e$ is glued to $e'$.
		
		We can then construct $\susp(\Sigma)$ by gluing the cones $C_i:=\pi^{-1}(P_i)$ for $i\in I$ along their faces $(\pi^{-1}(e))_{e\in\mathcal E}$ via the isometries $(\phi_{e,e'})_{(e,e')\in \mathcal G}$.

		\subsection{Radiant spacetimes}
		
		\begin{defi} A radiant spacetime is a Cauchy-compact Cauchy-maximal globally hyperbolic $\C_{\geq0}$-manifold $M$.
		\end{defi}
		We now state and prove a structure Theorem for radiant spacetimes. This result is in the line of Mess Theorem \cite{mess} and is akin to previous results by Bonsante and Seppi \cite{seppiH2} or the  author \cite{BTZI} though in a much simpler context. To the author's knowledge, while this result is expected and "folkoric", there is no existing reference to point to. We therefore provide a proof.

		\begin{theo}\label{theo:radiant_struc} Let $M$ be a radiant spacetime, there exists a compact singular $\H^2_{\geq0}$-manifold $\Sigma$ such that $M \simeq \susp(\Sigma)$.
		\end{theo}
		\begin{proof}
		 Let $\Sigma_0$ be a Cauchy-surface of $M$ and consider the natural projections $\pi_\alpha : \C_\alpha\rightarrow \H_\alpha$.		 
		 Consider  a $\C$-atlas $(\varphi_i,\U_i,\V_i)_{i\in I}$ of $\reg(M)$ such that each $\V_i$ is causally convex in $\C$, write $\U_{ij}:=\U_i\cap \U_j$ for $i\in I$ and for $i,j\in I$ such that $\U_i\cap \U_j\neq \emptyset$ write $\V_{ij}:= \varphi_i(\U_i\cap \U_j)$ as well as $\W_{ij}:=\pi(\V_{ij})\subset \H^2$. We then have a unique $\phi\in \GGL$ such that $\forall x\in \V_ij, \varphi_j\circ \varphi_{i}(x)=\phi\cdot x$. 
		 Hence, for any $i,j\in I$ such that $\U_i\cap \U_j\neq \emptyset$ we have the following commutative diagrams :
		 
		 $$
		 \xymatrix{\Sigma_0&&&M&&&\Sigma_0\\
		 \Sigma_0\cap \U_i\ar[u]\ar@/_2pc/[ddd]^\simeq&\ar[l]_ {\supset}\U_i\ar[ddl]^{\varphi_i}\ar[urr]^{\subset}&&\U_{ij}\ar[u]\ar[drr]^{\simeq}_{\varphi_j}\ar[dll]^{\varphi_i}_{\simeq}&&\U_j\ar[r]^{\subset}\ar[ull]_{\supset}\ar[ddr]_ {\varphi_j}& \Sigma_0\cap \U_j\ar[u]\ar@/^2pc/[ddd]_\simeq\\
		 &\V_{ij}\ar[r]_{\subset}\ar@{=}[d]&\C\ar@{-->}[rr]^{\exists! \phi\in \GGL}&&\C&\V_{ji}\ar[l]^{\supset}\ar@{=}[d]&
		 \\ \V_i\ar[d]^\pi&\ar[l]\V_{ij}\ar[d]^{\pi}\ar[r]_{\subset}&\pi^{-1}(\W_{ij})\ar[u]\ar@{-->}[rr]^\phi\ar[d]^\pi&&\pi^{-1}(\W_{ji})\ar[u]\ar[d]^{\pi}&\ar[l]^{\supset}\V_{ji}\ar[r]\ar[d]^{\pi}&\V_j \ar[d]^\pi 
		 \\ \W_i &\ar[l]\W_{ij}\ar@{=}[r]&\W_{ij}\ar@{-->}[rr]^\phi&&\W_{ji}\ar@{=}[r]&\W_{ij}\ar[r]& \W_j
		 } $$
		
		Since $\Sigma_0$ is acausal, the projection the maps $\Sigma_0\cap \U_i\rightarrow \W_i$ are injective and by definition surjective; $\Sigma_0$ as well as all the $\W_i$ are 2-dimensional manifolds, by invariance of domain the maps $\Sigma_0\cap \U_i\rightarrow \W_i$ are then homeomorphisms.
		The $\C$-structure on $M$ thus induces on $\Sigma_0$ a singular $\H^2$-structure, we call $\Sigma$ this singular $\H^2$-manifold. Proceeding the same way around singular points of $M$, the local models $\C_\alpha$ of $M$ induces a local model $\H_ \alpha$ for each singular point of $\Sigma$. The suspension $\susp(\Sigma)$ of $\Sigma$ is then given by the induced gluing of the cones $\pi^{-1}_{\alpha_i}(\W_i)$ along the $\pi^{-1}_{\alpha}(\W_{ij})$.
		
		One can then define a natural map $M\xrightarrow{\iota} \susp(\Sigma)$ on each chart $(\U,\V,\varphi)$ of the $(\C_\alpha)_{\alpha\geq0}$-atlas of $M$ with $\V\subset \C_\alpha$ as 
		$\iota : \U \rightarrow \pi_\alpha^{-1}(\pi_\alpha(\V)), x\mapsto \varphi(x)$. By construction,  the map $\iota$ is an injective a.e. $\C$-morphism.  Since $M$ is Cauchy-maximal and Cauchy-compact by Proposition 4 in \cite{BTZI}, the map $\iota$ is surjective thus an isomorphism.

		\end{proof}
		\begin{cor} Any radiant spacetime admits a embedded natural $\H^2_{>0}$-surface  which is a Cauchy-surface of its $\C_{>0}$ part.	
		 
		\end{cor}

		\begin{lem}\label{lem:future_radiant} Let $M$ be a radiant spacetime and let $\Sigma\subset \M$ be a Cauchy-surface. Denote by $\R:M\rightarrow \Sigma$ the function that associate to $x\in M$ the unique intersection point with $\Sigma$ of the leave  through $x$ of the natural foliation of $M$; denote as well $M_{>0}$ the $\C_{>0}$-part of $M$.
		
		Then, 
		$$J^+_M(\Sigma) = \overline{\{x\in M_{>0}~|~ Q(x)\geq Q(\R(x))\}}.$$ 
		\end{lem}
		\begin{proof}
		Since $\Sigma$ is a Cauchy-surface of $M$, $J^+_M(\Sigma)\cap J^-_M(\Sigma) = \Sigma$ and $J^+_M(\Sigma)\cup J^-_M(\Sigma) = M$. 
		Since $Q$ is increasing toward the future along the time-like leaves of the natural foliatiion of $M$, then $$\{x\in M_{>0}~|~ \pm Q(x)\geq  \pm Q(\R(x))\} \subset J^\pm_M(\Sigma).$$  Furthermore, since $M$ is globally hyperbolic and $\Sigma$ compact, $J^{\pm}_M(\Sigma)$ are closed. Hence 
		$$ \overline{\left\{x\in M_{>0}~|~ \pm Q(x)\geq  \pm Q(\R(x))\right\}} \subset J^\pm_M(\Sigma).$$
		Since $M_{>0}$ is dense in $M$, we have 
		$$ \bigcup_{\epsilon \in \{+,-\}}\overline{\left\{x\in M_{>0}~|~ \epsilon Q(x)\geq  \epsilon Q(\R(x))\right\}}= M $$
		furthermore 
		 $$\Sigma \subset \bigcap_{\epsilon \in \{+,-\}}\overline{\left\{x\in M_{>0}~|~ \epsilon Q(x)\geq  \epsilon Q(\R(x))\right\}} ~\subset~  J^+_M(\Sigma)\cap J^-_M(\Sigma) = \Sigma$$
		and it follows that 
		$$ \overline{\left\{x\in M_{>0}~|~ \pm Q(x)\geq  \pm Q(\R(x))\right\}} = J^\pm_M(\Sigma).$$
		
		 \end{proof}

				\section{Convex $\tau$-suspension and polyhedral embedding}
				\label{sec:tau_susp}

					In the present section we shall define and study a construction, we call $\tau$-suspension of a singular locally Euclidean surface $(\Sigma,S)$.  Every cellulation considered are geodesic in the sense that 2-facets are isometric to convex polygons of the Euclidean space $\E^2$ or equivalently if the image of developping map of each cell is a convex polygon of $\E^2$.
					\begin{defi}
					Let $(\Sigma,S)$ be a compact Euclidean surface with conical singularities with a finite subset $S$ of marked points such that $\sing(\Sigma)\subset S$,
					and let $\mathcal C$ be a cellulation of $(\Sigma,S)$.
				$\mathcal C$ is adapted if the set of vertices of $\mathcal C$ is exactly $S$.

					\end{defi}
					
					\begin{defi} Let $(\Sigma,S)$ be a compact Euclidean surface with conical singularities with a finite subset $S$ of marked points such that $\sing(\Sigma)\subset S$. Let $M$  be a singular $\mass{}$-manifold.
					An embedding $\iota : \Sigma \rightarrow M$ is  polyhedral if there exists a geodesic adapted cellulation $\mathcal C$ of $(\Sigma,S)$ such that on each cell $C$, the restriction of $\iota$ to  
					 $Int(C)$ is an isometric affine map into the regular locus of $L$.
					\end{defi}
					
					The notion of isometric affine map is well defined in this context. Indeed,  both $\E^2$ and $\mass{}$ are affine spaces endowed with a semi-Riemannian metric, the regular locii of $\Sigma$ and $M$ are endowed with a $\E^2$-structure and a $\mass{}$-structure respectively.

					The quadratic form on $\mass{}$ is a $\SO_0(1,2)$-invariant  function defined on the underlying vector space $\overrightarrow{\mass{}}$: 
			$$\fonction{\quadform}{\overrightarrow{\mass{}}}{\RR}{(t,x,y)}{t^2-x^2-y^2}$$
		We extend the definition of $\quadform$ to $\mass{}$ via the identification $\mass{}\rightarrow \overrightarrow{\mass{}}, x\mapsto x-O$.
		The map $T$ is positive on the future of the origin in $\mass{}$, namely $J^+(O):= \{(t,x,y)\in \RR^3 ~|~ -t^2+x^2+y^2\leq 0 \text{~and~} t>0\}$;  furthermore, it  induces a Cauchy time function on  $I^+(O)$ i.e. a increasing map for $\leq$ which is surjective on every inextendible future causal curve of $I^+(O)$. 
		Since the $T$ is $\SO_0(1,2)$-invariant, it induces a well defined non $\leq$-decreasing function on every radiant 2+1 singular spacetimes.

		In a radiant 2+1 singular spacetimes, the surface $T = 1$ is a hyperbolic surface with conical singularities and cusps which is complete and has finite volume. One can prove that the association $M \mapsto \{T=1\}$ induces a bijection from the deformation space of marked radiant 2+1 singular spacetimes to the deformation space of marked finite volume complete hyperbolic surfaces with conical singularities and cusps.
		
		\subsection{Affine embedding of triangles into $\mass{}$}

							\begin{lem}\label{lem:prolongement}
								Let $T=[ABC]$ be a non degenerated Euclidean triangle and let $\tau : \{A,B,C\} \rightarrow \RR$.
								
								There exists a unique $(\tau_0,\omega)\in \RR\times \E^2$
								such that the map
								 $$\fonction{\widetilde \tau}{\E^2}{\RR}{x}{\tau_0-d(x,\omega)^2 }$$
								 extends $\tau$.

								 Furthermore, if $\tau\geq 0$  then $\tau_0>0 $ and $\widetilde \tau > 0$ on the triangle $[ABC]$ except possibly at $A,B$ or $C$.
							\end{lem}
							\begin{proof} Identify $\E^2$ to $\RR^2$ via cartesian coordinates $(x,y)$; without loss of generality, we can assume $A=(0,0)$ and we write $B=(x_B,y_B)$ and $C=(x_C,y_C)$.
							 Finding  $\widetilde \tau$ is equivalent to solving the following system in $\omega=(x_\omega,y_\omega)$ and $\tau_0$.
							 \begin{eqnarray*}
							   \left\{ \begin{array}{l}
										\displaystyle \tau_A=\tau_0- x_\omega^2-y_\omega^2  \\
										\displaystyle  \tau_B=\tau_0- (x_\omega-x_B)^2-(y_\omega-y_B)^2\\
										\displaystyle \tau_C=\tau_0- (x_\omega-x_C)^2-(y_\omega-y_C)^2
									   \end{array}
									\right. &\Leftrightarrow& \left\{ \begin{array}{l}
										 \displaystyle x_\omega^2+y_\omega^2+\tau_A= \tau_0 \\
										\displaystyle  \tau_B-\tau_A-x_B^2-y_B^2=2x_\omega x_B +2y_\omega y_B\\
										\displaystyle  \tau_C-\tau_A-x_C^2-y_C^2=2x_\omega x_C +2y_\omega y_C\\
									   \end{array}
									\right.  \\
							 \end{eqnarray*}
							 Since $A,B,C$ are in general position, the second and third line form a non singular linear system with unknown $(x_\omega,y_\omega)$. The first line is already solved.
							 Existence and uniqueness of $\widetilde \tau$ follows.

							 Assume $\tau\geq 0$ (resp. $\tau>0$), since $A,B,C$ are distinct, $\omega$ is distinct from one of them, say $P \in \{A,B,C\}$, then $0\leq \tau_P = \tau_0-d(P,\omega)^2< \tau_0$.
							 Furthermore, $\widetilde \tau$ is strictly concave then its minimum on $[ABC]$  is reached in the set of extremal points eg $\{A,B,C\}$ and nowhere else.

							 \end{proof}

							\begin{lem} \label{lem:isom_equiv}
							Let $A,B,A',B' \in J^+(O)$, $A\neq B$, $A'\neq B'$, such that $\quadform(A)=\quadform(A')$, $\quadform(B)=\quadform(B')$ and $\quadform(B-A)=\quadform(B'-A')$, there exists a unique isometry $\gamma\in \SO_0(1,2)$ such that 
							$\gamma A = A'$ and $\gamma B=B'$.  Furthermore, if $C$ is on a given side of the oriented plane $(OAB)$ then $\gamma C$ is on the same side of $(OA'B')$.
							\end{lem}
							\begin{proof}
							  For all $\tau_0\geq0$, the group $\SO_0(1,2)$ acts transitively on each sets $(\quadform_{|J^+(O)})^{-1}(\tau_0)$. The thus exists some $\gamma_0\in \SO_0(1,2)$ such that $\gamma_0 A = A'$. The stabilizer of $A'$ under the action of $\SO_0(1,2)$ is a 1-parameter subgroup (either parabolic or elliptic depending on wether $(OA')$ is lightlike or timelike); under its action the orbit of $\gamma B$  is $$\left\{x \in J^+(O)~:~ \quadform(x-A') = \quadform\left(\gamma_0 B-A'\right), \quadform(x)= \quadform(\gamma_0 B)\right\}.$$
							  On the one hand, $\quadform(B')=\quadform(B)=\quadform(\gamma_0B)$ and $\quadform(\gamma_0 B-A')=\quadform(\gamma_0(B-A)) = \quadform(B-A)$
							\end{proof}
								
							\begin{prop}\label{prop:affine_emb}
								Let $T=[ABC]$ be an oriented non degenerated Euclidean triangle and let $\tau : \{A,B,C\}\rightarrow \RR_+$. 
								There exists an  isometric direct affine embedding $\iota : T\rightarrow J^+(O)$ such that $\tau = -\quadform\circ \iota_{|\{A,B,C\}}$ where $\iota(T)$ is endowed with the orientation  induced by a future pointing normal vector. 
								
								Furthermore, 
								\begin{itemize}
								 \item such an embedding is unique up to the action of $\SO_0(1,2)$;
								 \item $-\quadform\circ \iota = \widetilde \tau$ where $\widetilde \tau$ is given by Lemma \ref{lem:prolongement}.
								 \end{itemize}

							\end{prop}
							\begin{proof}
								Endow  $\mass{}$ with  cartesian coordinates $(x,y,t)$, write $O=(0,0,0)$ the origin and identify $\E^2$ with $\{t=0\}\subset \mass{}$.
								Take $(\tau_0,\omega)\in \RR\times \E^2$ and $\widetilde \tau$ given by Lemma \ref{lem:prolongement} and define $$ \fonction{\iota}{T}{\mass{}}{x}{x+
								\overrightarrow u} \quad \text{with} \quad 
								\overrightarrow u=\begin{pmatrix}\sqrt{\tau_0}\\ -\overrightarrow{O\omega}\end{pmatrix}
	$$
								Write $\omega = (x_\omega,y_\omega)$. For $(x,y)\in T$, we have $$-\quadform\circ \iota(x,y) = \sqrt{\tau_0}^2-(x-x_\omega)^2-(y-y_\omega)^2 = \widetilde \tau(x,y). $$
								Since $\tau\geq 0$, by Lemma \ref{lem:prolongement}, $\widetilde \tau\geq 0$ hence $-\quadform\circ \iota_{|T}\geq0$. Moreover, $\sqrt{\tau_0}>0$ thus $\iota(T)\subset J^+(O)$.  The existence statement follows as well as the second additionnal point.
								
								If $\iota$ and $\iota'$ are two such embeddings, by Lemma \ref{lem:isom_equiv}, there exists a unique isometry sending $\iota(A)$ on $\iota'(A)$ and $\iota(B)$ on $\iota'(B)$. The, there exists exactly two points $P_1,P_2\in J^+(O)$ such that $-\quadform(P_i)=\tau(C)$ and $\d(A,C)^2 = \quadform(\iota(P_i)-\iota(A))$ and  $\d(B,C)^2 = \quadform(\iota(P_i)-\iota(B))$
								for $i\in \{1,2\}$. Since these two points are image from one another by the reflexion across the plane $(O\,\iota(A)\,\iota(B))$ which is indirect and preserves $\leq$, exactly one induces the right orientation.

							\end{proof}

							\begin{defi}[Distance-like function]
								Let $(\Sigma,S)$ be singular locally Euclidean surface.
								A function $f:\Sigma \rightarrow \RR$ is distance-like if
								there exists a geodesic triangulation $\TT$ of $\Sigma$ whose vertices contains $S$ such that for all $T\in \TT$, there exists  
								$\omega \in \E^2$ and $\tau_0 \in \RR$ such that
								$$\forall x\in T, \quad f(x) = \tau_0-\d(\D(x),\omega)^2$$
								where $\D:T\rightarrow \E^2$ is a developping map of $T$.

								Such a triangulation is adapted to $f$ 
							\end{defi}
							
							\begin{rem} 
							 Let $(\Sigma,S)$ be singular locally Euclidean surface, let $M$ be a radiant spacetime. For any polyhedral embedding $\iota:\Sigma\rightarrow M$,  the map $Q\circ \iota : \Sigma\rightarrow \RR_+$ is distance-like.
							\end{rem}

							\begin{prop}\label{prop:prolongement} 
							Let $(\Sigma,S)$ be singular locally Euclidean surface.
							Let  $\TT$ be an adapted triangulation of $(\Sigma,S)$.

									For all $\tau:S\rightarrow \RR$, there exists a unique distance-like extension $\widetilde \tau$ for which $\TT$ is adapted.
							\end{prop}
							\begin{proof} Apply Lemma \ref{lem:prolongement} to each triangle of $\TT$.
							\end{proof}
							\begin{defi} 
							Let $(\Sigma,S)$ be singular locally Euclidean surface.
							For an adapted triangulation  $\TT$ of  $(\Sigma,S)$ and $\tau:S\rightarrow \RR_+$ we denote by
							  $\widetilde \tau_{\TT,\tau}$ the extension of $\tau$ given by Proposition \ref{prop:prolongement}.
							  
							\end{defi}

							\begin{defi}[Equivalent Triangulations ]\label{defi:triang_equi}
							Let $(\Sigma,S)$ be singular locally Euclidean surface.
								Let  $\tau:S\rightarrow \RR_+$,
								two triangulations $\TT_1,\TT_2$ of $(\Sigma,S)$ are $\tau$-equivalent if $$\widetilde \tau_{\tau,\TT_1}=\widetilde \tau_{\tau,\TT_2}.$$
							\end{defi}

							\begin{defi}[$\tau$-suspension] \label{defi:tau_susp}
							Let $(\Sigma,S)$ be singular locally Euclidean surface and 
							$\widetilde \tau : \Sigma \rightarrow \RR$ distance-like. 
	
							Choose a triangulation $\TT$ adapted to $\widetilde \tau$ (but not necessarily adapted to $(\Sigma,S)$.
							For each  $T\in \TT$, denote by $\iota_T:T\rightarrow J^+(O)$ the affine embedding of $T$ given by Proposition \ref{prop:affine_emb} and denote by $C_T:=\{t\cdot\iota_T(x) ~:~ t\in \RR_+^*, x\in T\}$. For each edge $e$ of $\TT$ bounding $T_1,T_2 \in \TT$ , let $\gamma_e$ be the isometry given by Lemma \ref{lem:isom_equiv} sending the face of $C_{T_2}$ associated to $e$ to the face of $C_{T_1}$ associated to $e$.

							Define $M(\widetilde \tau)$ as the radiant spacetime obtained by gluing the family $(C_T)_{T\in \TT}$ via the isometries $(\gamma_e)_{e \in \text{ edges}(\TT)}$

							\end{defi}

							\begin{prop}  
							 Let $(\Sigma,S)$ be singular locally Euclidean surface and 
							$\widetilde \tau : \Sigma \rightarrow \RR$ distance-like. The spacetime $M(\widetilde \tau)$ does not depend on the choice of the  triangulation $\TT$ adapted to $\widetilde \tau$.							 
							\end{prop}
							
							\begin{proof}							   
							  Consider two triangulations $\TT_1$ and $\TT_2$ adapted to $\widetilde \tau$. There exists a geodesic triangulation of $(\Sigma,S)$ adapted to $\widetilde \tau$ thinner that both $\TT_1$ and $\TT_2$. It thus suffices to show that 
							  on a given triangle $T\in \TT_1$ any decomposition of $T$ into smaller triangles $(T_i)_{i\in \lsem 1,n\rsem}$ induces a gluing isomorphic to $C_T$.
							  
							  Consider two adjacent triangles $T_1,T_2 \subset T$ sharing a common oriented edge $e$. On the one hand, we have  affine embeddings $\iota_{1}$ and $\iota_{2}$, on the other hand we have the affine embedding $\iota_T$. The isometry from $\gamma$ gluing $C_{T_2}$ to $C_{T_1}$ is the unique isometry $\gamma\in \SO_0(1,2)$ sending $\iota_{2}(e)$ on $\iota_1(e)$. We notice that   $\gamma\iota_2$ and $ \iota_{T|T_2}$ both satisfy the hypotheses of Proposition \ref{prop:affine_emb}, they thus differ by an isometry of $SO_0(1,2)$ and since the image of $e$ is identical, the uniqueness part of Lemma \ref{lem:isom_equiv} implies this isometry is the identity; hence $\gamma \iota_2 = \iota_{T|T_2}$. Therefore, $\gamma C_{T_2}$ is exactly the sub-domain $C_{T}$ corresponding to $T_2\subset T$, hence the gluing of $C_{T_1}$ with $C_{T_2}$ is isomorphic to the subdomain of $C_T$ corresponding to $(T_1 \cup T_2)\subset T$. By induction, the result follows.`

							\end{proof}

					\begin{defi} Let $(\Sigma,S)$ be singular locally Euclidean surface, let $(M_1,\iota_1)$ and $(M_2,\iota_2)$ be two radiant spacetimes together with a polyhedral embedding of $(\Sigma,S)$. 
					
					We say that $(M_1,\iota_1)$  is equivalent to $(M_2,\iota_2)$  if there exists an isomorphism $\varphi : M_1\rightarrow M_2$ such that $\iota_2 = \varphi \circ \iota_1$.
					 
					\end{defi}

					\begin{theo} 
					Denoting by $\sim$ the equivalence relation among polyhedral embeddings and distance-like functions, the function
					 $$ \fonctionn{ \left\{(M,\iota) ~:~ \begin{array}{c}
					 M \text{ radiant}\\
												  \iota~\text{polyhedral} \\ \text{embedding} 
												  \end{array}
\right\}/\sim }{ \left\{ \widetilde \tau   ~:~\begin{array}{c}
												  \widetilde \tau \text{ distance-like on $(\Sigma,S))$}
												  \end{array}\right\}/\sim   }{(\iota,M)}{Q\circ \iota  }$$
					is bijective of inverse $\widetilde \tau \mapsto M(\widetilde \tau)$.
					
					\end{theo}
					\begin{proof}
						Denote by $\Phi$ the function above.
						For any $\widetilde \tau$ distance-like on $(\Sigma,S)$, by Propostion \ref{prop:affine_emb},  the construction of $M(\widetilde \tau)$ ensures $\Phi(M(\widetilde \tau))=\widetilde \tau$. Hence, $\Phi$ is surjective.
						Let $(M_1,\iota_1)$ be polyhedral embedding of $(\Sigma,S)$, let $\widetilde\tau := \Phi(M_1,\iota_1)$,  $M_2 = M(\widetilde \tau)$ with its  polyhedral embedding $\iota_2:\Sigma\rightarrow M_2$. 
						By Theorem \ref{theo:radiant_struc}, for $i\in \{1,2\}$,  $M_i$ is isomorphic to $\susp(\Sigma_i$ with $\Sigma_i$, the space of leaves of the natural causal foliation of $M_i$ endowed with its $\H^2_{\geq0}$-structure. Define the natural projections $\pi_i:M_i\rightarrow \Sigma_i$. Denote by $\R : \C \rightarrow \H^2$ the map that associate to any $x\in \C$ the intersection point of the ray through $x$ with $\H^2\subset \C$. 
						
						 For $i\in \{1,2\}$, the map  $\pi_i\circ  \iota_i : \Sigma\rightarrow \Sigma_i$ is an homeomorphism.  The map $h := \pi_2\circ  \iota_2 \circ (\pi_1\circ \iota_1)^{-1}$ is then an homeomorphism.						
						We shall prove $g$ is an a.e. $\H$-morphism from $\Sigma_1$ to $\Sigma_2$ and hence that $\susp(h) : M_1\rightarrow M_2$ is an isomorphism.

						Choose a geodesic triangulation $\TT$ of $\Sigma$ adapted to $\widetilde \tau$, its image by $\pi_i\circ \iota_i$ is a geodesic triangulation of $\Sigma_i$. Note that $h$ sends cell of $\Sigma_1$ to cell of $\Sigma_2$, thus in order to prove that $h$ is a $\H^2$-morphism, it suffices to prove that its restrictions to each cell of $\Sigma_1$  are isometries.

						Let $T\in \TT$, $x\in T\setminus S$, and, for $i\in \{1,2\}$, choose a chart $(\U_i,\V_i,\varphi_i)$ of $M_i$ around $\iota_i(x)$ such that  $\V_i$ is a cone of $\C$.
						Let $T_\Sigma\subset T\setminus S$ be a triangle of $\Sigma$ containing $x$.
						For $i\in\{1,2\}$, write $T_{M_i} := \iota_i(T_\Sigma)$, $T_{\Sigma_i} := \pi_i\circ \iota_i(T_\Sigma)$, $T^{(i)}_{\C} := \varphi_i(T_{M_i})$ and $T^{(i)}_{\H^2}:=\R\circ \varphi_i(T_{M_i})$. By construction of the $\H^2$-structure on $\Sigma_i$, $\varphi_i$ induces a chart $\overline \varphi_i : T_{\Sigma_i}\rightarrow T_{\H^2}^{(i)}$. By Lemma \ref{lem:isom_equiv}, there exists a unique $\phi\in \GGL$ such that $\varphi_2\circ \iota_2 = \phi\circ \varphi_1\circ \iota_1$. Since $\R$ commutes with the action of $\GGL$, we then have $\R\circ \varphi_2\circ \iota_2 = \phi\circ\R\circ  \varphi_1\circ \iota_1$. The following commutative diagram sums up the situation.

						$$
						\xymatrix{
						\Sigma\ar[ddd]^{\pi_1\circ\iota_1}_{\rotatebox{90}{$\sim$}}\ar@{=}[rrrrrr]&&&&&&\Sigma\ar[ddd]_{\pi_2\circ \iota_2}^{\rotatebox{90}{$\sim$}}
						\\&T_{\Sigma}\ar[ul]_*-{\rotatebox{-45}{$\supset$}} \ar[d]_{\rotatebox{90}{$\sim$}}^{\iota_1}\ar@{=}[rrrr]&&&&T_{\Sigma}\ar[d]_{\rotatebox{90}{$\sim$}}^{\iota_2}\ar[ur]^*-{\rotatebox{45}{$\subset$}}&
						\\& T_{M_1}\ar[d]_{\rotatebox{90}{$\sim$}}^{\pi_1}\ar[r]_{\sim}^{\varphi_1}& T_{\C}^{(1)} \ar[d]_{\rotatebox{90}{$\sim$}}^{\R} \ar@{-->}[rr]^{\exists!\phi \in \GGL}
						&&T_{\C}^{(2)} \ar[d]_{\rotatebox{90}{$\sim$}}^{\R}&T_{M_2}\ar[d]_{\rotatebox{90}{$\sim$}}^{\pi_2}\ar[l]^{\sim}_{\varphi_2}&
						\\ \Sigma_1&\ar[l]_{\supset}T_{\Sigma_1} \ar[r]_{\sim}^{\overline\varphi_1}& T^{(1)}_{\H^2} \ar@{-->}[rr]^\phi
						&& T^{(2)}_{\H^2}&T_{\Sigma_2} \ar[l]^{\sim}_{\overline\varphi_2} \ar[r]^\subset& \Sigma_2
						}
						$$
						
						Therefore, the (co-)restriction of $h$ from $T_{\Sigma_1}$ to $T_{\Sigma_2}$ is an isometry. It follows that $h$ is an isometry from triangle of $\pi_1\circ \iota_1(\TT)$ to triangle of $\pi_2\circ\iota_2(\TT)$.

					\end{proof}

				\subsection{Convex embeddings}

					We start by precising the notion of convex embedding \ref{defi:plong_conv}.
					Corollary \ref{prop:plong_conv} is the main result of this subsection: it provides 
					a parametrization of convex embeddings by a domain of $\RR^S$.

					\begin{defi}[Convex Polyhedral embedding] \label{defi:plong_conv}
						Let $M$ be a radiant spacetime with  $\iota:\Sigma\rightarrow M$ a polyhedral embedding.

						The embedding $\iota$ is convex if $J^+(\iota(\Sigma))$ is convex in the sense that for any spacelike geodesic $c:[a,b]\rightarrow M$, if $\{c(a),c(b)\}\subset J^+(\iota(\Sigma))$ then 
						$c([a,b]) \subset J^+(\iota(\Sigma))$

					\end{defi}
					\begin{defi}[Q-convexity on $\RR$]
						A function $f : \RR \rightarrow \RR$ is Q-convex (resp. Q-concave)
						if $f$ is continuous, piecewise $\mathscr{C}^1$ and if for all $t_0\in I$,
						$$\lim_{t_0^-}f' \leq \lim_{t_0^+} f'\quad\quad \left( \mathrm{resp.}~~ \lim_{t_0^-}f' \geq \lim_{t_0^+} f'\right).$$
					\end{defi}

					\begin{defi}[Q-convexity on $\E^2_{>0}$-surface]

						A function $\widetilde \tau : \Sigma \rightarrow \RR$ is Q-convex (resp. Q-concave) if for all geodesic 
						$c: I \rightarrow \Sigma\setminus S$, the restriction of $\widetilde\tau$ to $c$ is Q-convex (resp. Q-concave).
					\end{defi}

					\begin{lem}\label{lem:diff_Q_convexe} Let $f,g:[a,b]\rightarrow \RR$ two functions piecewise of the form $x\mapsto-x^2+\alpha x+\beta$ and Q-convex such that $f(a)\geq g(a)$
					and $f(b)\geq g(b)$. If $f$ is $\mathscr C^1$ then $g\leq f$.
					\end{lem}
					\begin{proof}
						$g-f$ is piecewise affine, since $f$ is $\mathscr C^1$ and $g$ Q-convex, $g-f$ is Q-convex thus convex. Since $g-f$ is negative both at $a$ and $b$, it is negative on
						$[a,b]$.
					\end{proof}

					\begin{prop}\label{prop:conv_Q_conv}
						Let $\widetilde \tau:\Sigma \rightarrow \RR_+$ distance-like and $M:=M(\widetilde \tau)$ with its associated polyhedral embedding $\iota:\Sigma \rightarrow M$.

						The embedding $\iota$ is convex if and only if $\widetilde \tau$ is Q-convex.
					\end{prop}

					\begin{proof}
					\begin{itemize}
						\item Assume that $\widetilde \tau$ is  Q-convex and consider a spacelike geodesic $c:[a,b]\rightarrow$ such that $c(a),c(b) \in J^+(\Sigma)$. A direct computation in a chart gives that both $ Q\circ c$ and $ Q\circ \R\circ c$ are  piecewise of the form $s \mapsto -s^2+\alpha s+\beta$ and that  $\iota^{-1}\circ \R\circ c$ has the same Q-convexity as $\widetilde \tau$ on edges of any given triangulation adapted to $\widetilde \tau$.  Futhermore, $ Q\circ c$ is $\mathscr C^1$, then by Lemma \ref{lem:diff_Q_convexe}, the sign of $ Q\circ c- Q\circ \R\circ c$ is constant and Lemma \ref{lem:future_radiant} allows to conclude that  $c([a,b])\subset J^+(\Sigma)$. Finally, $J^+(\Sigma)$ is convex, hence $\iota$ is convex.
						\item Assume that $\widetilde \tau$ is  not Q-convex, there thus exists edge $e$ in $\Sigma$ around which $\widetilde \tau$ is strictly $Q$-concave. Consider two points $x,y$ in $\Sigma$ each in a different side of said edge. We can choose $x,y$ closed enough so that they lie in a chart of $M$ around $\iota(e)$. Then consider the geodesic $c:[a,b]\rightarrow M$ in this chart from $x$ to $y$; we can apply the same line of reasonning as above to show that $Q\circ c<Q\circ \R\circ c$ on $]a,b[$, thus that $c(]a,b[)$ is not in $J^+(\iota(\Sigma))$ and hence $J^+(\iota(\Sigma))$ is not convex.

					\end{itemize}

					\end{proof}

					\begin{prop} \label{prop:unique_Q_convexe}

					Let  $\tau\in \RR_+^{S}$, up to equivalence, there is at most one triangulation $\TT$ 
						such that the distance-like extension  $\widetilde \tau_{\TT,\tau}: \Sigma\rightarrow \RR_+$ is Q-convex.
					\end{prop}
					\begin{proof}
					  Let $\TT_1$ and $\TT_2$ be two triangulations of $\Sigma$ such that both $f_1:=\widetilde \tau_{\tau,\TT_1}$
					  and $f_2:=\widetilde \tau_{\tau,\TT_2}$ are Q-convex.
					  For all $e$ edge of $\TT_1$, the function $f_{1|e}$ is quadratic while
					  the function $f_{2|e}$ is piecewise quadratic and Q-convex;
					  by Lemma \ref{lem:diff_Q_convexe}, it thus follows that $f_2\leq f_1$ on $e$.
					  For any $T$ triangle of $\TT_1$, $f_1\geq f_2$ on $\partial T$ and applying again Lemma \ref{lem:diff_Q_convexe}
					  along any segment $[a,b]$ of $T$ with $a,b\in \partial T$, we deduce that $f_1\geq f_2$ on $T$.  Therefore, $f_1\geq f_2$ on $\Sigma_1$. We show that same way that $f_1\leq f_2$ hence  $f_1=f_2$. The triangulations $\TT_1$ and $\TT_2$ are then equivalent.

					\end{proof}

					\begin{defi}[Admissible times] \label{defi:temps_admissbles}

						Define $\PSigmaS$ the set of $\tau\in \RR_+^{S}$ such that there exists an adapted  triangulation $\TT$ of $\Sigma$ inducing a Q-convex distance like extension $\widetilde \tau_{\tau,\TT}$. Elements of  $\PSigmaS$ are called admissible times.

						For $\tau \in \PSigmaS$, we denote by  $\TT_\tau$ the unique adapted triangulation of 
						 $\Sigma$ (up to equivalence) such that $\widetilde \tau_{\tau,\TT_\tau}$ is Q-convex. We define as well $\widetilde \tau_\tau := \widetilde \tau_{\tau,\TT_\tau}$ and $M(\tau):=M(\widetilde \tau_\tau)$.

					\end{defi}

					\begin{cor} The function
					  $$ \fonctionn{\{ \widetilde \tau : \Sigma\rightarrow \RR_+ ~~\text{ distance-like Q-convex}\} }{\PSigmaS}{\widetilde \tau}{\tau=\widetilde\tau_{|S}}$$
					  is bijective.
					\end{cor}
					\begin{prop}\label{prop:plong_conv} With $\sim$ the equivalence relation defined by 
					$$ (M_1,\iota_1)\sim (M_2,\iota_2) \quad \Leftrightarrow \quad \exists h:M_1\rightarrow M_2 \text{ isomorphism s.t. } \iota_2=h\circ \iota_1$$The function
					  $$ \fonctionn{ \left\{(M,\iota) ~:~ \begin{array}{c}M \text{ radiant}\\
												  \iota:\Sigma\rightarrow M ~\text{polyhedral} \\ \text{convex embedding} 
												  \end{array}
\right\}/\sim }{\PSigmaS}{(\iota,M)}{\left(Q\circ \iota\right)_{|S}  }$$
					  is bijective.
					\end{prop}

			\section{Domain of admissible times}
				\label{sec:etude_domaine_convexe}
					In the whole section, we give ourselves a marked locally euclidean surface with conical singularities $(\Sigma,S)$. While Proposition \ref{prop:plong_conv} above parametrizes polyhedral embeddings by the domain $\PSigmaS\subset \RR^S$, for now little is known about the latter and before studying the image of  $\tau\mapsto M(\tau)$ we shall provide a thorough description. More precisely, we prove the following.
					\begin{theo}\label{theo:domain_description} Let  $\mathbf{1}_S$ the indicator function of $S$ and $H$ the linear hyperplane of $\RR^S$ normal to $\mathbf{1}_S$ and $\pi$ the orthogonal projection onto $H$. Define $\overline\PSigmaS = \pi(\PSigmaS)\subset H$. 
					Then 
					\begin{enumerate}[(a)]
					 \item $\overline\PSigmaS$ is a convex compact polyhedron
					 \item $\displaystyle\PSigmaS = (\overline{\PSigmaS} + \RR\cdot\mathbf{1}_S)\cap \RR_+^S.$
					 \item The interior of $\overline \PSigmaS$ contains $0\in \RR^S$
					 \item With $\mathscr T := \{\TT_\tau : \tau\in \PSigmaS\}$, each $\CellPSigmabar:= \{\pi(\tau) ~|~\TT_\tau=\TT\}\subset \overline \PSigmaS$  is a convex polyhedra of $H$  for $\TT\in E$. Furthermore, the family $(\CellPSigmabar)_{\TT\in \mathscr T}$ is a finite cellulation of $\overline{\PSigmaS}$.
					 \item The support planes of $\PSigmaS$ are either of the form $"\tau_\sigma = 0"$ or $"Q^*(\tau)=0"$ with $Q$ an unflippable immersed hinge (see Definitions \ref{defi:hinge}, \ref{defi:immersed_hinge}, \ref{defi:Q_forme_affine} and \ref{defi:immersed_hinge_flippable} below).					
					\end{enumerate}
					
					\end{theo}

					The starting point is to study "local" criteria for Q-convexity. By  local we meaning at each edge of a given triangulation, the following definitions make this notion precise.
					\begin{defi}[Hinge]\label{defi:hinge} A hinge is a tetragon $[ABCD]\subset \E^2$ together with a diagonal $[AC]$ such that $[AC]\subset [ABCD]$.
					\end{defi}
					Beware that the tetragon of a  hinge need not be convex. If convex with vertices in general positions, a tetragon may define two hinges: one for each interior diagonal; otherwise, only one hinge may be defined.
					\begin{defi}[Hinge flipping]
					 Let $Q=([ABCD],[AC])$ be hinge. If $[ABCD]$ is convex and the four points $A,B,C,D$ are in general positions, then $Q$ is flippable and its flipping is the hinge $Q'=([ABCD],[DB])$.
					\end{defi}

					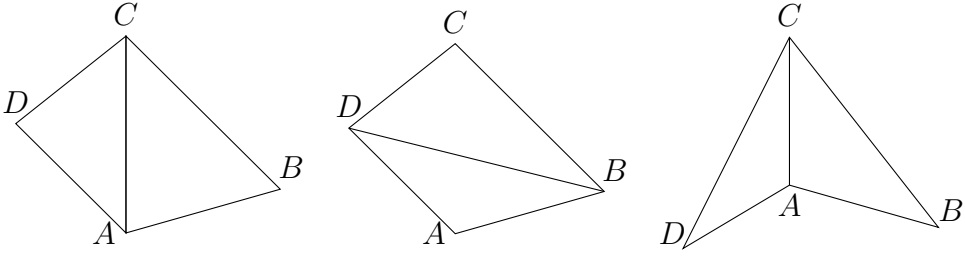
\begin{figure}[h]
					\caption{Different Hinges}
					
							\begin{tikzpicture}[scale=0.29]
								\draw (0,0) -- (-5,5) -- (0,9) -- cycle;
								\draw  (0,0) -- (7,2) -- (0,9) -- cycle;
								\coordinate[label=$A$] (A) at (-1,-1);
								\coordinate[label=$D$] (B) at (-5,5);
								\coordinate[label=$C$] (C) at (0,9);
								\coordinate[label=$B$] (D) at (7.5,2);
							\end{tikzpicture} 
							\begin{tikzpicture}[scale=0.28]
								\draw (0,0) -- (-5,5) -- (0,9) -- (7,2)-- cycle;
								\draw (-5,5) -- (7,2);
								\coordinate[label=$A$] (A) at (-1,-1);
								\coordinate[label=$D$] (B) at (-5,5);
								\coordinate[label=$C$] (C) at (0,9);
								\coordinate[label=$B$] (D) at (7.5,2);
							\end{tikzpicture}
							\begin{tikzpicture}[scale=0.28]
								\draw (0,0) -- (-5,-3) -- (0,7) -- (7,-2)-- cycle;
								\draw (0,0) -- (0,7);
								\coordinate[label=$A$] (A) at (0,-2);
								\coordinate[label=$D$] (B) at (-5.5,-3.3);
								\coordinate[label=$C$] (C) at (0,7);
								\coordinate[label=$B$] (D) at (7.6,-2.2);
							\end{tikzpicture}
							
							From left to right  a hinge $([ABCD],[AC])$, its flipping $([ABCD],[DB])$ and a non convex hinge.
					\end{figure}
					\begin{defi}[Weighted hinge]
						A weighted hinge is the datum of a hinge,  $Q=([ABCD],[AC])$, and 
						a function $\tau:\{A,B,C,D\}\rightarrow\RR$.
					\end{defi}
					\begin{defi}[$\tau$-legal/$\tau$-critical hinge] Let $(Q,\tau)$ is be weighted hinge. 
					Denote by $\widetilde\tau_{\tau,Q} : Q \rightarrow \RR$ the distance-like function induces by the triangulation $\TT=([ABC],[ADC])$.
					A hinge $Q$ is $\tau$-legal (resp. $\tau$-critical, resp. $\tau$-illegal) if $\widetilde \tau_{\tau,Q}$ is Q-convex (resp. $\mathscr C^1$, resp. strictly Q-concave) 					\end{defi}

					Each edge $e$ of a given triangulation $\TT$ provides a hinge, indeed $e$ bounds two triangles $T_1,T_2\in \TT$ and the gluing of this two triangles along $e$ is a hinge. Beware that two such  triangles might be actually the same in $\TT$ (a triangle glued to itself) but we take two copies to construct the hinge. More generally, we will need to consider immersed hinges.
					\begin{defi}\label{defi:immersed_hinge} An immersed hinge is the datum of a hinge $Q$ in  $\E^2$ and an isometric
						immersion $\eta : Q \rightarrow \Sigma$.		
						An immersed hinge $(Q,\eta)$ is embedded if $\eta$ is an embedding.
					\end{defi}
					The hinge associated to an edge is embedded if and only if the triangles bounded by $e$ are different in $\TT$. 
					
					After an analysis of critera ensuring $\tau$-legality of a given hinge, we notice the set of $\tau$ for which a given hinge   is $\tau$-legal is the set of solutions of an affine inequality hence a convex set. Then we turn to the whole surface and try to construct triangulations for which all hinges are $\tau$-legal for a given $\tau$. 
					
					\begin{defi}[$\tau$-Delaunay triangulation]
						Let $\TT$ be an adpated triangulation of $\Sigma$.

						The triangulation $\TT$ is $\tau$-Delaunay if the following equivalent properties are satisfied:
						\begin{enumerate}[(i)]
							\item $\widetilde \tau_{\tau,\TT}$ is Q-convex;
							\item every hinge of $\TT$ is $\tau$-legal.
						\end{enumerate}
						
					\end{defi}
					For a given triangulation $\TT$, the set of $\tau\in\RR_ +^S$ such that $\TT$ is $\tau$-Delaunay is the set solutions of a system of affine inequality hence a convex set (hence part $(d)$ of Theorem \ref{theo:domain_description}). However, $\PSigmaS$ is a possibly infinite union of such domains, therefore convexity is not a direct Corollary. We thus reverse the problem and construct a $\tau$-Delaunay triangulation  $\tau$ given a priori.
					
					The definition is of $\tau$-Delaunay triangulation is coherent with the usual definition of Delaunay triangulation. Indeed, an adapted triangulation of $(\Sigma,S)$ is a subtriangulation of the Delaunay cellulation if and only if it is $0$-Delaunay. 
					The Delaunay cellulation can either be constructed as the dual of the Voronoi cellulation (see for instance of a thorough exposition \cite{MR1135877}) or via a flipping algorithm starting from a given adapted triangulation. The flipping algorithm is based upon the following remark (Lemma \ref{lem:inverse_Qstar}): for a $\tau$ given, if a hinge is $\tau$-illegal then its flipping (if it exists) is $\tau$-legal. The algorithm then proceeds by flipping $\tau$-illegal hinges one by one in the hope that after finitely many iteration there won't be any $\tau$-illegal hinges left. Proposition \ref{prop:algo_stop} ensures the algorithm  behaves mostly as expected: it stops after finitely many iterations on a triangulation without any flippable $\tau$-illegal hinges. To complete the analysis of the flipping algorithm, we show the resulting triangulation is $\tau$-Delaunay if and only if there exists such a triangulation. 
					
					We end the section applying the results obtained on the flipping algorithm to prove Theorem \ref{theo:domain_description}.

					\subsection{Q-convexity on hinges}\label{subsec:Qconvexhinge}

						Before going any further, we notice that the group $\isom(\E^2)$ acts naturally on weighted hinges and preserves legality.

						\label{sec:critere_Qconvexe}

						 In this subsection, we give ourselves a hinge
						 $Q=([ABCD],[AC])$ and  weights $\tau$. For simplicity sake, we  choose a cartesian coordinated system $(x,y)$ of $\E^2$, set $A=O$ the origin of this coordinate system and put $C$ on the vertical axis above $A$. Write $\omega$ and $\tau_0$ (resp. $\omega'$ and $\tau_0'$) the parameters given 
						 by Lemma \ref{lem:prolongement} on $[ABC]$ (resp. $[ADC]$) for the weights $\tau$;
						 define $$\fonction{\tau_{ABC}}{\E^2}{\RR}{x}{t_0-\d(x,\omega)^2} \quad \quad \fonction{\tau_{ADC}}{\E^2}{\RR}{x}{t_0'-\d(x,\omega')^2}.$$
						Note that $\d(\omega,C)^2-\d(\omega,A)^2 = \d(\omega',C)^2-\d(\omega',A)^2  $
						hence $\overrightarrow{\omega\omega'} \perp \overrightarrow{AC}$,					
						the following picture sums-up the situation.
						\begin{center}
							\begin{tikzpicture}[scale=0.4]
								\draw (0,0) -- (-5,5) -- (0,9) -- cycle;
								\draw  (0,0) -- (7,2) -- (0,9) -- cycle;
								\coordinate[label=$A$] (A) at (-1,-1);
								\coordinate[label=$D$] (B) at (-5,5);
								\coordinate[label=$C$] (C) at (0,9);
								\coordinate[label=$B$] (D) at (7.5,2);
								\draw [ style=dashed] (-10,4) -- (10,4);
								\draw [line width = 0.5pt](0.5,4) -- (0.5,4.5) -- (0,4.5);
								\draw (-2,4) node[above]{$\omega'$} ;
								\draw (-2,4) node{$\bullet$} ;
								\draw (6,4) node[above]{$\omega$} ;
								\draw (6,4) node{$\bullet$} ;
							\end{tikzpicture}
						\end{center}

						\begin{prop}[Q-convexity criteria]\label{prop:Qconv_critere}
							Under this subsection hypotheses, the following are equivalent:
							\begin{enumerate}[(i)]
							 \item $\widetilde \tau_{Q,\tau}$ is Q-convex;
							 \item $\widetilde \tau_{Q,\tau}$ is Q-convex along some segment crossing $[AC]$;
							 \item  $\tau_{ABC}\leq \tau_{ACD}$ on $[ACD]$ and $\tau_{ABC}(B)\geq \tau_{ACD}$ on $[ABC]$.
							 \item $\tau_{ABC}(D)\leq \tau_{ACD}(D)$ or $\tau_{ABC}(B)\geq \tau_{ACD}(B)$.
							 
							 \item $x_{\omega}\geq x_{\omega'}$
							 \item
							 $$\left(\frac{y_B}{|x_B|}+\frac{y_D}{|x_D|}\right)\tau_C+\left(\frac{AC-y_B}{|x_B|}+\frac{AC-y_D}{|x_D|}\right)\tau_A\leq \frac{AC}{|x_D|}\tau_D+\frac{AC}{|x_B|}\tau_B+K$$
							 with
							 $$K=\frac{AC}{|x_B|}\left(AB^2-ACy_B\right)+\frac{AC}{|x_D|}\left(AD^2-ACy_D\right).$$
							\item Denoting by $\overrightarrow u\wedge \overrightarrow v$ the determinant $|\overrightarrow u \overrightarrow v|$ :
							 $$\left(\overrightarrow{AB}\wedge\overrightarrow{AD}\right) \tau_C+\left(\overrightarrow{CD} \wedge\overrightarrow{CB}\right)\tau_A- \left(\overrightarrow{CA}\wedge \overrightarrow{CB}\right)\tau_D-\left(\overrightarrow{AC}\wedge \overrightarrow{AD}\right)\tau_B-K\leq 0$$
							 with
							 $$K=
							 \overrightarrow{AC}\wedge \overrightarrow{AD}
								\left(\overrightarrow{AB}\cdot \overrightarrow{CB}\right)
							 +
							 \overrightarrow{CA}\wedge \overrightarrow{CB}
								\left(\overrightarrow{AD}\cdot \overrightarrow{CD}\right)
							 .$$
							\end{enumerate}
						\end{prop}
						\begin{proof}

						 \begin{itemize}
							\item $(i)\Rightarrow (ii)$ is by definition.
							\item $(ii)\Rightarrow (i)$. Since the line $(\omega\omega')$ is normal to $(AC)$ it follows that $\frac{\partial\tau_{ABC}}{\partial y} = \frac{\partial\tau_{ACD}}{\partial y} $, then $\overrightarrow\grad\, \tau_{[ABC]}-\overrightarrow\grad \tau_{[ACD]}$ is horizontal and the sign of $\langle \overrightarrow\grad\,\tau_{[ABC]}-\overrightarrow\grad \tau_{[ACD]}| \overrightarrow u\rangle$ does not depend on $\overrightarrow u$ as long a $\overrightarrow u$ is directed toward increasing $x$. 
							\item $(v)$ is equivalent to $\frac{\partial\tau_{ABC}}{\partial x} \geq \frac{\partial\tau_{ACD}}{\partial x}$ which is equivalent to Q-convexity along the direction normal to $[AC]$. Hence, $(i)\Rightarrow (v)$ and $(v)\Rightarrow (ii)$.
							\item $(i)\Rightarrow (iii)$. Let $P\in [ABC]$, choose some $P'\in [ACD]$ such that $[P'P]$ crosses $[AC]$. The function $\tau_{[ACD]}$ is $\mathscr C^1$ while  $\widetilde \tau_{Q,\tau}$ is Q-convex along $[P'P]$. The same argument as in the proof of  Lemma  \ref{lem:diff_Q_convexe} gives the first inequality. The second is proven the same way. 
							\item $(iii)\Rightarrow(iv)$ is trivial.
							
							\item To prove $(iv)\Rightarrow (ii)$, consider any segment $[PB]$ with $P\in [ACD]$. Along such a segment,  $\widetilde \tau_{Q,\tau}$ is either Q-convex or  strictly Q-concave. The inequality $\tau_{ABC}(B)\geq \tau_{ACD}(B)$ implies it is the former. The same argument shows $\tau_{ABC}(D)\leq \tau_{ACD}(D) \Rightarrow (ii)$.
							
						  \item $(v)\Leftrightarrow (vi)$ Solve explicitely the system in the proof of Lemma \ref{lem:prolongement}  for both sides in $(v)$.
						  \item $(vii)$ is a geometric rewriting  of $(vi)$.
						 \end{itemize}

						\end{proof}

						The previous Proposition shows Q-convexity is an affine constraint on $\tau$ for a given hinge. Since we will have to consider multiple hinges for multiple triangulations, we introduce the following.
						\begin{defi}[Affine form of a hinge] \label{defi:Q_forme_affine}

							Let $Q=([ABCD],[AC])$ be  a hinge, define the affine form associated to $Q$ by:
							$$\fonction{Q^*}{\RR_+^{\{A,B,C,D\}}}{\RR}{\tau}{
								  \lambda_C\tau_C+
								  \lambda_A\tau_A
									-\lambda_D\tau_D-
									\lambda_B\tau_B-K
								 }
							$$
							where
							\begin{eqnarray*}
							 \lambda_C=\overrightarrow{AB}\wedge\overrightarrow{AD},&&\lambda_A=\overrightarrow{CD} \wedge\overrightarrow{CB},\\
							 \lambda_D=\overrightarrow{CA}\wedge \overrightarrow{CB},&&\lambda_B=\overrightarrow{AC}\wedge \overrightarrow{AD},
							\end{eqnarray*}
							  $$K= \overrightarrow{AC}\wedge \overrightarrow{AD}
								\left(\overrightarrow{AB}\cdot \overrightarrow{CB}\right)
							 +
							 \overrightarrow{CA}\wedge \overrightarrow{CB}
								\left(\overrightarrow{AD}\cdot \overrightarrow{CD}\right).$$

						\end{defi}
						\begin{rem}\label{rem:forme_linaire_Qconvexe} The affine form $Q^*$ is defined  in such a way that $\widetilde \tau_{Q,\tau}$
						is Q-convex if and only if $Q^*(\tau)\leq 0$.

						\end{rem}

						\begin{rem}
						 If $(Q,\eta)$ is an immersed hinge of $(\Sigma,S)$ with $\eta$ sending vertices into $S$
						 and with $Q=([ABCD],[AC])$,
						 we can then define a corresponding affine form  $\RR_+^{S}\rightarrow \RR$
						  $$\fonctionn{\RR_+^S}{\RR}{\tau}{Q^*(\tau\circ \eta_{|\{A,B,C,D\} }) }.$$
						 If there is no ambiguity we shall also denote it by $Q^*$.
						\end{rem}
						\begin{rem} A hinge $Q$ is $\tau$-critical if and only if $Q^*(\tau)=0$.

						\end{rem}

						\begin{lem}\label{lem:inverse_Qstar}
				 Let $Q=([ABCD],[AC])$ be a flippable hinge and let $Q'$ its flipped hinge. 
				 As functions $\RR^{\{A,B,C,D\}}\rightarrow \RR$ we have:
				 
				 $$Q'^*=-Q^* $$
				\end{lem}
				\begin{proof}
					Following the notations of definition \ref{defi:Q_forme_affine} we write : 
					$$\fonction{Q^*}{\RR_+^{\{A,B,C,D\}}}{\RR}{\tau}{
								  \lambda_C\tau_C+
								  \lambda_A\tau_A
									-\lambda_D\tau_D-
									\lambda_B\tau_B-K
								 }$$ $$
					\fonction{Q'^{*}}{\RR_+^{\{A,B,C,D\}}}{\RR}{\tau}{
								  \lambda_C'\tau_C+
								  \lambda_A'\tau_A
									-\lambda_D'\tau_D-
									\lambda_B'\tau_B-K'
								 }
							$$
				where
							\begin{eqnarray*}
							 \lambda_C=\overrightarrow{AB}\wedge\overrightarrow{AD},&&\lambda_A=\overrightarrow{CD} \wedge\overrightarrow{CB},\\
							 \lambda_D=\overrightarrow{CA}\wedge \overrightarrow{CB},&&\lambda_B=\overrightarrow{AC}\wedge \overrightarrow{AD},\\
							 \lambda_D'=-\overrightarrow{BC}\wedge\overrightarrow{BA},&&\lambda_B'=-\overrightarrow{DA} \wedge\overrightarrow{DC},\\
							 \lambda_A'=-\overrightarrow{DB}\wedge \overrightarrow{DC},&&\lambda_C'=-\overrightarrow{BD}\wedge \overrightarrow{BA},
							\end{eqnarray*}
						We check that
					 $$ \lambda_A'=-\overrightarrow{DB}\wedge \overrightarrow{DC}
					  = -(\overrightarrow{DC}+\overrightarrow{CB})\wedge\overrightarrow{DC}
					  =-\overrightarrow{CB}\wedge \overrightarrow{DC} = -\overrightarrow{CD}\wedge \overrightarrow{CB}
					  = -\lambda_A$$
					 and we check the same way that $\lambda_B'=-\lambda_B$, $\lambda_C'=-\lambda_C$ and $\lambda_D'=-\lambda_D$.

					A quick way to prove that $K' = -K$ is to notice that 
					$$K = (AB\cdot CB\cdot CD \cdot DA) \sin\left(\widehat{BAD}+\widehat{DCB}\right) $$
					$$K' = (AB\cdot CB\cdot CD \cdot DA) \sin\left(\widehat{CBA}+\widehat{ADC}\right) $$
					and that $\widehat{BAD}+\widehat{DCB}+\widehat{CBA}+\widehat{ADC} = 0\, \mathrm{mod}\,  2\pi$
				
				\end{proof}

						\begin{cor}
						 Let $(Q,\tau)$ be a weighted flippable hinge,
						 then $Q$ is $\tau$-critical if and only if its flipping $Q'$ is $\tau$-critical.
						\end{cor}

						\begin{cor} \label{cor:inversion_Qconvexe} Let $(Q,\tau)$ be a weighted flippable hinge and $Q'$ the flip of $Q$.
						
						If $Q$ is not $\tau$-critical, then the following are equivalent:
						\begin{enumerate}
						 \item $Q$ is $\tau$-legal;
						 \item $Q'$ is $\tau$-illegal.
						\end{enumerate}
						\end{cor}

						\begin{lem}\label{lem:kernel_hinge} For any hinge $Q$, 
						  the indicator function $\mathbf{1}_S$ is in the kernel of the linear part of $Q^*$, eg 
						  $$\forall \tau\in \RR^S, \forall \lambda\in \RR, Q^*(\tau+\lambda\mathbf{1}_S)=Q^*(\tau).$$
						\end{lem}
						\begin{proof}
						 Using notations of Definition \ref{defi:Q_forme_affine}, we have
						 \begin{eqnarray*}
						  \lambda_A+\lambda_C-\lambda_B-\lambda_D &=& \overrightarrow{CD}\wedge \overrightarrow{CB} +\overrightarrow{AB}\wedge\overrightarrow{AD} - \overrightarrow{AC}\wedge \overrightarrow{AD}- \overrightarrow{CA}\wedge \overrightarrow{CB}  \\
						  &=& \overrightarrow{AD}\wedge \overrightarrow{CB}+ \overrightarrow{CB}\wedge\overrightarrow{AD} \\
						  &=&0.
						 \end{eqnarray*}

						\end{proof}

						\begin{cor}\label{cor:translation}
							For all $\tau \in \calP_{\Sigma}$ and all $\lambda \in \RR$, 
							$$\tau+\lambda \mathbf{1}_S\in \PSigmaS \quad \Leftrightarrow \quad \tau+\lambda \mathbf{1}_S\geq 0.$$
						\end{cor}
						\begin{cor}\label{cor:theo3_b}
						With the notations of Theorem \ref{theo:domain_description},
							$$\PSigmaS = (\overline{\PSigmaS} + \RR\cdot\mathbf{1}_S)\cap \RR_+^S$$
						\end{cor}

					\subsection{Flipping algorithm}
										
						Let $\TT$ be an adapted triangulation of $(\Sigma,S)$. Consider $(Q,\eta)$ an immersed hinge  given by an edge of $\TT$.
						We would like to flip $(Q,\eta)$ ie construct a new triangulation of $(\Sigma,S)$ with $\eta(Q)$ replaced by $\eta(Q')$ with $Q'$ the flip of $Q$.
 						There are four cases: 
						\begin{itemize}
						 \item if $\eta$ is not an embedding then the diagonal one wants to replace is also a side of the hinge. Hence, one cannot simply replace it without modifying the triangulation $\TT$ elsewhere;
						 \item if $\eta$ is embedded but $Q$ is not flippable.

						 \item if $\eta$ is embedded and $Q$ is flippable, then the flipped hinge $Q'$ is well defined, $\eta : Q' \rightarrow \Sigma$ is well defined, $\eta(Q')= \eta(Q)$ so that we only modify $\TT$ locally and the new triangulation $\TT'$ is composed of non degenerated triangles.
						\end{itemize}

						This remark  motivates the following definitions.

						\begin{defi}[Flippable immersed hinge]\label{defi:immersed_hinge_flippable}
						An immersed hinge $(Q,\eta)$ is  flippable if it is embedded and $Q$ is flippable; it is unflippable otherwise.
						\end{defi}

						\begin{defi}[Flipping algorithm]
							Let  $\TT_0$ be any adapted triangulation of  $(\Sigma,S)$ and let  $\tau : S \rightarrow \RR_+$.
							The flipping algorithm proceeds as follows:
								\begin{enumerate}
								\item Set $i=0$.
								\item Make a list $L$ of $\tau$-illegal  flippable immersed hinge $(Q,\eta)$ induced by the edges of the current triangulation  $\TT_i$.
								\item If $L$ is non empty,
									\begin{enumerate}
										\item Choose some immersed hinge $(Q,\eta)$ in $L$.
										\item Replace the hinge $(Q,\eta)$ by its flip $(Q',\eta)$ in $\TT_i$ to obtain a new
										triangulation $\TT_{i+1}$.
										\item increment $i$ and go to step $(2)$.
									\end{enumerate}
								\item If $L$ is empty, the algorithm stops and returns $(\TT_j)_{j\leq i}$.
							 \end{enumerate}
						\end{defi}
						The goal of the section is to prove the following.
						\begin{prop} \label{prop:algo_stop}
						  Let $\tau : S\rightarrow \RR_+$.
						  For any starting triangulation $\TT_0$, the flipping algorithm for $\tau$ starting at $\TT_0$
						  stops on some triangulation $\TT_\tau$ after finitely many iterations and  every flippable immersed hinge in $\TT_\tau$ is $\tau$-legal. 
						  Furthermore, 
						  
						  \begin{itemize}
						   \item $\tau\in \PSigmaS$ if and only if $\TT_\tau$ is $\tau$-Delaunay;
						   \item $\displaystyle\max_\Sigma\widetilde \tau_{\tau, \TT_\tau}\leq \max_S \tau + \max_\Sigma \widetilde \tau_{0,\TT_0}$.
						  \end{itemize}
						\end{prop}
						\begin{rem} The notation $\TT_\tau$ of this last Proposition is coherent with the one introduced in Definition \ref{defi:temps_admissbles}.
						 
						\end{rem}

						Two Lemmas are key to the proof, the first is Lemma \ref{lem:decroissante} which states that $\widetilde \tau_{\tau, \TT_i}$ is decreasing along the iterations of the algorithm; the second is Lemma \ref{lem:arrete_ext} which implies that unflippable hinges are always $\tau$-legal for $\tau\in \PSigmaS$. Lemma \ref{lem:arrete_ext} will be again useful in the following section.
						
						\begin{lem}\label{lem:decroissante}
						 Let $\tau : S \rightarrow \RR_+$ and let $\TT_0$ be a an adapted triangulation. Let $(\TT_i)_{i\in I}$ be the sequence of triangulation given by 
						 the flipping algorithm with weights $\tau$ and starting at $\TT_0$, where
						 $I=\lsem 0,n\rsem$ or
						 $\NN$.

						 Then, the associated sequence of distance-like functions $(\widetilde \tau_{\tau,\TT_i})_{i\in I}$ is decreasing :
						 \begin{itemize}
						  \item for all $i,j\in I$ with $i\leq j$ then  $\widetilde \tau_{\tau,\TT_i}\geq \widetilde \tau_{\tau,\TT_j}$;
						  \item for all $i,j\in I$ with $i<j$ there exists $x\in \Sigma_1$ such that $$\widetilde \tau_{\tau,\TT_i}(x)> \widetilde \tau_{\tau,\TT_j}(x).$$
						 \end{itemize}
						\end{lem}
						\begin{proof} Let $i\in I$ such that $i+1\in I$.
						The triangulation $\TT_{i+1}$ is obtained from $\TT_i$ by flipping an embedded hinge, say $(Q,\eta)$ of $\TT_i$ with $Q=([ABCD],[AC])$.
						Then
						\begin{itemize}
						 \item $\forall x\in \Sigma_1\setminus \eta(Int(Q)),~~ \widetilde \tau_{\tau,\TT_i}(x)=\widetilde \tau_{\tau,\TT_{i+1}}(x)$.
						Indeed, for $x\notin \eta(Q)$, the triangle containing $x$ is the same in $\TT_i$ and $\TT_{i+1}$.

						 \item $\forall x\in \eta(Int(Q)),~~ \widetilde \tau_{\tau,\TT_i}(x)>\widetilde \tau_{\tau,\TT_{i+1}}(x)$.
								Indeed, $\widetilde \tau_{\tau,Q}$ and $\widetilde \tau_{\tau,Q'}$   are equal on $[AB],[BC],[CD]$ and $[DA]$ ;  by hypothesis
								$\widetilde \tau_{\tau,Q}$ is strictly Q-concave and, from Lemma \ref{cor:inversion_Qconvexe}, $\widetilde \tau_{Q',\tau}$ is strictly Q-convex. Applying Lemma \ref{lem:diff_Q_convexe} on segments going from side to side of $[ABCD]$
								we obtain $x\in 
								Q$, $\widetilde \tau_{\tau,Q} > \widetilde \tau_{\tau,Q'}$ on $Int(Q)$.
						\end{itemize}
						\end{proof}
						\begin{cor}\label{cor:flip_injective}
							The sequence $(\TT_i)_{i\in I}$ given by the flipping algorithm is injective.
						\end{cor}

						\begin{lem}\label{lem:max_quad}
						Let $\widetilde \tau$ be a non-negative distance-like function on $(\Sigma,S)$, if $\widetilde \tau$ is $\mathscr C^1$ on some geodesic of length $\ell$ then $$\max\widetilde \tau \geq \frac{\ell^2}{4}$$
					
						\end{lem}

						\begin{proof}

						Let $c:[a,b]\rightarrow \Sigma$ be an arc length parametrization of such a geodesic and let $f:=\widetilde\tau\circ c$, we have  $$\fonction{f}{[a,b]}{\RR}{x}{-x^2+\alpha x+ \beta}$$
						for some $\alpha,\beta\in \RR$. 
						Furthermore,$\widetilde\tau\geq 0$ thus  that $f(a)\geq 0$ et $g(b)\geq 0$.
					
						 Define $u:[a,b]\rightarrow \RR$ the unique affine function such that $u(a)=f(a)$ and $u(b)=b$. We thus have $\forall x\in [a,b], f(x)=u(x)-(x-a)(x-b)$.  
						 On the one hand,  $f(a)$ and $f(b)$ are non-negative so $u$ is non-negative.
						 On the other hand, $$\max_{x\in[a,b]}(-(x-a)(x-b))=\frac{(b-a)^2}{4}=\frac{\ell^2}{4}.$$ The result follows.

						\end{proof}

						\begin{lem} \label{lem:triangulation_finie}
						For $C\in \RR_+^*$, let $E_C$ the set of adapted triangulations $\TT$  of $(\Sigma,S)$ such that $$\exists \tau\in \RR_+^S, \max\widetilde \tau_{\tau,\TT}\leq C.$$
						Then, $E_C$ is finite.
						\end{lem}
						\begin{proof}
						 Let $\TT$ be an adapted triangulation such that such that $\exists \tau\in \RR_+^S, \max\widetilde \tau_{\tau,\TT}\leq C$. Choose such a $\tau$. 
						 Let $e$ the longuest edge of $\TT$, the restriction of $\widetilde \tau_{\tau,\TT}$ to $e$ is  non-negative distance-like and $\mathscr C^1$.
						 From Lemma \ref{lem:max_quad} with $L=\mathrm{length}(e)$
						 $$ \frac{L^2}{4} \leq\max_{[a,b]}\widetilde \tau_{\tau,\TT} \leq C $$
						 thus $L\leq 2\sqrt{C}$. Therefore, the triangulation $\TT$ only has edges of length less than $2\sqrt{C}$.
						 
						 Consider a finite covering $\widehat\Sigma$ of $\Sigma$ ramified above $S$ such that all cone angles of $\widehat\Sigma$ are bigger than $2\pi$. The universal (unbranched) covering $\widetilde{\Sigma}$ of $\widehat \Sigma$ is Cat(0) and for any two points above $S$, there exists at most one (unbroken) geodesic. Furthermore, to any (unbroken) geodesic of length at most $2\sqrt{C}$ in $\Sigma$ for a point $A$ of $S$ to a point $B$ of  $S$ one can associate an (unbroken) geodesic in $\widetilde \Sigma$ of same length starting from a a priori fixed $\hat A$ to some unfixed lift $\hat B$ of $B$ in the ball of radius $2\sqrt L$ around $\hat A$. 
						 There are finitely many such $\hat B\in\widetilde \Sigma$, thus finitely such geodesics. 
						 There are thus only finitely many geodesics of $\Sigma$ from $S$ to $S$ of length bounded by $2\sqrt{C}$, hence there are only finitely many triangulation with edges of length at most $2\sqrt{C}$. 
						\end{proof}

						\begin{lem} \label{lem:arrete_ext}
							Let $Q$ be an unflippable hinge with $Q=([ABCD],[AC])$.
							If there exists some distance like Q-convex function  $f$ on $[ABCD]$ extending $\tau : \{A,B,C,D\} \rightarrow \RR$ then $Q$ is $\tau$-legal.
						\end{lem}
						\begin{rem} Beware that the triangulation of $[ABCD]$ adapted to the distance-like function $f$ be different from $([ABC],[ACD])$. 
						 
						\end{rem}

						\begin{proof}\
							Without loss of generality, we may assume that $C$ is in the convex hull of 
							$[ABCD]$.
							Define $g:=\widetilde \tau_{Q,\tau}$ and $h$  the 
							distance-like extension of $\tau_{|\{A,B,D\}}$ on $[ABD]$ given by Lemma \ref{lem:prolongement}.
							Both fonctions $f$ and $g$ are defined on $[ABCD]\subset [ABD]$ and  $h$ is defined on $[ABD]$.

							$h$ is $\mathscr C^1$ on $[ABD]$ while $f$ is
							Q-convex on $[ABCD]$. By Lemma 
							\ref{lem:diff_Q_convexe} applied on the edges $[AB],[AC]$ and $[AD]$ then on the edges $[xy]\subset [ABCD]$ with $x\in [AB]$ and $y\in [AD]$ we have $f\leq h$.
							We show the same way that if $g$ is Q-concave then $g\geq h$, thus $g\leq h $ implies that $g$ is Q-convex.

							Furthermore, $g-h$ is affine on each triangles $[ACB]$ and $[ACD]$ and null at $A,B$ et $D$. Therefore,
							$g-h$ is non positive	if and only if $g(C)-h(C)\leq 0$.

							Since, $g(C)=f(C)\leq h(C)$, we deduce that $g\leq h$ hence
							$g$ is Q-convex.
							Finally, $\widetilde \tau_{Q,\tau}$ is Q-convex.

						\end{proof}
						\begin{lem}  \label{lem:arrete_deg} Let $\tau:S\rightarrow \RR_+$ and 
						 let $(Q,\eta)$ be an immersed hinge with $Q=([AB_{-1}B_0B_{1}],[AB_0])$ such that $[AB_0B_{1}]$ is obtained from $[AB_{-1}B_0]$ via  a rotation and  $\eta([AB_1])=\eta([AB_{-1}])$ and $\eta(B_0)=\eta(B_1)=\eta(B_{-1})$.
						 
						Then, there exists an immersed hinge $(\hat Q,\hat\eta)$ with $\hat Q$ unflippable such that 
						$(Q,\eta)$ is $\tau$-legal if and only if $(\hat Q,\hat\eta)$ is $\tau$-legal.				
						\end{lem}
						\begin{proof}  Let $\theta = \widehat{B_0AB_1}$ and $n=\lceil\pi/\theta\rceil-1$,
						if $\theta\geq \pi/2$ then take $\hat Q=Q$ and $\hat\eta=\eta$.
						
						 Otherwise, construct the polygon $[AB_{-n}\cdots B_0\cdots B_{n}]$, such that for each $k\in\lsem 1-n,n\rsem$, the triangle $[AB_kB_{k+1}]$  is obtained from $[AB_0B_{1}]$ via the rotation of center $A$ and  angle $\alpha_k=k\theta$. Define $\hat Q:=[AB_nB_0B_{-n}]$ and  $$\fonctiondeux{\hat\eta}{\hat Q}{\Sigma}{x\in [AB_{2k}B_{2k+1}]}{\eta(\rho_{-2\alpha_{k}}(x))}{x\in [AB_{2k-1}B_{2k}]}{\eta(\rho_{-2\alpha_{k}}(x))}$$ 
						 where $\rho_\beta$ denote the rotation of center $A$ and angle $\beta$.

						\def\sq{0.707106781}
						 \begin{center}
							\begin{tikzpicture}[scale=3]
								\draw (0,0) -- (-\sq,\sq) -- (0,1) -- cycle;
								\coordinate[label=$A$] (A) at (0,-0.2);
								\coordinate[label=$B_0$] (B) at (0,1);
								\coordinate[label=$B_{-1}$] (B) at (\sq,\sq);
								\coordinate[label=$B_{1}$] (B) at (-\sq,\sq);
								\coordinate[label=$B_{-2}$] (B) at (1.07,0);
								\coordinate[label=$B_{-3}$] (B) at (\sq,-\sq-0.2);
								\coordinate[label=$B_{2}$] (B) at (-1.1,0);
								\coordinate[label=$B_{3}$] (B) at (-\sq,-\sq-0.2);
								\draw[black,cm={cos(45) ,-sin(45) ,sin(45) ,cos(45) ,(0 cm,0 cm)}] (0,0) -- (-\sq,\sq) -- (0,1) -- cycle;
								\draw[black,cm={cos(90) ,-sin(90) ,sin(90) ,cos(90) ,(0 cm,0 cm)}] (0,0) -- (-\sq,\sq) -- (0,1) -- cycle;
								\draw[black,cm={cos(90) ,-sin(90) ,-sin(90) ,cos(90) ,(0 cm,0 cm)}] (0,0) -- (-\sq,\sq) -- (0,1) -- cycle;
								\draw[black,cm={cos(135) ,-sin(135) ,sin(135) ,cos(135) ,(0 cm,0 cm)}] (0,0) -- (-\sq,\sq) -- (0,1) -- cycle;
								\draw[black,cm={cos(90) ,-sin(-90) ,sin(-90) ,cos(90) ,(0 cm,0 cm)}] (0,0) -- (-\sq,\sq) -- (0,1) -- cycle;
								\draw[black,dashed] (\sq,-\sq) -- (0,1) -- (-\sq,-\sq);
							\end{tikzpicture}
						\end{center}
						We have $\forall k\in \lsem -n,n\rsem, \hat\eta(B_k)=\eta(B_0)$ the weights $\tau \in \RR^S$ thus induce weights $\hat\tau:=\tau\circ \hat\eta$ such that $\hat\tau(B_{k})=\tau(B_0)=\hat\tau(B_1)=\hat\tau(B_{-1})$.
						For $I,J,K \in \{A,B_{-n},\cdots,B_n\}$, denote by $\omega_{[IJK]}$ the center of $\widetilde \tau_{\tau,[IJK]}$ and by $\omega_{[IJ]}$ the orthogonal projection of $\omega_{[IJK]}$ on the line $(IJ)$. Note that $\omega_{[IJ]}$ does not depend on $K$.  
						
						Since $\hat\tau(B_n)=\hat\tau(B_0)=\hat\tau(B_{-1})$, then $\omega_{[B_0B_n]}$ (resp. $\omega_{[B_0B_{-n}]}$) is the middle of $[B_0B_n]$ (resp. of $[B_0B_{-n}]$). Since the lengths $(AB_k)_{k\in \lsem -n,n\rsem}$ are equal, the segment bissectors of $[B_0B_{-n}]$ and $[B_0B_{n}]$ intersect at $A$, therefore $\omega_{AB_{-n}B_0}$ on the right of $\omega_{AB_{n}B_0}$ on the normal to $(AB_0)$ at $\omega_{[AB_0]}$ if and only if $\omega_{[AB_0]}$
						 is on the ray $[AB_0)$. Hence, by Proposition \ref{prop:Qconv_critere}.(v), $\hat Q$ is $\hat\tau$-legal if and only if $\omega_{[AB_0]}$ is on the ray $[AB_0)$.
						
						The same argument shows $Q$ is $\hat\tau$-legal if and only if $\omega_{[AB_0]}$ is on the ray $[AB_0)$. Finally, 
						$(Q,\eta)$ is $\tau$-legal if and only if $(\hat Q,\hat \eta)$  is $\tau$-legal.

						\end{proof}

						\begin{proof}[Proof of Proposition \ref{prop:algo_stop}]
							By Corollary \ref{lem:decroissante}, the sequence on of distance-like functions given by the flipping algorithm is bounded above by the first of the sequence $\widetilde \tau_{\tau,\TT_0}$. Since $\widetilde\tau_{\tau,\TT_0}-\widetilde\tau_{0,\TT_0}$ is affine on each triangle of $\TT_0$ is is bounded by its value on $S$ thus by $\max_S \tau$.  
							Hence, $\forall i, \widetilde \tau_{\tau,\TT_i}\leq \max_S{\tau}+\max\widetilde \tau_{0,\TT_0}$.
							By Lemma \ref{lem:triangulation_finie}, the flipping algorithm runs through a finite set of triangulation. Finally, by Corollary \ref{cor:flip_injective}, the algorithm reach a given triangulation at most once and thus stops after finitely many steps, say $n\in \NN^*$. The algorithm stops when the set of flippable $\tau$-illegal hinges is empty, then $\TT_n$ has no flippable $\tau$-illegal hinges.
							
							If the final triangulation $\TT_n$ is $\tau$-Delaunay then by definition $\tau\in \PSigmaS$. Assume $\tau \in \PSigmaS$ and consider $(Q,\eta)$ some unflippable hinge of $\TT_n$. Either $\eta$ is an embedding, in which case the weighted hinge $(Q,\tau\circ \eta)$ satisfies the hypotheses of Lemma \ref{lem:arrete_ext} and the immersed hinge $(Q,\eta)$ is then $\tau$-legal; or $\eta$ is not an embedding in which case $(Q,\eta)$ satisfies the hypotheses of Lemma \ref{lem:arrete_deg}, the immersed hinge $(\hat Q,\hat \eta)$ provided by Lemma \ref{lem:arrete_deg} satisfies the hypotheses of  Lemma \ref{lem:arrete_ext}, it is thus $\tau$-legal and so is $(Q,\eta)$. Finally, $\TT_n$ is $\tau$-Delaunay.
						\end{proof}

					\subsection{Description of the Domain of admissible times}

						We may interpret Lemma \ref{lem:arrete_ext} together with \ref{lem:arrete_deg} in the following way: if $\tau \in \PSigmaS$, then all unflippable immersed hinges of $(\Sigma,S)$ with vertices in $S$ are $\tau$-legal. Furthermore, Proposition \ref{prop:algo_stop} shows the converse: the flipping algorithm stops on a triangulation $\TT$ which flippable hinges are all $\tau$-legal, if all unflippable hinges of $(\Sigma,S)$ are $\tau$-legal, in particular those of $\TT$ are $\tau$-legal, hence $\TT$ is $\tau$-Delaunay.
						We thus proved the following

						\begin{prop}\label{prop:domaine_convexe}
						Let $\mathcal H$ the set of unflippable immersed hinge of $(\Sigma,S)$ with vertices in $S$. Then:
						   $$\PSigmaS = \bigcap_{(Q,\eta)\in \mathcal H} (Q^*)^{-1}(\RR_-)$$
						   in particular $\PSigmaS$ is a convex domain of $\RR_+^{S}$.
						\end{prop}
						\begin{rem}
						 Lemma \ref{lem:charniere_type_existence} belows implies that  $\mathcal H$ is non-empty. For now we can, take the convention that the intersection is $\RR_+^S$ if $\mathcal H =\emptyset$.
						\end{rem}

						\begin{prop}
						   For $\tau\in \PSigmaS$, if $\widetilde \tau$ is the unique Q-convex distance-like extension of $\tau$ to $(\Sigma,S)$ then
						   $$\widetilde \tau = \min_{\TT'} \widetilde \tau_{\TT',\tau} $$
						   where $\TT'$ runs through all adapted triangulations of $(\Sigma,S)$.
						\end{prop}
						\begin{proof}
						 Take any adapted triangulation $\TT$ of $(\Sigma,S)$ and consider $T$ a triangle of $\TT$. On $T$, $\widetilde \tau_{\tau,\TT}$ is $\mathscr C^1$  on $T$  while $\widetilde \tau$ is Q-convex on $T$. By Lemma \ref{lem:diff_Q_convexe},  $\widetilde \tau \leq \widetilde \tau_{\tau,\TT}$ on $T$. 
						 The triangle $T$ is arbitrary, thus  $\widetilde \tau \leq \widetilde \tau_{\tau,\TT}$ on $\Sigma$. 
						\end{proof}

				\begin{prop}\label{prop:delaunay_interieur}
				  The indicator function $\mathbf{1}_S$ of $S$ is in the interior of $\PSigmaS$.
				\label{prop:penint}
				\end{prop}
				\begin{proof}
					To begin with,  by Theorem 4.4 of \cite{MR1135877} each cell of the Delaunay cellulation $\C$ of $(\Sigma,S)$ is isometric to a polygon inscribed into a circle of $\E^2$ whose center is a vertex of the Voronoi cellulation which is in the interior of the cell. For any given cell $C$ of the Delaunay cellulation,  with $R_C$ the radius and $\omega_C\in C$ the center of its circumscribed circle, the function 
					$$\fonction{f}{C}{\RR_+}{x}{R_C^2-\d(x,\omega_C)^2}$$
					is distance-like $\mathscr C^1$ on $C$ and $f(p)=0$ for any vertex $p$ of $C$ hence for any subtriangulation $\TT$ of $\C$, $\forall x\in C, \widetilde \tau_{0,\TT}(x)=f(x)$. Let $e$ be an edge of the Delaunay cellulation, let $C,C'$ be the two cells on each sides of $e$, since $\omega_C,\omega_{C'}$ are each in the interior of their respective cell, in particular their developpement in $\E^2$ are disjoint and each on the same side of $e$ as their respective cell. Therefore, denoting by $Q^*_{e,\TT}$ the affine form associated with the hinge of axis $e$ for any subtriangulation $\TT$ of $\C$, we have $Q^*_{e,\TT}(0)<0$.
					Define $$\U:=\RR_+^S\cap \bigcap_{\TT}\bigcap_e {Q^*_{e,\TT}}^{-1}(\RR_-^*)$$
					where $\TT$ runs through the subtriangulations of the Delaunay cellulation and $e$ runs through the edges of the Delaunay cellulation. The intersection is finite since there are only finitely many subtriangulations and edges.
					$\U$ is thus an open subset of $\RR^S_+$ which contains $\RR_+ \mathbf{1}_S$.
					
					We now show $\U\subset \PSigmaS$. Apply the flipping algorithm to a substriangulation of the Delaunay cellulation for some $\tau\in \U$. The conditions $Q_{e,\TT}^*(\tau)<0$ ensure that the edges $e$ bording Delaunay Cells are always $\tau$-legal, hence the flipping algorithm  will never flip such edges and only runs through subtriangulations of the Delaunay cellulation. 
				From Proposition \ref{prop:algo_stop} the algorithm stops on some triangulation $\TT$, which is a subtriangulation of the Delaunay and all flippable hinges are $\tau$-legal. On the one hand, the edges of $\C$ are $\tau$-legal since $\tau\in \U$. On the other hand all hinges inside a cell of the Delaunay cellulation are flippable. Finally, all the edges of $\TT$ are $\tau$-legal. 
				$\U$ is thus a subset of $\PSigmaS$.

				\end{proof}

				In order obtain a finite cellulation of $\PSigmaS$ as well as caracterising its boundary we prove his transverse compactness.
				By transverse compactness of $\PSigmaS$ we mean that the projection of $\PSigmaS$ into the hyperplane $\{\tau\in \RR^S~|~\sum_{s\in S}\tau(s) = 0 \}$ is compact. Note that, for instance, if $\PSigmaS $ was equal to the whole $\RR_+^S$ then it wouldn't be transversaly compact in this sense.  The proof that $\PSigmaS$ is transversaly compact rely upon the construction of affine constraints  of the form $\tau_A-\tau_C\leq \varepsilon (\tau_A+\tau_B+\tau_C+\tau_D) + K$ with $\varepsilon>0$ arbitrary small and $A,C$ arbitrary in $S$.
				Such constraints are provided by type $(x,L)$ hinges, see Definition \ref{defi:charniere_type}, via Lemma \ref{lem:charniere_type_rel}. Lemma \ref{lem:charniere_type_existence} focuses on the construction of such immersed hinges.

				\begin{defi}[Type $(x,L)$ hinge]\label{defi:charniere_type}
					Let $x,L>0$, a hinge $([ABCD],[AC])$ of $\E^2$ is of type $(x,L)$ if it is non-convex with $C\in [ABD]$
					and \begin{eqnarray*}
					        \d(B,(AC))\leq x,&&  \d(D,(AC))\leq x\\
							AB>L,&& AD>L
					       \end{eqnarray*}

				\end{defi}

				\begin{lem}\label{lem:charniere_type_rel} Let $l>0$ and $x>0$.
				For all hinge $Q$, write
				$$Q^*:\tau \mapsto \alpha(Q)\tau_A+\beta(Q)\tau_B+\gamma(Q)\tau_C+\delta(Q)\tau_D+K(Q)$$
				associated affine form $Q$.

				Then, for all sequence  $(Q_n)_{n\in \NN}$ of hinges
				such that for all $n\in \NN$, $Q_n$ is of type $(x,n)$  and axis length $l$ we have :
				\begin{eqnarray*}
				 \lim_{n\rightarrow +\infty} \frac{\alpha(Q_n)}{\gamma(Q_n)}= -1&\quad\quad&\displaystyle
				 \lim_{n\rightarrow +\infty} \frac{\beta(Q_n)}{\gamma(Q_n)}= 0\\
				  \lim_{n\rightarrow +\infty} \frac{\delta(Q_n)}{\gamma(Q_n)}= 0&&\forall n\in\NN,\,\gamma(Q_n)>0
				\end{eqnarray*}
				\end{lem}
				\begin{proof} Let $L>0$,
				and let $Q=([ABCD],[AC])$ be a hinge of type $(x,L)$ such that $AC=l$.
				 Whithout loss of generality, we may choose cartesian coordinates of $\E^2$ such that $A:(0,0)$ is the origin,
				 $C :(0,l)$, $x_B>0$ and  $x_D<0$.

				 There exists some $\lambda>0$ such that
				 \begin{eqnarray*}
					\beta(Q)=\lambda\frac{l}{|x_B|}&\quad\quad& \alpha(Q)=\lambda\left(\frac{l-y_B}{|x_B|}+\frac{l-y_D}{|x_D|} \right)\\
					\delta(Q) = \lambda\frac{l}{|x_D|} &\quad\quad & \gamma(Q)=\lambda\left(\frac{y_B}{|x_B|}+\frac{y_D}{|x_D|}\right) \\
				 \end{eqnarray*}

				We have $|x_B|\leq x$, $|x_D|\leq x$,  $y_B\geq \sqrt{L^2-x^2}$ and  $y_D\geq \sqrt{L^2-x^2}$ ;
				 thus $\gamma(Q)>0$ and  $$ -1\leq\frac{\alpha(Q)}{f(Q)}\leq -1+ \frac{l}{\sqrt{L^2-x^2}}$$
				  $$ 0\leq\frac{\beta(Q)}{\gamma(Q)}\leq \frac{l}{\sqrt{L^2-x^2}} \quad  0\leq\frac{\delta(Q)}{\gamma(Q)}\leq \frac{l}{\sqrt{L^2-x^2}}.$$
				The result follows.
				\end{proof}
				\begin{lem}\label{lem:charniere_type_existence}
				 	Let $e$ be non trivial geodesic segment of $(\Sigma,S)$ going from some $P_1 \in S$ to some $P_2\in S$.

				 	There exists $x_0>0$ such that for all $L>0$, there is an immersed hinge
				 	$Q=([ABCD],[AC],\eta)$ of type 
				 	$(x_0,L)$ such that $\eta([AC])=e$.
				\end{lem}
				\begin{proof}
				 Let $M := \max_{x\in \Sigma} \d(x,S)$ and $m:=\min_{s\in S} \min _{s'\in S\setminus\{s\}}\d(s,s')$.
				 
				 Define $\Phi : \U\rightarrow \Sigma$ as the exponential map at $P_1$ defined on some maximal starshaped open neighborhood $\U$ of $0$ in the tangent plane $T_{P_1}\Sigma$ above $P_1$ such that $\Phi(\U\setminus \{0\})\subset \Sigma\setminus S$.  We identify $T_{P_1}\Sigma$ with $\E^2_\alpha$ where $\alpha$ is the cone angle at $P_1$ so that $\Phi$ is an isometric map from an open of $\E^2_\alpha$ to $\Sigma$. We choose polar coordinates $(r,\theta)$ of $\E^2_\alpha$ so that the direction $\theta=0$ is the initial derivative of the segment $e$.

				 With $\beta=\min(\alpha/2,\pi/6)$ define \begin{eqnarray*}r_{\max}&:&
				 \fonctionn{]-\beta,\beta[}{\RR_+^*\cup\{+\infty\}}{\theta}{ \max \{r \in \RR_ + ~|~  (r,\theta') \in \U \}}
				                                          \\R_{\pm}&:& \fonctionn{]0,\beta[}{\RR_+^*\cup\{+\infty\}}{\theta}{\min_{\theta'\in ]0,\theta]} r_{\max}(\pm\theta')}. 
				                                          \end{eqnarray*}

				 For any given $\theta \in ]-\beta,\beta[$, if $r_{\max}(\theta)<+\infty$ we extends $\Phi$ continuously to $(r_{\max}(\theta),\theta)$; note that in this case $\Phi(r_{\max}(\theta),\theta)\in S$.
				 
%
%
				\begin{claim}{$\varlimsup_{\theta\rightarrow 0^+} \theta R_\pm(\theta)< 2M$.}
				\end{claim}
				 \\ Let $\theta\in ]0,\beta[$, $\Phi$ is defined on the interior  the triangle $[OAB]\subset \E^2_{\alpha}$ with $A=(R_+(\theta),0)$ and $B=(R_+(\theta),\theta)$ in polar coordinates. The inscribed circle of $[OAB]$ bounds an open disc whose image by $\Phi$ does not contain any element of $S$, hence the radius $\frac{1}{2}R_+(\theta)\left(\cos(\theta)+\sin(\theta)-1)\right)$ of this inscribed circle is lesser than $M$. One easily checks that $\cos(\theta)+\sin(\theta)-1\sim_{\theta \rightarrow 0^+} \theta$. The result follows for $R_+$ and one may proceed the same way for $R_-$.

				\begin{claim}{$\lim_{\theta\rightarrow 0^+} R_\pm(\theta)= +\infty$.}
				\end{claim}\\
					The function $R_+$ is non-decreasing by definition so the limit is well defined. Define a sequence $(\theta_n)_{n\in\NN}$ as follows: choose some $\theta_0\in]0,\beta[$ such that $r_{\max}(\theta_0)=R_+(\theta_0)$ and $\sin(\theta_0)\leq m/2$, then for all $n\in\NN$ take $\theta_{n+1}\in]0,\theta_{n}/2[$ such that $r_{\max}(\theta_{n+1})=R_+(\theta_{n+1})$. The map $\Phi$ can be continously extended to the domain $$D:=\bigcup_{n\in\NN}\left\{(r,\theta)~:~\theta \in [0,\theta_n], r\leq R(\theta_n)\right\}.$$
					Write $Q_n:=(R_+(\theta_n),\theta_n)$; since $\forall n\in\NN, \Phi(Q_n)\in S$, then for all $n\in\NN$,
					$$R_+(\theta_{n+1})-R_+(\theta_n)+\theta_n-\theta_{n+1} =\d_D(Q_n,Q_{n+1})\geq \d_\Sigma(\Phi(Q_n),\Phi(Q_{n+1}))\geq m,$$ thus 
					$$\forall n\in\NN, \quad R_+(\theta_{n})\geq nm + R_+(\theta_0)+\theta_0-\theta_n\xrightarrow{n\rightarrow +\infty} +\infty$$
					One may proceed the same way for $R_-$.
									\def\sq{0.707106781}
				\def\myrad{6cm}
				\def\deltarad{1cm}
				\def\deltatheta{7}
				\def\thetamax{30}
						 \begin{center}
							\begin{tikzpicture}[scale=1]
								\coordinate (O) at (0,0);
								\draw  (\thetamax:\myrad) coordinate (slcoord) -- (\thetamax:0);
								\draw 
									(\thetamax-\deltatheta:\myrad) coordinate (slcoord) arc[start angle=\thetamax-\deltatheta,end angle=\thetamax,radius=\myrad];
								\draw  (\thetamax-\deltatheta:\myrad) coordinate (slcoord) -- (\thetamax-\deltatheta:\myrad+\deltarad) ;
								\draw 
									(\thetamax-2*\deltatheta:\myrad+\deltarad) coordinate (slcoord) arc[start angle=\thetamax-2*\deltatheta,end angle=\thetamax-\deltatheta,radius=\myrad+\deltarad];
								\draw  (\thetamax-\deltatheta:\myrad) coordinate (slcoord) -- (\thetamax-\deltatheta:\myrad+\deltarad) ;
								\draw 
									(\thetamax-3*\deltatheta:\myrad+2*\deltarad) coordinate (slcoord) arc[start angle=\thetamax-3*\deltatheta,end angle=\thetamax-2*\deltatheta,radius=\myrad+2*\deltarad];
								\draw  (\thetamax-2*\deltatheta:\myrad+\deltarad) coordinate (slcoord) -- (\thetamax-2*\deltatheta:\myrad+2*\deltarad) ;
								\draw[dashed]
									(0:\myrad+3*\deltarad) coordinate (slcoord) 
									arc[start angle=0,end angle=\thetamax-3*\deltatheta,radius=\myrad+3*\deltarad];
									
								\draw  (\thetamax-3*\deltatheta:\myrad+2*\deltarad) coordinate (slcoord) -- (\thetamax-3*\deltatheta:\myrad+3*\deltarad) ;
								
								\draw (0,0) -- (\myrad+4*\deltarad,0);
								\coordinate[label=$O$] (O) at (0:0) coordinate (slcoord) ;
								\coordinate[label=$Q_0$] (A) at (\thetamax:\myrad) coordinate (slcoord) ;
								\coordinate[label=$Q_1$] (B) at (\thetamax-\deltatheta:\myrad+\deltarad) coordinate (slcoord) ;
								\coordinate[label=$Q_2$] (C) at (\thetamax-2*\deltatheta:\myrad+2*\deltarad) coordinate (slcoord) ;
								\coordinate[label=$Q_3$] (D) at (\thetamax-3*\deltatheta:\myrad+3*\deltarad) coordinate (slcoord) ;


							\end{tikzpicture}
						\end{center}
					
				We now come back to the proof of the Lemma. Take some $x_0>M$, for any $L\in \RR_+$, from the claims above, there exists some $\theta_+\in]0,\beta[$ and $\theta_-\in ]-\beta,0[$ such that $|\sin(\theta_\pm)r_{\max}(\theta_\pm)| \leq x_0 $  and $r_{\max}(\theta_{\pm})\geq L$. Choose such a $\theta_\pm\in ]-\beta,\beta[$ and notice $\Phi$ is well defined on the hinge
				$Q=[ABCD]$ with $A=O$ $B:=(r_{\max}(\theta_-),\theta_-)$, $C:=(\mathrm{len}(e),0)$ and $D:=(r_{\max}(\theta_+),\theta_+)$. The hinge $Q$  is of type $(x_0,L)$ and $\eta := \Phi_{|Q}$ is an isometric immersion.  $(Q,\eta)$  is then an immersed hinge of type $(x_0,L)$ of $(\Sigma,S)$ with vertices in $S$ and such that $\Phi([AC])=e$.
				
				\end{proof}

				\begin{lem} \label{lem:admissible_compact}
					There exists  $C>0$ such that for all $A,B\in S$ and all $\tau\in \PSigmaS$, $$|\tau(A)-\tau(B)|\leq C.$$

					\end{lem}
					\begin{proof}
						For Corollary \ref{cor:translation}, it is enough to find a $C>0$ such that 
						$$\forall \tau\in \PSigmaS,\quad  \min\tau=0~\Rightarrow~ \max \tau\leq C$$
					
					From Lemmas \ref{lem:charniere_type_rel} and \ref{lem:charniere_type_existence} and from  Proposition \ref{prop:domaine_convexe},
						for all $\varepsilon>0$, and $A,B\in S$, if there exists a geodesic from $A$ to $B$, then  there exists $K>0$ such that $\tau\in \PSigmaS$,
						$$ \left|\tau_A-\tau_B\right| \leq \varepsilon \max \tau +K $$
						For all $A,B\in S$ there exists a broken geodesic from $A$ to $B$ with breaking points in $S$ hence 
						$$\forall \varepsilon>0, \forall A,B\in S,\,\exists K>0,\quad \mathrm{s.t.}\quad \forall \tau \in \PSigmaS,\, |\tau_A-\tau_B|\leq \varepsilon \max \tau + K $$

						Since $S$ is finite, 
						$$\exists K>0,\forall A,B\in S,\forall \tau\in \PSigmaS,\quad  |\tau_A-\tau_B|\leq \frac{1}{2}\max \tau +K$$

						Choose such a $K>0$ and define $C=2K$, then for all $\tau\in \PSigmaS$ such that $\min \tau=0$
						 $$\max \tau = |\max \tau-\min\tau| \leq \frac{1}{2} \max\tau +K $$
						thus 
						$$\max \tau \leq 2K = C.$$
					\end{proof}

					\begin{proof}[Proof of Theorem \ref{theo:domain_description}]
						Let $\pi$ be the orthogonal projection of $\RR^S$ onto $H:=\{\tau\in \RR^S~|~ \sum_{s\in S}\tau(s)=0\}$, note that the kernel of $\pi$ is $\RR\cdot \mathbf{1}_S$. 
						For each triangulation $\TT$, the set of $\tau\in \RR_+^S$ such that $\widetilde \tau_{\tau,\TT}$ is Q-convex is the domain $$\CellPSigma := \RR_+^S\cap \bigcap_{e \in \mathrm{Edge(\TT)}} (Q_e^*)^{-1}(\RR_-).$$
						Since $\mathbf{1}_S$  is in the kernel of the linear part of all the affine form $Q^*_e$ and since the number of edges of $\TT$ is finite, $\overline \CellPSigma := \pi(\CellPSigma)$ is a convex polyhedron and $\CellPSigma = (\overline \CellPSigma +\RR\cdot \mathbf{1}_S)\cap \RR_ +^S$.
						
						On the one hand, $$\PSigmaS = \bigcup_{\TT}P_{\TT}$$ where $\TT$ runs through all adpated triangulations of $(\Sigma,S)$. Then defining $\overline{\PSigmaS} := \bigcup_{\TT}\overline{\P}_{\TT}$ we have $\PSigmaS= (\overline{\PSigmaS} + \RR\cdot\mathbf{1}_S)\cap \RR_+^S$.
						
						On the other hand, by Lemma \ref{lem:triangulation_finie}, there are only finitely adapted triangulations $\TT$ such that $\CellPSigma$ is non empty. The domain $\overline \PSigmaS$ is thus a polyhedron. 
						
						By Lemma \ref{lem:admissible_compact}, $\overline \PSigmaS = \pi(\PSigmaS)$ is compact and by Proposition \ref{prop:domaine_convexe}, $\PSigmaS$ is convex hence $\overline \PSigmaS$ is convex. 
						
						Let $B$ be an open ball in $\RR_+^S$ of center $0$ and some radius $R$ such that $\pi(B)\supset\overline \PSigmaS$; such a ball exists by compactness of $\overline \PSigmaS$.					
						Let $\TT_0$ be a subtriangulation of the Delaunay triangulation and define $C:=\max\widetilde\tau_{0,\TT_0}+R$. By  Lemma \ref{lem:triangulation_finie}, there exists only finitely triangulations $\TT_0,\cdots, \TT_n$ such that $\exists \tau\in \RR_+^S, \widetilde \tau_{\tau,\TT_i}\leq C$. Let $Q_1,\cdots,Q_N$ the unflippable immersed hinges appearing in these triangulations $\TT_1,\cdots,\TT_n$.
						
						The following claim will imply the last point of the Theorem.
						\begin{claim}{$\PSigmaS = (B+\RR_+\cdot \mathbf{1}_S)\cap \bigcap_{i=1}^N {Q^*_i}^{-1}(\RR_-)$}
						 
						\end{claim}
						Let $\tau\in \PSigmaS$, by Proposition \ref{prop:domaine_convexe} $\tau\in \bigcap_{i=1}^N {Q^*_i}^{-1}(\RR_-)$ and $\pi(\tau)\in \overline \PSigmaS \subset \pi(B)$ thus $\tau\in B+\RR_+ \mathbf{1}_S$. 
						
						Let $\tau \in  (B+\RR_+\cdot \mathbf{1}_S)\cap \bigcap_{i=1}^N {Q^*_i}^{-1}(\RR_-)$, let $\lambda\in\RR_+$ such that $\tau' := \tau-\lambda\mathbf{1}_S \in B$. 						
						Apply the flipping algorithm for $\tau'$ starting from $\TT_0$;  by Proposition \ref{prop:algo_stop}, the ending triangulation $\TT$ is such that $\widetilde\tau_{\tau,\TT}\leq C$ hence $\exists i\in \lsem1,n\rsem,\TT=\TT_i$. By the same Proposition, every flippable hinge of $\TT$ are $\tau'$-legal; since we have $\forall i\in \lsem 1,N\rsem, Q_i^*(\tau)\leq 0$, 
						by Lemma \ref{lem:kernel_hinge} we have $\forall i\in \lsem 1,N\rsem, Q_i^*(\tau')\leq 0$, in particular all unflippable hinges of $\TT$ are $\tau'$-legal. Finally $\TT$ is $\tau'$-Delaunay thus, by Lemma \ref{lem:kernel_hinge}, $\TT$  is $\tau$-Delaunay and $\tau\in \PSigmaS$. We have proven the claim.
						
						Any support plane of $\PSigmaS$ is thus given either by the boundary of $\RR_+^S$, hence "$\tau(\sigma)=0$", or boundary of a halfspace "$Q_i^*\leq 0$", hence "$Q^*_i=0$", for some $i\in\lsem1,N\rsem$.

					\end{proof}

			\section{Volkov Lemma for Lorentzian Convex cones}
				\label{sec:etude_domaine_convexe_bord}
				In effective methods used to prove Alexandrov-like Theorems, at some point a Volkov Lemma bounding the cone angle  $\Theta$ of a cone around a singular line of angle $\kappa$ in a Riemannian manifold.
				Though results such as stated below are used in one way or another \cite{MR2127379,MR2410380,MR2453328,MR2469522,MR2813423}, to our knowledge a complete proof of the bounds we use is not available in english (one may appear in the original thesis of Volkov which is in Russian and only a summary is available in english \cite{MR3904042}), we thus provide a complete proof.

				Let $\SS^{1,1}_\kappa:=\{(t,x,y) \in \mass{\kappa}~|~ r^2-t^2=1\}$, it is a timelike surface of constant positive curvature 1.
				\begin{theo}\label{theo:volkov_lorentz} Let $\Theta>0$ and $\kappa>0$. Let $\C$ be a convex spacelike cone in $\mass{\kappa}$ of cone angle $\Theta$. 
					
					 Assuming  has a coplanar wedge of euclidean angle at least $\min(\pi,\Theta)$:

					\begin{itemize}
					 \item if $\Theta> 2\pi$  then $\kappa > 2\pi$.;
					 \item if $\Theta = 2\pi$ then $\kappa = 2\pi
					 $;
					 \item if $\Theta\in ]\pi,2\pi[$ then  $\kappa \geq \Theta$;
					 \item if $\Theta = \pi$ then $\kappa = \pi$;
					 \item if $\Theta<\pi$ then  $\kappa\in ]0,\Theta]$ with $\kappa=\Theta$ if and only if $\C$ is the horizontal plane;
					\end{itemize}
					and all the bounds above are sharp.
					
					\end{theo}
					
				\begin{rem} In Minkoswki, a convex cone always has a cone angle bigger than $2\pi$. One may expect this to carry out in $\mass{\kappa}$ for arbitrary $\kappa\geq0$. Theorem \ref{theo:volkov_lorentz} shows this intuition is  only valid for $\kappa \in [0,\pi]\cup\{2\pi\}$
				 
				\end{rem}

				\subsection{Stalks of Lorentzian cones}

					\begin{defi}[Stalk of a spacelike cone]
					Let $\kappa>0$, $\Theta>0$ and $\C$ a spacelike cone of $\mass{\kappa}$ of cone angle $\Theta$
					and class $\mathscr{C}^2_{pw}$. We assume the vertex of $\C$ is on the origin of $\mass{\kappa}$.
					In cylindrical coordinates $(r,\theta, t)$, the set $\SS^{1,1}_\kappa\cap \mathcal C$ can be parametrized by arc length 
					$$\mathcal C\cap \SS^{1,1}_\kappa= \left\{ \left(\begin{array}{c}
						t(s)\\
						r(s)\\
						\theta(s)
					\end{array}\right) ~:~ s\in \RR\right\}.
					$$
					The stalk $\rho_{\mathcal C}$ of $\C$ is the function $ t : \RR\rightarrow \RR$ of this  parametrization.
					\end{defi}
					\begin{rem}
					  The stalk $\rho$ of a cone is unique up to pre-composition by an affine transformation slope $\pm1$.
					\end{rem}

					\begin{prop}\label{prop:T_convexe_base}
					Let $\kappa>0$, $\Theta>0$ and $\mathcal C$ a polyhedral cone of $\mass{\kappa}$ of cone angle $\Theta$ whose vertex is on the origin and of stalk $\rho:=\rho_{\mathcal C}$. We have the following:
					\begin{enumerate}
					\item  $\rho: \RR \rightarrow \RR$ is $\Theta$-periodic continuous and $\mathscr C^\infty_{pw}$;
					\item  $\rho$ is piecewise trigonometric (ie piecewise of the form $\theta \mapsto A\cos(\theta+\varphi)$);
					\item  $\C$ is convex if and only if $\rho$  Q-convex;
					\item  $\displaystyle\kappa=\int_{0}^{\Theta} \frac{\sqrt{1+\rho(\theta)^2+\rho'(\theta)^2 }}{1+\rho(\theta)^2}\d \theta$.
					\end{enumerate}
					\end{prop}
					\begin{proof} The first 3 points are simple enough. 
					To obtain the last item we first choose $\theta$ increasing and notice:
					$$2\pi= \int_{0}^\Theta \theta'(s)\d s,\quad \quad
					r^2-\rho^2=1, \quad\quad -(\rho')^2+(r')^2+\left(\frac{\kappa}{2\pi}\right)^2 r^2(\theta')^2=1.$$
					Therefore $rr'= \rho\rho' $ and
					\begin{eqnarray*}
					 (\theta')^2&=&\left(\frac{2\pi}{\kappa}\right)^2 \frac{1+(\rho')^2-(r')^2}{r^2}\\
					 &=& \left(\frac{2\pi}{\kappa}\right)^2 \frac{1+(\rho')^2-\left( \frac{\rho \rho'}{r}\right)^2}{1+\rho^2}\\
					 &=& \left(\frac{2\pi}{\kappa}\right)^2 \frac{\left(1+(\rho')^2\right)\left(1+\rho^2\right)-\rho^2 (\rho')^2}{(1+\rho^2)^2}\\
					 (\theta')^2&=& \left(\frac{2\pi}{\kappa}\right)^2 \frac{1+(\rho')^2 + \rho^2}{(1+\rho^2)^2}\\
					 \theta'&=& \frac{2\pi}{\kappa}\frac{\sqrt{1+(\rho')^2 + \rho^2}}{1+\rho^2}.
					\end{eqnarray*}
					Insert the last line  in $2\pi = \int_0^\Theta \theta'$ to get the result.

					\end{proof}
					\begin{rem}\label{rem:inversion_sens} For $\rho:I\rightarrow\RR$ continuous $\mathscr C^1_{pw}$, $\rho$ is Q-convex if and only if $s\mapsto \rho(-s)$ is Q-convex.
					\end{rem}
					
					\begin{cor}\label{cor:kappa_coplanaire}
					Let $\kappa>0$, $\Theta>0$ and $\C$ a spacelike polyhedral cone in
					$\mass{\kappa}$ of cone angle $\Theta$.

					If its stalk $\rho$ is $\mathscr C^1$ then $\kappa=\Theta$. 
					Furthermore, if $\kappa$ is not a multiple of $2\pi$ then $\rho=0$.
					\end{cor}

					\begin{proof}
					If $\rho=0$, the result follows from point 4 of Proposition \ref{prop:T_convexe_base}. 
					Otherwise, the minimal period $\rho$ is $2\pi$ thus $\Theta \in 2\pi \NN^*$.
				 	Noticing the integral  in point 4 of Proposition \ref{prop:T_convexe_base}  yields $2\pi$ on a period of a trigonometric function and we obtain $\Theta = \kappa$.
					\end{proof}

					\begin{defi}[Mass of a stalk] \
					For $\rho : [a,b]\rightarrow \RR$ (resp. $\rho : \RR\rightarrow \RR$ $\Theta$-periodic), define 	$$ \kappa(\rho) : = \int_a^b \frac{\sqrt{1+\rho^2+\rho'^2}}{1+\rho^2}\quad \left(\mathrm{resp.} ~~\int_0^\Theta \frac{\sqrt{1+\rho^2+\rho'^2}}{1+\rho^2} \right).$$

					\end{defi}

					\begin{rem}
					  Every $\rho:\RR\rightarrow \RR$ piecewise trigonometric Q-convex and $\Theta$-périodique induces
					a convex embedding of $\mathbb E^2_\Theta$ into $\mass{\kappa(\rho)}$.
					Furthermore, this embedding is essentially unique:
					From Proposition \ref{prop:T_convexe_base}, the mass $\kappa$ is given by $\rho$  there is thus no choice for the space $\mass{\kappa}$ and two embeddings of same germ only differ by rotation or plane symmetry.
					\end{rem}

					\begin{defi} Let $\SS^2_\infty$ be the universal covering of the round sphere branched over its north and south poles eg $[-\pi/2,\pi/2]\times \RR /\sim$ endowed the metric $$\d s^2 = \d \phi^2 +\cos(\phi)^2\d \theta^2$$ 
					where $\sim$ identifies all points such that $\phi=\pi/2$ together as the north pole $N$ and all points such that $\phi=-\pi/2$ as the south pole $S$.
					\end{defi}
					\begin{defi} A  piecewise geodesic curve $\gamma: I\rightarrow \SS^2_\infty$ is Q-convex if $\theta\circ \gamma$ is injective and $\phi\circ \gamma$ is Q-convex.

					\end{defi}

					\begin{lem} \label{lem:stalk_to_geodesic}
						Let $\rho : [a,b]\rightarrow \RR$ continous piecewise trigonometric Q-convex. Then $\gamma: [a,b]\rightarrow \SS^2_\infty, \theta \mapsto (\arctan \rho(\theta),\theta)$ is a piecewise geodesic Q-convex curve. Furthermore, $$\kappa(\rho)=\mathrm{length}(\gamma).$$
					\end{lem}
					\begin{proof} 
					 Curves of the form $t\mapsto (\phi(t),\theta(t))$ with $\phi(t)= \arctan(\alpha\cos(\theta+\phi)$ and $\theta(t)=t$ are geodesic of $\SS^2_\infty$ and  Q-convexity is trivial. Then, simply write \begin{eqnarray*}\mathrm{length}(\gamma)&=&\int_a^b \sqrt{\frac{(\rho')^2(\theta)}{(1+\rho^2(\theta))^2}+\cos^2(\arctan\circ \rho(\theta))  }\d \theta \\
					 &=&  \int_a^b \sqrt{\frac{(\rho')^2(\theta)}{(1+\rho^2(\theta))^2}+\frac{1}{1+
					 \rho^2(\theta)} }\d \theta   \\
					 &=& \int_a^b \sqrt{\frac{\rho'(\theta)+1+\rho^2(\theta)}{(1+\rho^2(\theta))^2} }\d \theta   \\
					 &=& \kappa(\rho)
					 \end{eqnarray*}

					\end{proof}

					\subsection{Lower bounds}
					Lemma \ref{lem:stalk_to_geodesic} provides a neat geometrical traduction from Lorentzian to Riemannian. Indeed spacelike minimizing curves do not exists in a Lorentzian manifold while in a Riemannian manifold they are geodesics.

					\begin{prop}\label{prop:volkov_lorentz_case_bigger_than_2pi} Let $\Theta\geq\pi$  then 
					$$\inf\kappa(\rho) = \min(2\pi,\Theta)$$ the infimum being taken over $\rho:\RR/\Theta\ZZ\rightarrow \RR$  piecewise trigonometric and Q-convex which are trigonometric on a interval of length at least $\pi$. 
					Furthermore, the infimum is a minimum if and only if $\Theta\leq 2\pi$

					\end{prop}
					\begin{proof} Define $S_\Theta$ the set of $\Theta$-periodic convex stalk which have a trigonometric interval of length at least $\pi$.
					\begin{itemize}
					 \item Assume $\Theta\geq2\pi$. Consider for $\alpha>0$ the stalk  $$\fonction{\rho}{\RR/\Theta\ZZ}{\RR}{\theta}{\left\lbrace \begin{array}{ll}
																			\sinh(\alpha) \sin(\theta) & \mathrm{for }~ \theta \in [-3\pi/2,\pi/2] \\
												\sinh(\alpha)	& \mathrm{otherwise}
																		\end{array}\right.
					} $$
				so that $\rho$ is convex and $\kappa(\rho)=2\pi+\frac{\Theta-2\pi}{\cosh(\alpha)}$. As a result, $\inf_{\rho\in S_\Theta} \kappa(\rho)\leq 2\pi.$
				
				\item Assume $\Theta< 2\pi$, then the stalk $\rho=0$ is such that $\kappa(\rho)=\Theta$.
				
				\item Let $\gamma : [0,\Theta-\pi]\rightarrow \SS^2_\infty$ be a Lipchitz curve  from $(\phi_0,0)$ to $(-\phi_0,\Theta)$ minimizing the length with $\gamma(0),\gamma(\Theta-\pi) \notin \{N,S\}$.  
				The curve $\gamma$ is a possibly broken geodesic with breaking points in $\{N,S\}$.

				If $\gamma$ is unbroken, then up to reparametrization, $\phi\circ \gamma$ is of the form $\theta\mapsto\arctan(\alpha\cos(\theta+\theta_0))$ and since $\phi\circ \gamma(0)=-\phi\circ \gamma(\Theta-\pi)$. If $\Theta\geq2\pi$ then the length of such a curve is at least $\pi$; otherwise up to reparametrization $\rho(\theta)=\sinh(\alpha)\sin(\theta)$ with $\theta\in [m-\pi/2,\pi/2-m]$ and $m=\pi-\Theta/2$. Then $$\kappa(\rho) = \pi-2\arctan\left(\frac{\tan(m)}{\cosh(\alpha)}\right)$$ which is minimal if and only if $\alpha=0$ in which case the length of $\gamma$ is $\Theta-\pi$.
				
				If $\gamma$ has one breaking point, then $\gamma$ is formed of a geodesic from $\gamma(0)$ to $N$ (resp. $S$) and then a geodesic from $N$ (resp. $S$) to $\gamma(\Theta)$. Hence, the length of $\gamma$ is exactly $\pi$. 
				
				If $\gamma$ has more than one breaking point, then it contains a geodesic from $N$ to $S$ and its length is strictly bigger than $\pi$.
								
				In any case, the length of the curve associated by Lemma \ref{lem:stalk_to_geodesic} to a stalk $\rho$ in $S_\Theta$ is bounded from below by $\pi$ plus the length of such a minimizing curve $\gamma$ hence $\inf_{\rho\in S_\Theta} \kappa(\rho)= \min(\Theta,2\pi)$. Furthermore, the infimum is a minimum if and only if a minimizing curve can be associated to a stalk, which is only possible if $\gamma$ reach neither $N$ nor $S$; this excludes the broken geodesic cases and the unbroken geodesic case is minimizing if and only if $\Theta\leq 2\pi$.

					\end{itemize}

					\end{proof}

					\begin{prop}\label{prop:volkov_inf_pi} Let $\Theta < \pi$ then 
					 $$\inf\kappa(\rho) = 0$$ the infimum being taken over $\rho:\RR/\Theta\ZZ\rightarrow \RR$  piecewise trigonometric and Q-convex which are trigonometric on a interval of length at least $\Theta$. 
					 
					 Furthermore, among such stalks, $\kappa(\rho)\leq \Theta$ with equality if and only if $\rho= 0$.
					\end{prop}
					\begin{proof}
					 Define $S_\Theta$ the set of $\Theta$-periodic convex stalk which have a trigonometric interval of length at least $\Theta$.
					 Any element of $S_\Theta$ is of the form $$ \fonction{\rho_\alpha}{\RR/\Theta\ZZ}{\RR}{\theta+\Theta\ZZ}{\sinh(\alpha) \cos(\theta)} $$ 
					 for some $\alpha\in \RR$ up to translation.  On the one hand, since $\rho_\alpha$ is Q-convex only for $\alpha\geq 0$ then  $$\forall \alpha\in\RR,\quad  \rho_\alpha \in S_\Theta ~~\Leftrightarrow ~~ \alpha\geq 0.$$ On the other hand $$\forall \alpha\geq  0,\quad  \kappa(\rho_\alpha) = 2\arctan\left(\frac{\tan(\Theta/2)}{\cosh(\alpha)}\right)\xrightarrow{\alpha\rightarrow +\infty} 0.$$
					
					The infimum $\inf_{\rho \in S_\Theta} \kappa(\rho)=0$ follows. Note that $\alpha\mapsto \kappa(\rho_\alpha)$ is decreasing, the maximum is thus reached for $\alpha=0$ hence $\rho=0$. The formula above gives $\kappa(0)=\kappa(\rho_0) = 2\arctan(\tan(\Theta/2))=\Theta$.
					\end{proof}
%
%
%
%
%
%
%
%
%

			\section{Einstein-Hilbert functional}
			\label{sec:fonctionnelle}
				\def\PsqTT{\CellPSigma^{1/2}}
				\def\Psq{\PSigmaS^{1/2}}
				We give ourselves a Euclidean surface $\Sigma$ with conical singularities and marked points $S\supset \sing(\Sigma)$, we will keep this surface fixed in the whole section.
				
				To sum up the results of the preceding sections, we have a construction that associate to any $\tau\in \PSigmaS$ a radiant spacetime $M(\tau)$ and a convex polyhedral embedding $\iota_\tau$ of $(\Sigma,S)$ into $M(\tau)$, we know from Proposition \ref{prop:plong_conv} this construction reaches every equivalence classes of such couple $(M,\iota)$ and is injective. By Theorem \ref{theo:domain_description}, $\PSigmaS$ is a convex domain of $\RR_+^S$ and is the union of finitely many convex cells each corresponding to a triangulation of $(\Sigma,S)$.
				
				The objective is now to construct polyhedral embeddings $(M,\iota)$ such that the singularities of $M$ have masses (hyperbolic angles) we gave ourselves a priori.

				\begin{defi}[Mass function]

				Let $\tau \in \PSigmaS$ and $(M(\tau),\iota_\tau)$ its associated polyhedral embedding of $(\Sigma,S)$, for $\sigma\in S$ define  $\kappa_\sigma(\tau)$ the mass (or hyperbolic angle) of $M(\tau)$ at $\iota_\tau(x) \in M(\tau)$.

				We define $\kappa : \PSigmaS \rightarrow \RR_+^{S}$ the map that associat to $\tau$ the vector $(\kappa_\sigma(\tau))_{\sigma\in S}$.
				\end{defi}
				\begin{rem} On each cell $\CellPSigma:=\{\tau\in \PSigmaS ~|~ \TT_\tau=\TT \}$, the function $\tau \mapsto \kappa$ is continuous and furthermore $\mathscr C^1$. Since we will actually compute the derivative later on we do not prove it now.
				Furthermore, if  $\tau \in \CellPSigma\cap \P_{\TT'}$ the triangulations  $\TT$ and $\TT'$ are  $\tau$-equivalent, $\kappa$ computed with either triangulation yields the same result since $M(\tau,\TT)\simeq M(\tau,\TT')$. The map $\tau\mapsto \kappa$ is thus continuous on $\PSigmaS$.
				\end{rem}

				Reformulating with these notations, we thus aim at solving the following.
				\begin{problem}\label{prob:kappa_bar}
				 Let  $\bar\kappa \in \RR_+^{S}$, is there some $\tau\in \PSigmaS$ such that $\kappa(\tau)=\bar\kappa$ and if so, is is unique?
				\end{problem}
				There is an restriction on the possible $\bar\kappa$. Indeed, for any $\tau\in \PSigmaS$, the spacetime $M(\tau)$ is the suspension of some marked closed $\H_{\geq0}$-surface $(\Sigma_{\H^2},S')$ homeomorphic to $\Sigma$ and the cone angles at $S'$ are the masses $\kappa(\tau)$. Therefore, by Gauss-Bonnet formula, 
				$\sum_{\sigma\in S} (2\pi-\kappa(\tau)_\sigma)- \mathrm{Area}(\Sigma_{\H^2}) = 2\pi\chi(\Sigma) = \sum_{\sigma\in S}(2\pi-\theta_\sigma)$
				hence 
				$$\forall \tau\in \PSigmaS, \quad \sum_{\sigma \in S} \theta_\sigma> \sum_{\sigma\in S}\kappa(\tau)_\sigma $$
				We are unable to prove that all $\bar\kappa$ such that  $\sum_{\sigma \in S} \theta_\sigma> \sum_{\sigma\in S}\bar\kappa_\sigma$ are reached and settle for a weaker statement.
				 \begin{theon}  Let $(\theta_\sigma)_{\sigma\in S}$ be the cone angles of $\Sigma$ at $S$. Recall the notation $(M(\tau),\iota_\tau)$ introduced in section \ref{sec:tau_susp}.			 
					 With $\kappa(x)$ the cone angle around $x$ for $x$ some point of a radiant spacetime, the map $$\fonctionn{\PSigmaS}{\RR^S_+}{\tau}{(\kappa\circ \iota_\tau)_{|S}}$$ 
					 reaches each point of $\prod_{\sigma\in S} [0,2\pi]\cap [0,\theta_\sigma[$ exactly once.
					 
						\end{theon}
				
				The proof rely on the analysis of a so-called Einstein-Hilbert functional, the first step is thus, is to define 
				a functional $\mathcal H_{\bar\kappa}$ on $\PSigmaS$ for any $\bar\kappa$  whose critial points are solution to Problem \ref{prob:kappa_bar}. In fact, one could check that such a functional exists checking $\frac{\partial\kappa_{\sigma_1}}{\partial h_{\sigma_2}} = \frac{\partial\kappa_{\sigma_2}}{\partial h_{\sigma_1}} $.
				Actually for technical reasons which will shortly make themselves clear, it will be more appropriate to define such a functional on the domain $\Psq := \{h\in \RR_+^S~:~ h^2\in \PSigmaS\}$.
				
				A standard analysis of the critical points of $\mathcal H_{\bar\kappa}$ as well its gradient on the boundary of $\Psq$ follows. Under the assumption that for all $\sigma\in S$, $\bar\kappa_\sigma$ is no greater than $2\pi$ and lesser than the cone angle of $\Sigma$ at $\sigma$, we show that critical points of $\mathcal H_{\bar\kappa}$ are positive definite and that the gradient of $\mathcal H$ on the boundary of $\PSigmaS$ is homotopic to an outward vector field.

					\subsection{Reminders on Lorentzian angles and Shläffli's Formula}
					The following is an adaptation of the exposition of Rabah Souam \cite{MR2030165}.

					To begin with, the modulus  $|u|$ of a vector  $u$ of $\mass{}$ is
					$$|u|=\sqrt{\langle u|u\rangle}$$
					with the convention that when $\langle u|u\rangle<0$ then $|u|=\lambda i$ with $\lambda>0$ and $i^2=-1$.
					Let $u,v$ be two vectors of $\mass{}$, the angle $\angle uv$ is defined so that it satisfies the following properties:
					\begin{enumerate}
					 \item for all vectors $u,v$, $\angle uv\in \RR+i\RR/(2\pi\ZZ)$;
					 \item for all vectors $u,v$, $\langle u|v\rangle = |u||v|\cosh(\angle uv)$:
					 \item for all vectors $u,v,w$ coplanar,  $\angle uv+\angle vw=\angle uv'$.
					\end{enumerate}
					Beware that if $u,v$ are spacelike, $\angle uv$ is not the usual angle $\widehat{uv}$ but actually $\widehat{uv}\cdot i$
					Angles are well defined only if neither  $u$ nor $v$ are lightlike.
					\begin{defi}[Type of a vector of $\E^{1,1}$]
					Choose a direct cartesian coordinate system $(t,x)$ of the vector space underlying $\E^{1,1}$.
					Let $u$ be a non lisghtlike vector of $\E^{1,1}$, the type $k_u\in \ZZ/4\ZZ$   of $u$ is defined as follows :
					\begin{itemize}
					 \item $k_u=0$ if $u$ future timelike;
					 \item $k_u=1$ if $u$ is spacelike with negative spacelike coordinate,
					 \item $k_u=2$ if $u$ is past timelike;
					 \item $k_u=3$ if $u$ is spacelike with positive spacelike coordinate.
					\end{itemize}
					\end{defi}

					\begin{defi} Let $u,v$ be two non lightlike unit vector $\mass{}$ and let $\Pi$ the vectorial plane generated by $u$
					and $v$.
					\begin{itemize}
					 \item If $\Pi$ is spacelike, $$\angle uv=i\theta$$ with $\theta$ the angle from $u$ to $v$ in $\Pi$ oriented by future timelike normal.
					 \item  if $\Pi$ is timelike and $u,v$ of types $k_u,k_v$ in $\Pi$ identified with $\E^{1,1}$ oriented by the the basis $(u,v)$, then
					 $$\angle uv=\alpha+i(k_v-k_u)\frac{\pi}{2}$$
					 avec $\alpha$ the length of the geodesics from $u'$ to $v'$ where $u'$ (resp. $v'$) is the unique future unit timelike vector of $\Pi$ orthogonal or colinear to $u$ (resp. $v$)
					\end{itemize}

					\end{defi}

					\begin{defi}[Dihedral angle]
					Let $\Pi_1$ and $\Pi_2$ two vectorial half-planes intersecting along their common boundary $\Delta$.
					Assume none of $\Pi_1,\Pi_2$  and $\Delta$ are lightlike and write $\nu_i=\Delta^\perp\cap\Pi_i$, $i\in \{1,2\}$.
					We choose some $u\in \Delta$ and for $i\in \{1,2\}$ define $n_i$ the unique unit vector  normal to $\Pi_i$  such that $(u,\nu_i,n_i)$  is a direct basis.

					The signed dihedral  angle $\angle \Pi_1\Pi_2$ between the planes $\Pi_1$ and $\Pi_2$ is then defined as follows 
					$$\angle \Pi_1\Pi_2 := \angle n_1n_2.$$

					\end{defi}
					\begin{rem}
					In the definition above, the dihedral angle does not depend on the choice of $u$. 
					\end{rem}

					\begin{defi}[1-parameter family of locally Minkowski polyedron]\

						A 1-parameter family of locally Minkowski polyedron is the data of a simplicial complex $\mathcal K$ and a map $\psi:[0,1]\times \mathcal K \rightarrow \mass{}$ such that
						\begin{enumerate}
						 \item for all simplex $P$ of $\mathcal K$  and all $t\in [0,1]$,  $\psi_{|\{t\}\times P}$ is injective and $\psi(t,P)$ a polyhdron of $\mass{}$;
						 \item for all $t\in [0,1]$, the map  $\psi(t,\cdot)_{|
						 Int(\mathcal K)}$ is  a local homeomorphism ;
						 \item For all $x\in \mathcal K$, the map $\psi(\cdot,x)$ is $\mathscr C^1$.
						\end{enumerate}
					\end{defi}

					Let $(\mathcal K,\psi)$ a 1-parameter family of locally Minkowski polyedron.
					If $e$ is an edge of $\mathcal K$, then for all $t\in [0,1]$,
					 we write $l_{e,t}\geq0$ the length of the edge $\psi(e,t)\subset \mass{}$
					 and $\theta_{e,t}$ the sum of the dihedral angles between the faces of the simplices of $\mathcal K$ around the edge $e$.
					\begin{theon}[Schläffli's formula \cite{MR2030165}]\
					
					 Let $(\mathcal K,\psi)$ a 1-parameter family of locally Minkowski polyedron such that none of its faces or edges are lightlike.
					 Denoting by $\mathcal A$ the set of edges of $\mathcal K$, we have
					 $$ \sum_{e\in \mathcal A} l_{e,t} \frac{\d \theta_{e,t}}{\d t}=0.$$
					\end{theon}

				\subsection{Einstein-Hilbert functional, definition and gradient}

				We begin by defining the set of squared root of admissible times:
				$$\PSigmaS^{1/2}:=\{(\sqrt{\tau_\sigma})_{\sigma\in S} ~:~\tau \in \PSigmaS \}.$$
				Elements of  $\PSigmaS^{1/2}$ will be denoted systematically by $h$ while elements of $\PSigmaS$ will be denoted by $\tau$. Going from the one to the other being simple we extends all definitions to $\Psq$:
				$M(h) := M(h^2,\TT_{h^2})$, etc.
				
				\begin{defi}[Einstein-Hilbert functional]

						Let $\bar\kappa \in \RR_+^{S}$.
						For $h\in \PSigmaS^{1/2}$ and for $e$ edge of $\TT_h$, we denote by
						$l_e$ the length of $e$ and by $\theta_e$ the dihedral angle of the embedding $\iota_h$ at the edge $e$.

	 					The Einstein-Hilbert functional is defined as follows :
						 $$\fonction{\mathcal H_{\bar\kappa}}{\PSigmaS}{\RR}{h}{\sum_{\sigma \in S} h_\sigma (\kappa_\sigma - \bar \kappa_\sigma)+ \sum_{e \in \A_\tau} l_e\theta_e}.$$
				\end{defi}

					\begin{prop} \label{prop:diff_H}
					Let  $\bar\kappa \in \RR_+^{S}$, 
					the functional $\mathcal H_{\bar \kappa}$ is well defined, $\mathscr C^1$ on $\PSigmaS^{1/2}$ and
					$$\d \mathcal H_{\bar \kappa} = \sum_{\sigma \in S} (\kappa_\sigma-\bar\kappa_\sigma)\d h_\sigma $$
					\end{prop}
					\begin{proof}
					 Consider the family of compact locally Minkowski polyhedron
						$(Q_h)_{h\in \PSigmaS^{1/2}}$
					 given by the past of the polyhedral Cauchy-surface $\iota_h(\Sigma)\subset M(h)$. 
					 
					 For  any triangulation $\TT$ defining a cell $\CellPSigma$ of $\PSigmaS$, the underlying simplicial complex $\mathcal K_h$  of $Q_h$ is constant on $\PsqTT$. Since the edges of $\A$ are always spacelike, the angles $\theta_e$ are well defined; furthermore one may write down exact an formula for each of the quantities $\theta_e,\kappa_\sigma$ depending on $h$, they are thus of class $\mathscr C^1$ relatively to $h$ and $\mathcal H_{\bar\kappa}$ is $\mathscr C^1$ on each cell $\PsqTT$. 
					 
					 At some $h$ in the intersection of two such cells $\CellPSigma^{1/2}, \P_{\TT'}^{1/2}$ of $\PsqTT$, the functional $\mathcal H_{\bar\kappa}$ can be computed by either triangulation but the edges $e$ changed from $\TT$ to $\TT'$ are $h$-critial ones. An edge is critical if and only if the distance like function on the induced hinge is $\mathscr C^1$ eg if their support plane in $M(h)$ are coplanar, hence if $\theta_e=0$. 
					 The functional is thus well defined and continuous on $\Psq$
					 
					 Consider a triangulation $\TT$ defining a cell $\PsqTT$ of $\Psq$. 
					  On the interior of $\PsqTT$,
					 none of the $h_\sigma$ are zero and none of the edges of the simplicial complex $\mathcal K_h$ are lightlike; Schläffli's thus formula applies.
					 Edges of $Q_h$ are of two types: the one parametrized by  $\A_h$  are spacelike, 
					 and those parametrized by $S$ are timelike. On the interior of a cell $\PsqTT$ Schläffli's formula thus gives:
					 $$\sum_{\sigma\in S} h_\sigma \d \kappa_\sigma + \sum_{e\in \A_h} l_e \d\theta_e = 0.$$
					 hence on the interior of $\PsqTT$:
					 \begin{eqnarray*}
					  \d \mathcal H_{\bar\kappa} &= & \sum_{\sigma \in S} (\kappa_\sigma-\bar\kappa_\sigma) \d h_\sigma  +\sum_{\sigma\in S}  h_\sigma \d \kappa_\sigma + \sum_{e\in \A_h} l_e \d \theta_e\\
					  &=& \sum_{\sigma \in S} (\kappa_\sigma-\bar\kappa_\sigma) \d h_\sigma
					 \end{eqnarray*}
					 The right hand side is well defined and continuous on the whole $\Psq$ even for possibly zero $h$,  $\mathcal H_{\bar\kappa}$ is thus $\mathscr C^1$ on $\Psq$ and
					 $$\d\mathcal H_{\bar\kappa}=\sum_{\sigma \in S}(\kappa_\sigma -\bar\kappa_\sigma)\d h_\sigma.$$
					\end{proof}
		
				\subsection{Convexity of the Einstein-Hilbert functional}

					We now study the hessian of Einstein-Hilbert on the interior of the domain of admissible times $\Psq$. The aim is to prove the following result.

					\begin{prop}\label{prop:H_convexe}
					 For $\bar\kappa\in\RR_+^S$,  the functional $\mathcal{H}_{\bar\kappa}$ is convex on $\Psq$ and strictly convex on
					 the interior of $\Psq$ in $\RR_+^S$.
					\end{prop}

					Consider an adapted triangulation $\TT$ of $(\Sigma,S)$ and consider a cell $\PsqTT$ of  $h\in \PSigmaS^{1/2}$ of non-empty interior.

					For $h\in \PsqTT$, the past of $\Sigma$ in $M(h)$ is
					a locally Minkowski polyhedron with each simplex being a pyramid of $\mass{}$ as represented on  Figure \ref{fig:pyramide} the notations of which we give a more precise meaning. If $T$ is a triangle of $\TT$ of vertices sommets $\sigma_1,\sigma_2,\sigma_3$ while
					$e=\overrightarrow{\sigma_1\sigma_2}$ and $e'=\overrightarrow{\sigma_1\sigma_3}$ are two edges on the boundary of $T$,
					define : $\rho_e$ the real part of the angle from $\overrightarrow{\sigma_1O}$ to  $\overrightarrow{\sigma_1\sigma_2}$,
					$\theta_{ee'}$ the real part of the angle from $\overrightarrow{\sigma_1\sigma_2}$ to $\overrightarrow{\sigma_1\sigma_3}$
					and $\alpha_e$ the real part of the dihedral angle from the plane $(O\sigma_1\sigma_1)$ to the plane $(\sigma_1\sigma_2\sigma_3)$.
					In this section, edges are oriented so that we distinguish $\alpha_e$ and $\alpha_{-e}$ :
					the angle $\alpha_e$ is on the left of $e$, thus $\alpha_{-e}$ is the angle on the right of $e$.

					\begin{figure}[h]
					\caption{Simplex cell of the past of $\Sigma$ in $M(h)$.
					}
					The following angles are represented: $\rho_e$ the angle from $\overrightarrow{\sigma_1O}$ to  $\overrightarrow{\sigma_1\sigma_2}$,
					$\theta_{ee'}$ the angle from $\overrightarrow{\sigma_1\sigma_2}$ to $\overrightarrow{\sigma_1\sigma_3}$
					and $\alpha_e$ the angle from the plane $(0\sigma_1\sigma_1)$ to the plane $(\sigma_1\sigma_2\sigma_3)$.

					\begin{center}\label{fig:pyramide}
					\includegraphics[width=0.48\linewidth]{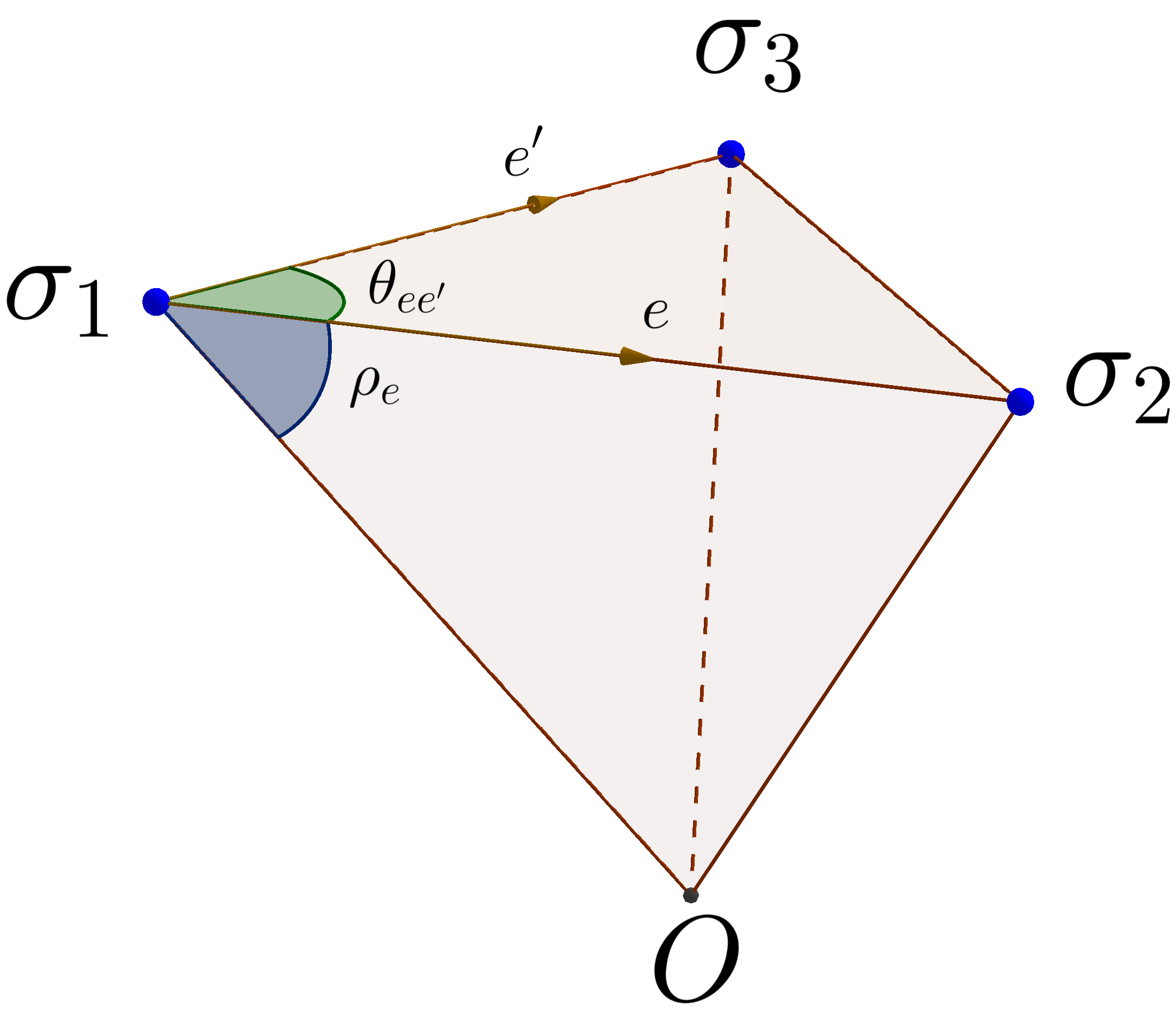} \includegraphics[width=0.48\linewidth]{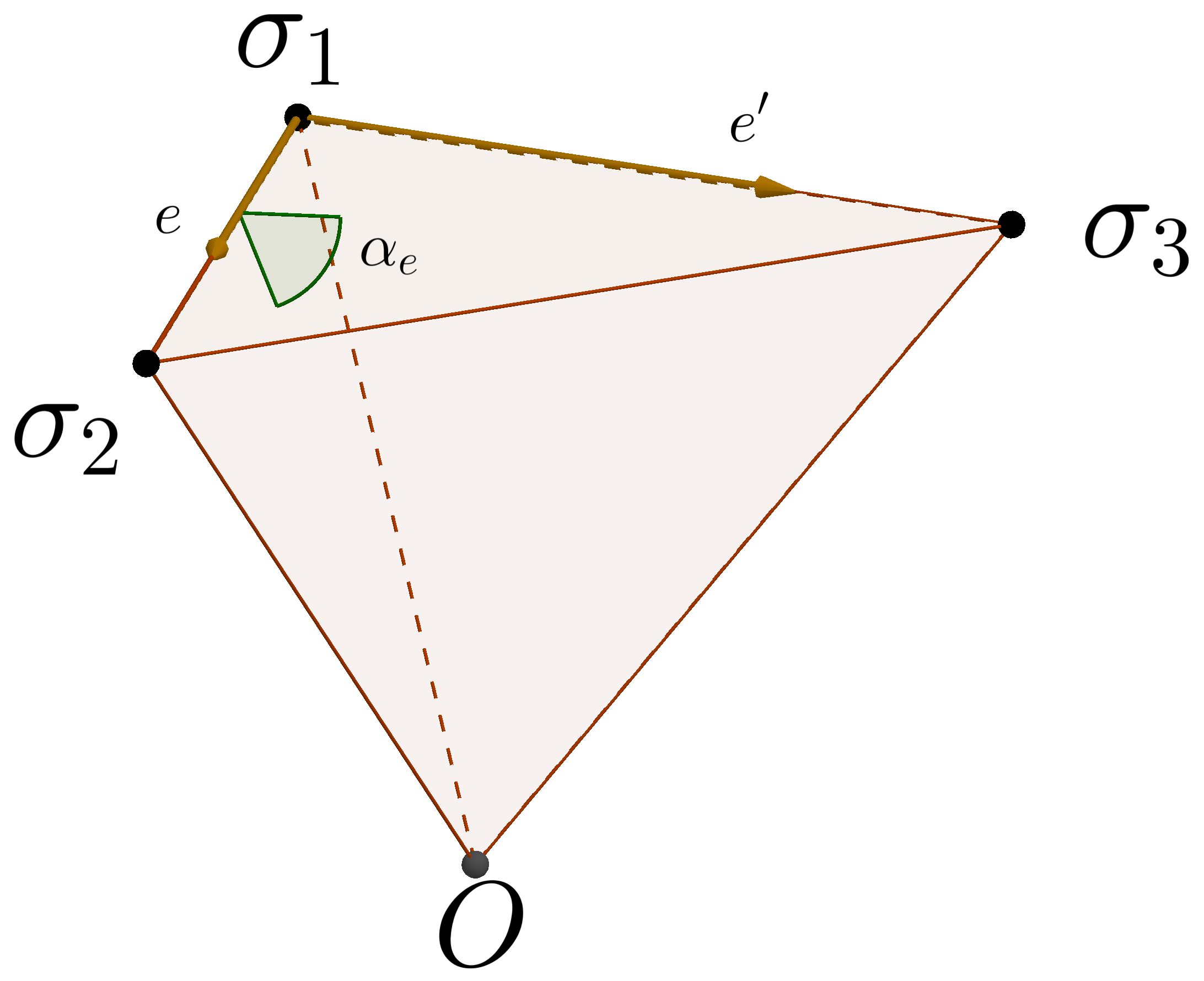}
					\end{center}
					\end{figure}

					We aim at computing the partial derivatives $$\frac{\partial \kappa_{\sigma_1}}{\partial {h_{\sigma_2}}},\quad \quad \sigma_1,\sigma_2\in S.$$
					If there is no edge from $\sigma_1$ to $\sigma_2$, then this derivative is null.
					Il there is an edge $e$ from $\sigma_1$ to $\sigma_2$, then in both pyramids $P_+$ and $P_-$ on both sides of $e$, we need to study the variations of the dihedral  on the edge $[O\sigma_1]$ with respect to $h_{\sigma_2}$ and $h_{\sigma_1}$.
					Consider $P_+$ and use the notations of Figure \ref{fig:pyramide}, the idea is to consider the tetragon  of $\H^2$ with 2 sides given by the geodesics
					corresponding to the planes $(O\sigma_1\sigma_2)$, $(O\sigma_1\sigma_3)$ and 2 other sides given by the planes given by the planes normal to $(\sigma_1\sigma_2)$ and $(\sigma_1\sigma_3)$ through $O$; see Figure \ref{fig:cerf_volant}.
					Notice that this tetragon has two right angles and that the lengths are the angles $\alpha_e,\alpha_{-e'},\rho_{e}$ and $\rho_{e'}$. These lengths can be positive or negative, the resulting tetragon may then have autointersections. The two non right angles are $\kappa_{ee'}$ or $\pi-\kappa_{ee'}$ and  $\pi-\theta_{ee'}$ or $\theta_{ee'}$ depending on the signs of  $\alpha_e,\alpha_{-e'},\rho_{e}$ and $\rho_{e'}$.
					In such a hyperbolic tetragon, which we call a {\it kite} we have the following relations.
					
					\begin{prop}[\cite{Fenchel_hyp_kite}]\label{lem:loisinus}~\\~	\begin{minipage}[t]{7cm}
					\begin{minipage}[t]{7cm}
					 A kite such as on the adjacent figure is fully determined
					 up to isometry by 3 of the 6 parameters  $\alpha_1,\alpha_2,\rho_1,\rho_2,\kappa,\Theta$ which can be either positive or negative. Furthermore,
					  \begin{eqnarray*}
					   \cos(\kappa) &=& \frac{\sinh(\rho_1) \sinh(\rho_2)-\cos(\Theta)}{\cosh(\rho_1)\cosh(\rho_2)}\\
					   \sinh(\rho_2) &=& \frac{\cos(\kappa)\sinh(\alpha_1)+\sinh(\alpha_2)}{\sin(\kappa)\cosh(\alpha_1)}\\
					   \frac{\sin(\kappa)}{\sin(\Theta)}&=& \frac{\cosh(\alpha_2)}{\cosh(\rho_2)} = \frac{\cosh(\alpha_1)}{\cosh(\rho_1)}.
					  \end{eqnarray*}
					\end{minipage}
					\end{minipage}
					\begin{minipage}[t]{4cm}
					\ \\

					\includegraphics[width=\linewidth]{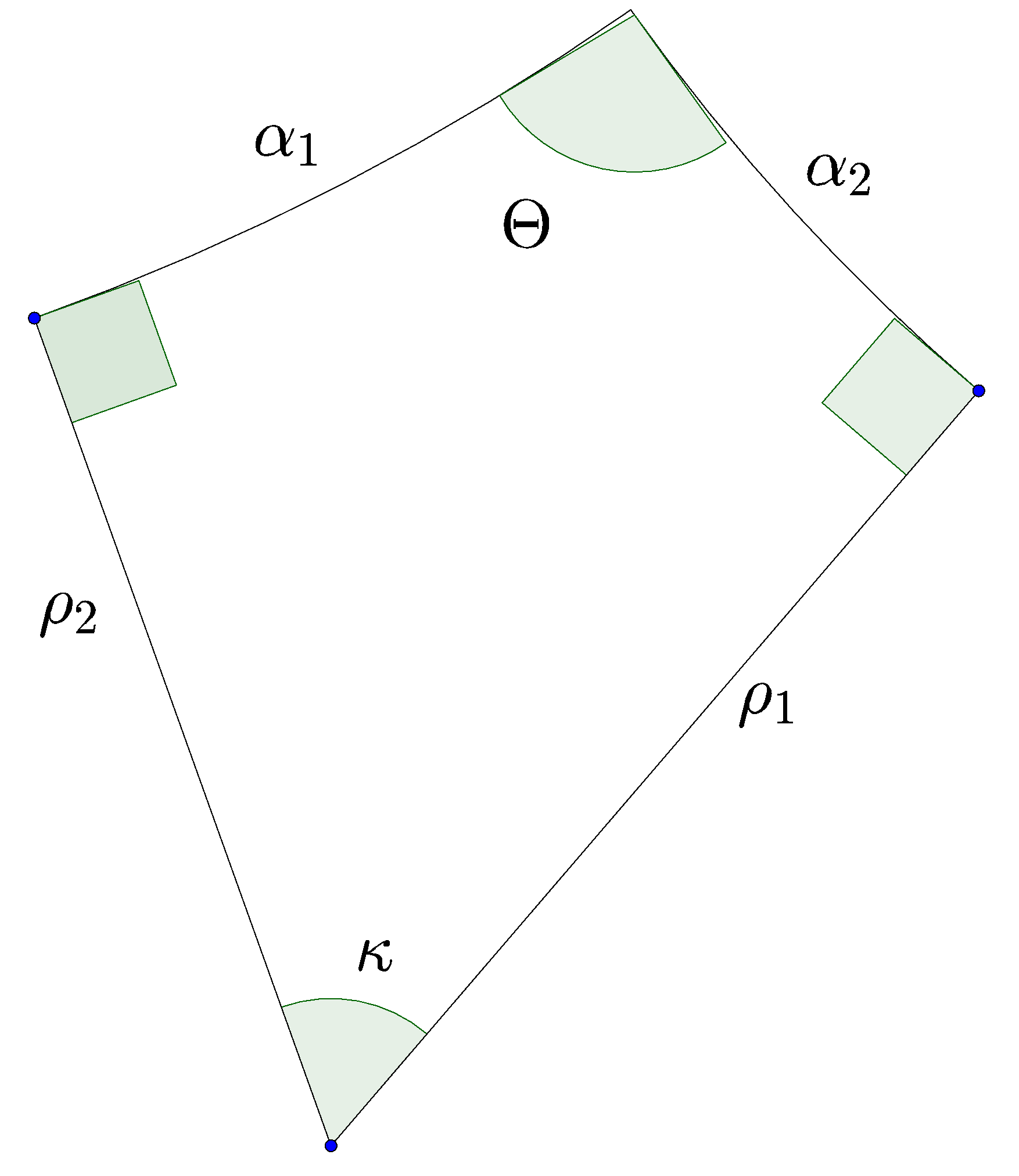}
					\end{minipage}
					\end{prop}

					\begin{figure}[h] 
						\caption{Kite associated to an edge\label{fig:cerf_volant}}
						With $e$ the edge $\overrightarrow{\sigma_1\sigma_2}$ and $e'$ the edge $\overrightarrow{\sigma_1\sigma_3}$.
						\begin{center}
							\includegraphics[width=0.6\linewidth]{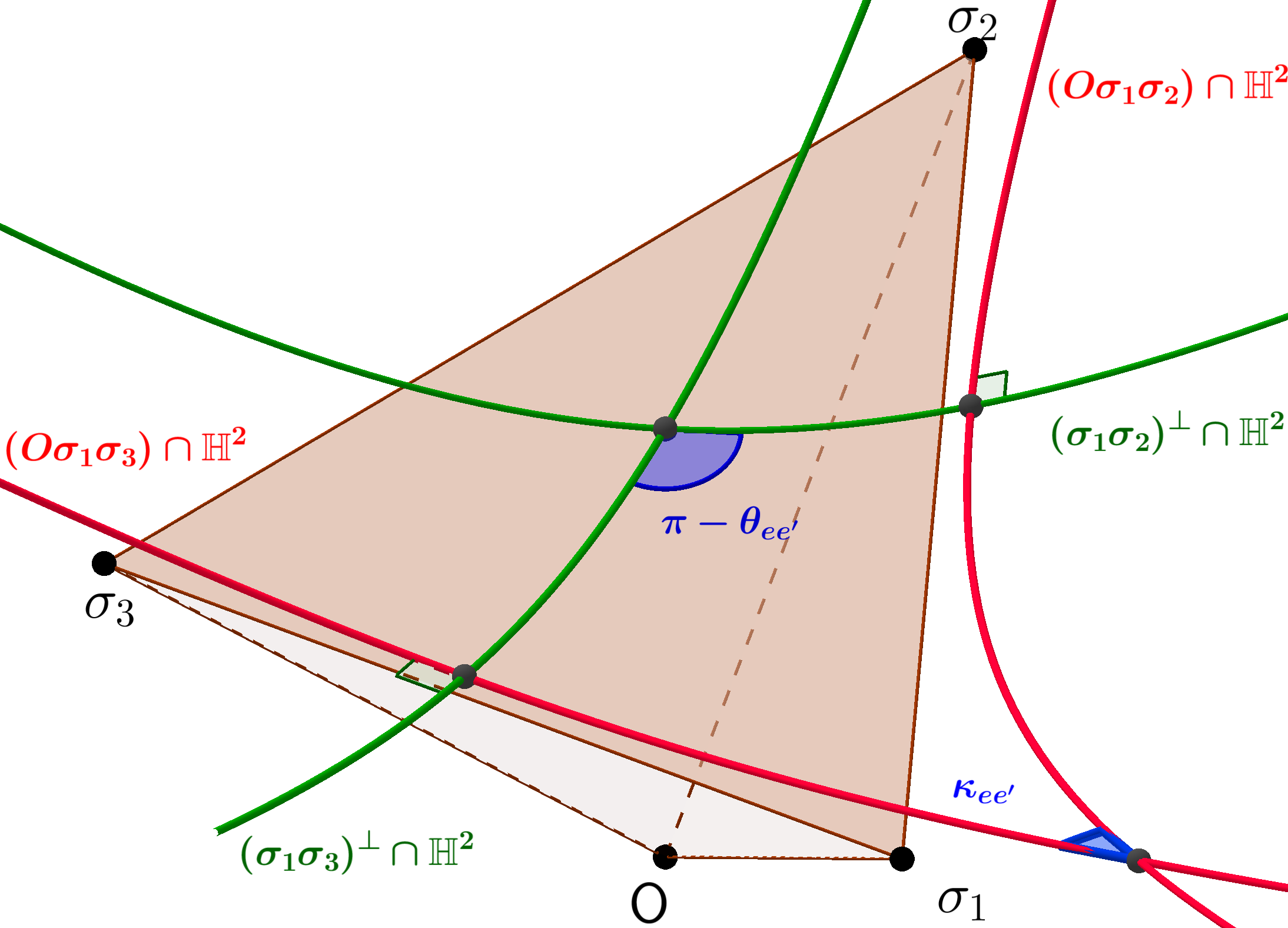}
							\includegraphics[width=0.35\linewidth]{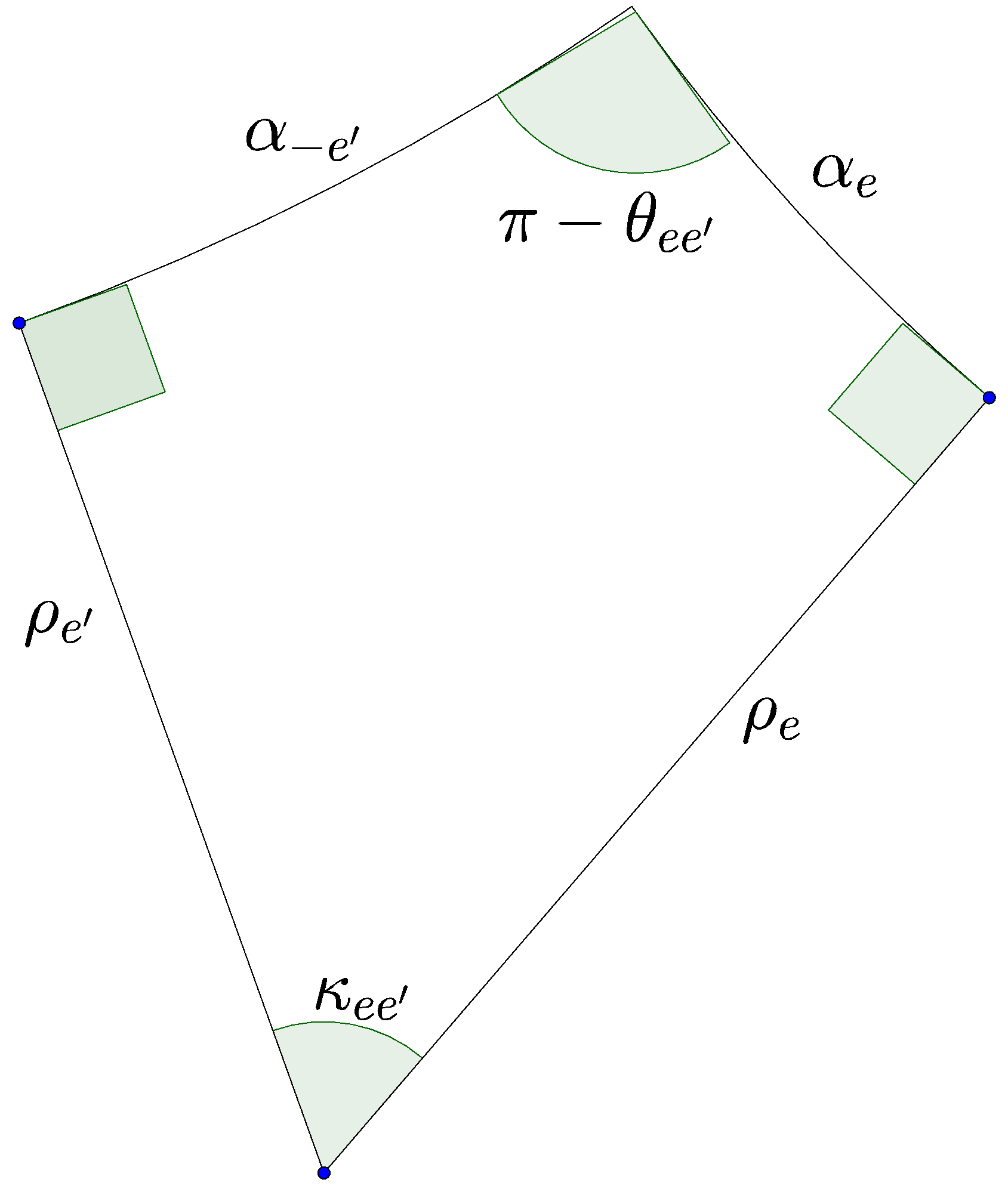}
						\end{center}

					\end{figure}
					\begin{cor} \label{lem:drapeau}
						Using the same notations as in Proposition \ref{lem:loisinus} and choosing
						$\Theta,\rho_1,\rho_2$ as paramaters we have:
							$$\frac{\partial \kappa}{\partial \rho_1} = -\frac{\tanh(\alpha_2)}{\cosh(\rho_1)}.$$
					\end{cor}
					
					We thus need to compute the  derivative of $\rho_e$ with respect to the heights $(h_\sigma)_{\sigma\in S}$ for each edge $e$.
					\begin{lem}
					Using the notations of Figure \ref{fig:pyramide} then
					$$\d \rho_{e} =  - \frac{(h_{\sigma_1}^2+h_{\sigma_2}^2+l_e^2)\d h_{\sigma_1}  - 2 h_{\sigma_1} h_{\sigma_2} \d h_{\sigma_2} }{2 l_e h_{\sigma_1}^2\cosh(\rho_{e})}$$

					\end{lem}
					\begin{proof}
					From the cosine law in $\mass{}$:
					\begin{eqnarray*} -h_{\sigma_2}^2   &=&-h_{\sigma_1}^2+l_e^2 - 2 l_e h_{\sigma_1}  \sinh(\rho_e) \\
					\cosh(\rho_e)\d \rho_e &=& \frac{h_{\sigma_1}(-2h_{\sigma_1})-(h_{\sigma_2}^2+l_e^2-h_{\sigma_1}^2)}{2l_eh_{\sigma_1}^2} \d h_{\sigma_1} + \frac{2 h_{\sigma_2} h_{\sigma_1}}{2 l_e h_{\sigma_1}^2} \d h_{\sigma_2} \\
					\d \rho_e&=&- \frac{(h_{\sigma_2}^2+l_e^2 + h_{\sigma_1}^2)\d h_{\sigma_1}  - 2 h_{\sigma_1} h_{\sigma_2} \d h_{\sigma_2} }{2 l_e h_{\sigma_1}^2\cosh(\rho_e)}
					\end{eqnarray*}
					\end{proof}

					\begin{prop}\label{prop:diff_kappa}
					The map $\kappa$ is  $\mathscr C^1$ on $\PSigmaS^{1/2}$ and for all $h\in \PSigmaS^{1/2}$ and all $\sigma\in S$, we have
					$$\d_h \kappa_\sigma=\sum_{e\in \mathcal E_h, e:\sigma \rightsquigarrow \sigma'} \frac{\tanh(\alpha_e)+\tanh(\alpha_{-e})}
					{\cosh^2(\rho_{e})} \frac{(h_{\sigma'}^2+l_e^2 + h_{\sigma}^2)\d h_{\sigma}  - 2 h_{\sigma} h_{\sigma'} \d h_{\sigma'} }{2 l_e h_{\sigma}^2}$$
					where $\mathcal E_h$ is the set of edges of any $h$-Delaunay triangulation.
					\end{prop}
					\begin{proof}
						On the interface of two cells $\P^{1/2}_{\TT}$ and $\P^{1/2}_{\TT'}$, the edges that change are those which are $\tau$-critical hence $\alpha_e=-\alpha_{-e}$ and the corresponding term in the right hand side of the formula is zero. Hence, the right hand side is continuous on $\PSigmaS^{1/2}$. 
						
						On a given cell $\PsqTT$, denote by $\mathcal E$ the set of edges of $\TT$.
						For $\sigma\in S$ and $h\in \PsqTT$, denote by $(e_i)_{i\in \ZZ/n\ZZ}$ the family of outgoing edges from $\sigma$ enumerated coherently with the orientation of $\Sigma$. Define $\sigma_i\in S$ the other end of 
						$e_i$ so that:

					 \begin{eqnarray*}
					  \d_h \kappa_\sigma &=& \sum_{i\in \ZZ/n\ZZ} \d_h \kappa_{e_i e_{i+1}}\\
					  &=& \sum_{i\in \ZZ/n\ZZ} \left(-\frac{\tanh(\alpha_{e_i})}{\cosh(\rho_{e_i})} \d \rho_{e_i}-
					  \frac{\tanh(\alpha_{-e_{i+1}})} {\cosh(\rho_{e_{i+1})}}\d \rho_{e_{i+1}}\right)\\
					  &=&-\sum_{i\in \ZZ/n\ZZ} \left(\frac{\tanh(\alpha_{e_i})}{\cosh(\rho_{e_i})}+
					  \frac{\tanh(\alpha_{-e_{i}})} {\cosh(\rho_{e_{i})}}\right)\d \rho_{e_{i}}\\
					  &=& -\sum_{i\in \ZZ/n\ZZ}\frac{\tanh(\alpha_{e_i})+
					  \tanh(\alpha_{-e_{i}})} {\cosh(\rho_{e_{i})}}\d \rho_{e_{i}}\\
					  &=&\sum_{i\in \ZZ/n\ZZ}\frac{\tanh(\alpha_{e_i})+
					  \tanh(\alpha_{-e_{i}})} {\cosh(\rho_{e_{i})}} \frac{(h_{\sigma}^2+h_{\sigma_i}^2+l_e^2)\d h_{\sigma}  - 2 h_{\sigma} h_{\sigma_i} \d h_{\sigma_i} }{2 l_e h_{\sigma}^2\cosh(\rho_{e})}\\
					  &=&\sum_{i\in \ZZ/n\ZZ}\frac{\tanh(\alpha_{e_i})+
					  \tanh(\alpha_{-e_{i}})} {\cosh^2(\rho_{e_{i})}} \frac{(h_{\sigma}^2+h_{\sigma_i}^2+l_e^2)\d h_{\sigma}  - 2 h_{\sigma} h_{\sigma_i} \d h_{\sigma_i} }{2 l_e h_{\sigma}^2}\\
					 \end{eqnarray*}
					\end{proof}

					\begin{proof}[Proof of Proposition \ref{prop:H_convexe}]
					  From Propositions \ref{prop:diff_H} and \ref{prop:diff_kappa}, $\mathcal H_{\bar\kappa}$ is  $\mathscr C^2$ and 
					  its Hessian matrix $H$ has the following coefficients:
						\begin{center}
					  \begin{tabular}{rcll}
					   $\displaystyle H_{\sigma,\tau}$ &$\displaystyle=$ &$\displaystyle-\sum_{e:\sigma\rightsquigarrow \tau}
					   \frac{\tanh(\alpha_e)+\tanh(\alpha_{-e})}
					{\cosh^2(\rho_{e})} \frac{ 2 h_{\sigma} h_{\tau}  }{2 l_e h_{\sigma}^2} $ & $\displaystyle\forall \sigma,\tau\in S, \sigma\neq \tau$\\
					$\displaystyle H_{\sigma,\sigma}$&$\displaystyle=$&$\displaystyle\sum_{\tau\in S}\sum_{e:\sigma\rightsquigarrow \tau}
					\frac{\tanh(\alpha_e)+\tanh(\alpha_{-e})}
					{\cosh^2(\rho_{e})} \frac{ h_{\sigma}^2 +h_{\tau}^2+l_e^2  }{2 l_e h_{\sigma}^2}  $&$\displaystyle\forall \sigma\in S$\\
					&&$\displaystyle\quad\quad-\sum_{e:\sigma\rightsquigarrow \sigma}
					\frac{\tanh(\alpha_e)+\tanh(\alpha_{-e})}
					{\cosh^2(\rho_{e})} \frac{ 2 h_{\sigma} h_{\sigma}  }{2 l_e h_{\sigma}^2} $&
					\end{tabular}
					 \end{center}
					 Since the embedding of $\Sigma$ into $M(h)$ is convex, $\tanh(\alpha_e)+\tanh(\alpha_{-e})\geq 0$ with equality if and only if the edge is $h$-critical.
					 Therefore, for all $\sigma\in S$ :
					 \begin{eqnarray*}
					 H_{\sigma,\sigma} - \sum_{\tau\neq \sigma} |H_{\sigma,\tau}|&\geq& \sum_{\tau\in S} \sum_{e:\sigma\rightsquigarrow \tau}
					 \frac{\tanh(\alpha_e)+\tanh(\alpha_{-e})}
					{\cosh^2(\rho_{e})} \frac{h_{\tau}^2+l_e^2 + h_{\sigma}^2 - 2 h_{\sigma} h_{\tau}  }{2 l_e h_{\sigma}^2}\\
					&=&\sum_{\tau\in S} \sum_{e:\sigma\rightsquigarrow \tau}
					 \frac{\tanh(\alpha_e)+\tanh(\alpha_{-e})}
					{\cosh^2(\rho_{e})} \frac{(h_{\tau}-h_{\sigma})^2 +l_e^2  }{2 l_e h_{\sigma}^2}\\
					&\geq&0.
					 \end{eqnarray*}
					The Hessian matrix of $\mathcal H_{\bar\kappa}$ is thus diagonally dominant.

					Consider some $h\in \Psq$ and $\sigma\in S$ such that $ H_{\sigma,\sigma} - \sum_{\tau\neq \sigma} |H_{\sigma,\tau}| = 0$. Then all outgoing edges from $\sigma$ are $h$-critical. One can construct an immersed unflippable hinge $(Q,\eta)$ of $\Sigma$  with vertices in $S$ such that $Q$ is unflippable, the vertices contains at least two points of $S$ and inscribed into the neighborhood of $\sigma$ given by the union of the triangles of $\TT_h$ containing $\sigma$.  Such a hinge is $h$-critical and unflippable hence $h$ is in the boundary of $\Psq$. Finally, the Hessian matrix $H$  is strictly diagonally dominant on the interior of $\Psq$.
					\end{proof}

		\subsection{Proof of the main Theorem}
			\label{sec:conclusion_alexandrov}

			We now prove the main Theorem, 
			\begin{theo} \label{theo:main}
				Let $\Sigma$ be a closed locally Euclidean surface of genus $g$ with $s$ marked conical singularities of angles $(\theta_i)_{i\in \lsem 1,s\rsem}$.
				For all $\bar\kappa\in \prod_{i=1}^s[0,2\pi]\cap [0,\theta_i[$, there exists a radiant singular flat spacetime   $M$ homeomorphic to $\Sigma\times \RR$ with exactly $s$ marked lines $\Delta_1,\cdots,\Delta_s$ of respective masses $\bar\kappa_1,\cdots, \bar\kappa_s$ and
				a convex polyhedral embedding $\iota : (\Sigma,S) \rightarrow (M,(\Delta_i)_{i\in \lsem 1,s\rsem})$ 
				
				Furthermore, such a couple $(M,\iota)$ is unique up to equivalence.
			\end{theo}
			
			Denoting by $\kappa(x)$ the cone angle at $x$ if $x$ is a point in a $\H_{\geq0}$-manifold, in view of Theorem \ref{theo:radiant_struc}  the Theorem can  also be stated as follows.
			\begin{cor} 
			 	Let $\Sigma$ be a closed locally Euclidean surface of genus $g$ with $s$ marked conical singularities of angles $(\theta_\sigma)_{\sigma\in S}$.
				For all $\bar\kappa\in \prod_{\sigma \in S}[0,2\pi]\cap [0,\theta_\sigma[$, there exists a closed $\H_{\geq0}$-manifold $\Sigma_{\bar \kappa}$ together with an homeomorphism  $h:\Sigma\rightarrow \Sigma_{\bar\kappa}$  and a polyhedral embedding $\iota: (\Sigma,S)\rightarrow \susp(\Sigma_{\bar\kappa})$ such that :
				\begin{itemize}
				 \item for all $\sigma\in S$, $\bar\kappa_\sigma = \kappa(h(\sigma))$
				 \item with $\susp(\Sigma_{\bar\kappa})\xrightarrow{~~\pi~~} \Sigma_{\bar\kappa}$ the natural projection, we have $\pi\circ \iota = h$
				\end{itemize}
				Furthermore, such a triple $(\Sigma_{\bar\kappa},h,\iota)$ is unique up to equivalence.
			\end{cor}
			\begin{rem} Equivalence  between triple $(\Sigma_{\bar\kappa}^{(i)},h^{(i)},\iota^{(i)})$ for $i\in\{1,2\}$ is understood as an isomorphism $\varphi : \Sigma_{\bar\kappa}^{(1)} \rightarrow \Sigma_{\bar\kappa}^{(2)}$ such that $\iota^{(2)} = \hat\varphi \circ \iota^{(1)} $ with $\hat\varphi : \susp(\Sigma_{\bar\kappa}^{(1)}) \xrightarrow{\sim} \sup(\Sigma_{\bar\kappa}^{(2)})$ the isomorphism induced by $\varphi$
			 
			\end{rem}

			Before diving into the proof, we prove a last Lemma.
			\begin{lem}\label{lem:limite_infty}With $\theta=(\theta_\sigma)_{\sigma\in S}$ the cone angles of $\Sigma$, we have
				  $$\lim_{\tau\in \PSigmaS,\tau\rightarrow+\infty} \kappa(\tau)=\theta.$$
				\end{lem}
				\begin{proof}We use the same notations as in the preceding section.
				 In a given cell $\CellPSigma$ of $\PP_\Sigma$, for each vertex $\sigma \in S$  and for all edge $e$ of $\TT$ outgoing from $\sigma$ by cosine law 
				 $$ -\tau_{\sigma_2}   =-\tau_{\sigma_1}+l_e^2 - 2 l_e \sqrt{\tau_{\sigma_1}}  \sinh(\rho_e).$$
				 Since $|\tau_{\sigma_1}-\tau_{\sigma_2}|$ is uniformly bounded on $\PSigmaS$ and $l_e$ is constant, we have $\rho_e \xrightarrow{\tau\rightarrow+\infty} 0$. Then from Proposition
				 \ref{lem:loisinus}, with $e'$ the subsequent edge around $\sigma$, we have $\kappa_{ee'} \xrightarrow{\tau\rightarrow +\infty} \theta_{ee'}$. Hence, $$\kappa(\tau)_\sigma\xrightarrow{\tau\in \CellPSigma,\tau\rightarrow +\infty} \theta_\sigma.$$
				 Finally, there are only finitely many cells $\CellPSigma$ and $S$ is finite, the result follows.
				\end{proof}

			\begin{proof}[Proof of Theorem \ref{theo:main} ] Let $Z\subset S$, denote by $z:=|Z|$ and $s:=|S|$,  define $\Pi_Z := \{\tau \in \RR^S ~|~\forall \sigma\in Z, \tau_\sigma=0\}$ and recall  that $\kappa_\sigma(\tau)=0$ if and only if $\tau_\sigma=0$. We prove the Theorem for $\bar\kappa$ such that 
				$\{\sigma\in S~|~\bar\kappa_\sigma=0\}=Z$. 
				It suffices to show that for such $\bar\kappa$ the Einstein-Hilbert functional $\mathcal H_{\bar\kappa}$ has exactly one critical point in $\PSigmaS\cap \Pi_Z$.
				Define $K_Z:=\{\bar\kappa \in \prod_{\sigma\in S} [0,2\pi]\cap[0,\theta_\sigma[ ~|~ \forall \sigma\in Z,\bar \kappa_\sigma=0\}$.  If $z=s$ then $K_Z=\{0\}$ and $\PSigmaS\cap \Pi_Z = \{0\}$ by Theorem \ref{theo:domain_description}.(c), there is nothing else to prove. Otherwise, we proceed as follows.

				By Proposition \ref{prop:H_convexe} $\mathcal H_{\bar\kappa}$ is strictly convex in the interior of  $\PSigmaS^{1/2}$, defined on the interior of $\PSigmaS$ the functional $\mathcal H_{\bar\kappa}$ only has critical points of index 1. Hence, the restriction of $\mathcal H_{\bar\kappa}$ to $\PSigmaS\cap \Pi_Z$  only has index 1 critical points in the relative interior of $\PSigmaS\cap \Pi_Z$.
				
				Let $\tau \in \partial \PSigmaS\cap\Pi_Z$. By Theorem \ref{theo:domain_description}.(e), on the boundary of $\PSigmaS$, there exists $\sigma\in S\setminus Z$ such that either $\tau_\sigma =0$ or $\tau$ is in the kernel of the affine form of an unflippable immersed hinge. In the former situation, we have $0=\kappa_\sigma<\bar\kappa_\sigma$. In the latter situation, 
				consider such a hinge $(Q,\eta)$ with $Q=([ABCD],[AC])$.
				\begin{itemize}
				 \item  If $Q$ is embedded then $Q$ is unflippable and without loss of generality we may  assume $C\in [ABD]$, the cone around $\sigma=\eta(C)$ is then convex and contains a coplanar wedge of Euclidean angle at least $\pi$. By Lorentzian Volkov's Lemma  (Theorem \ref{theo:volkov_lorentz}), if $\theta_\sigma>2\pi$ we have $\kappa_\sigma>2\pi\geq \bar\kappa_\sigma$ and if $\pi\leq\theta_\sigma\leq 2\pi$ we have $\kappa_\sigma\geq \theta_\sigma>\bar\kappa_\sigma$. Either way, $\kappa_\sigma>\bar\kappa_\sigma$.
				 \item If $\eta$ is not an embedding,  then without loss of generality we may assume $\eta(A)=\eta(B)=\eta(D)$ and  the cone around $\sigma:=\eta(C)$ is coplanar. Then 
				 $\theta_\sigma<\pi$ and by Lorentzian Volkov's Lemma, $\kappa_\sigma=\theta_\sigma>\bar\kappa_\sigma$
				\end{itemize}
				Together with Proposition \ref{prop:diff_H} this implies that $\mathcal H_{\bar\kappa}$ has no critical points on the boundary $\partial\PSigmaS\cap \Pi_Z$. 
				
				If $z=s-1$, then $\kappa$ is a  function defined on an interval, continuous and increasing from $0$ to some $\kappa_{\max}>\bar\kappa$. The result follows.

				We now assume $z\leq s-2$. Define $\P_Z := \PSigmaS \cap \Pi_Z$ if $Z\neq\emptyset$ and   $\P_Z : = \PSigmaS \cup \{\infty\}$ if $Z=\emptyset$. This way $\P_Z$ is homeomorphic to a $s-z$ dimensionnal closed ball and its boundary $\partial \P_Z$ is homeomorphic to a $s-z-1$-dimensionnal sphere. The homeomorphism may be explicited by the radial map from some $\tau_0\in Int(\P_Z)$.
				Consider the family of vector fields $$\fonctiondeux{X}{K_Z\times \P_Z}{\SS^{s-1-z}}{\bar\kappa,\tau\neq\infty}{(\kappa_\sigma(\tau)-\bar\kappa_\sigma)_{\sigma\in S\setminus S}}{\bar\kappa,\infty}{(\theta_\sigma-\bar\kappa_\sigma)_{\sigma\in S\setminus Z}} $$
				and notice that  $X(\bar\kappa,\cdot)_{|Int(\P_Z)}$ is the gradient of $\mathcal H_{\bar\kappa|Int(\P_Z)}$ for $\bar\kappa\in K_Z$ by Proposition \ref{prop:diff_H}.
				 By Lemma \ref{lem:limite_infty}, $X$ is continuous at $\infty$ if $Z=\emptyset$; thus $X$ is continuous on $K_Z\times\partial\P_Z$ and, from the discussion above, non singular on the boundary of $\PSigmaS\cap \Pi_Z$. The number of singular points of the vector field  $X(\bar\kappa,\cdot)$ in the interior of $\P_Z$ is equal to the index of $\frac{X(\bar\kappa,\cdot)}{\|X(\bar\kappa,\cdot)\|}$ on $\partial \P_Z$. Since $\bar\kappa \mapsto X(\bar\kappa,\cdot)$ is continuous and $K_Z$ is connected, the index of $\frac{X(\bar\kappa,\cdot)}{\|X(\bar\kappa,\cdot)\|}$ is independant from $\bar\kappa$.
				 Finally, take some $\bar\tau$ in the interior of $\P_Z$ close enough to $0$ so that $\prod_{\sigma\in S\setminus Z} [0,2\kappa_\sigma] \subset \P_Z$ and consider the vector field  $Y:\tau \rightarrow \frac{\tau-\bar\tau}{\|\tau-\bar\tau\|}$ which can be continously extended to the whole $\P_Z$. On the one hand, for $\tau$ on a  $"\tau_\sigma=0"$ boundary component, $Y(\tau)_\sigma<0$ while $\kappa(\tau)_\sigma=0$; on the other hand, for $\tau$ on a $"Q^*(\tau)=0"$ boundary component, there is a $\sigma \in S\setminus Z$ such that $\kappa_\sigma-\bar\kappa_\sigma>0$ and on such a component, $\forall \sigma'\in S\setminus Z, (\tau-\bar\tau)_{\sigma'}>0$. In any case, $\forall \tau\in \partial \P_Z, Y(\tau)\neq - \frac{X(\bar\kappa,\tau)}{\|X(\bar\kappa,\tau)\|}$ thus  $\frac{X(\bar\kappa,\cdot)}{\|X(\bar\kappa,\cdot)\|}$ is homotopic to $Y$. The latter has index 1, thus so has the former and 
				 for all $\bar\kappa\in K_Z$, $\mathcal H_{\bar\kappa}$ has exactly one critical point on $\P_Z$.
				
			\end{proof}
\bibliographystyle{halpha2}
\bibliography{note.bib}
\end{document}